\documentclass[12pt]{amsart}
\usepackage{amsfonts,amssymb,amscd,xy,bm}
\usepackage[leqno]{amsmath}
\usepackage{mathrsfs}
\usepackage{amssymb}
\usepackage{amsmath}
\usepackage{amsxtra}
\usepackage{amsthm}
\usepackage{hyperref}

\newcommand{\Hom}{{\mathrm{Hom}}}

\DeclareMathAlphabet{\mathbbmsl}{U}{bbm}{m}{sl}

\def\picxys{{\rm Pic}_{X\be\times_{\be S}\lbe Y\be/\be S}}
\def\br{{\rm{Br}}\e}

\def\brp{{\rm{Br}}^{\e\prime}\lbe\lbe}
\def\picgs{\pic_{\be G\lbe/\lbe S}}	
	
\newcommand{\xs}{X^{\rm{s}}}

\newcommand{\abe}{{\rm{\bf Ab}}}

\def\brxs{{\rm{Br}}^{\e\prime}_{\be X\be/\lbe S}}

\def\upicxs{\be{\rm{UPic}}_{\lbe X\be/\lbe S}}
\def\upicgs{\be{\rm{UPic}}_{\e G\lbe/\lbe S}}
\def\upicys{{\rm{UPic}}_{\e Y\be/\lbe S}}

\def\upicxys{{\rm{UPic}}_{\lbe X\lbe\times_{\be S}\lbe Y\be/\lbe S}}

\def\ses{S_{\et}^{\e\sim}}

\def\bra{{\rm{Br}}_{\lbe\rm{a}}^{\e\prime}\lbe(\lbe X\be/\lbe S\e)}
\def\bras{{\rm{Br}}_{\lbe{\rm a},\e S}^{\e\prime}}
\def\bray{{\rm{Br}}_{\lbe\rm{a}}^{\e\prime}\lbe(\le Y\!/\be S\e)}

\def\braxy{{\rm{Br}}_{\lbe\rm{a}}^{\e\prime}\lbe(X\!\times_{\be S}\! Y\!/\lbe S\e)}

\def\npicxs{{\rm N}\lle\pic\be(\lbe X\!/\lbe S\e)}
\def\npiczs{{\rm N}\lle\pic\be(\lbe Z\be/\lbe S\e)}

\def\npic{{\rm N}\lle\pic\lle}

\def\brys{{\rm Br}^{\e\prime}_{Y\!/\lbe S}}

\newcommand{\nat}{\le\natural}

\def\uxs{{\rm U}_{\be X\be/\lbe S}\le}

\def\ruxs{{\rm RU}_{\be X\be/S}}
\def\ruys{{\rm RU}_{Y\be/S}}

\def\bg{{\mathbb G}}
\def\g{\varGamma}

\def\npic{{\rm N}\lle\pic\lbe }

\def\der{\e\rm{der}}
\def\rad{\e{\rm{rad}}}

\def\pic{{\rm{Pic}}\,}
\def\upic{{\rm{UPic}}}

\def\hom{{\rm{Hom}}\e}

\def\tor{\e\rm{tor}}
\def\cbs{C^{\e b}\lbe(S_{\et})}
\def\dbs{D^{\e b}\lbe(S_{\et})}

\def\Gtil{{\widetilde{G}}}

\def\Ttil{{\widetilde{T}}}

\newcommand{\isoto}{\overset{\!\sim}{\to}}

\def\g{\varGamma}

\def\uys{{\rm U}_{\le Y/S}\le}

\def\ugs{{\rm U}_{\lbe G/S}\le}

\def\picxs{{\rm Pic}_{X\be /S}}
\def\picys{{\rm Pic}_{\e Y\be /S}}

\def\nps{{\rm N\le Pic}_{\le S}}

\def\bro{{\rm{Br}}_{\lbe 1}^{\le\prime}\lbe\lle(\lbe X\!/\be S\e)}
\def\broys{{\rm{Br}}_{\lbe 1}^{\le\prime}(\le Y\!/\be S\e)}
\def\broxys{{\rm{Br}}_{\lbe 1}^{\le\prime}\lbe(\lbe X\!\times_{\be S}\!Y\!/\be S\e)}

\usepackage[usernames,dvipsnames]{xcolor}
\usepackage{tabularx}
\usepackage{epsfig,amssymb}
\usepackage{bm} %%% to write bold math symbols
\usepackage{verbatim} %%to use \begin{comment}

\usepackage{hyperref}
\definecolor{labelkey}{rgb}{1,0,0}

\usepackage{verbatim} %%for verbatim

\makeatletter
\def\bm#1{\mathpalette\bmstyle{#1}}
\def\bmstyle#1#2{\mbox{\boldmath$#1#2$}}

\DeclareMathAlphabet{\mathcalligra}{T1}{calligra}{m}{n}

\numberwithin{equation}{section}

\newcommand{\sh}{\kern -.4em\phantom{a}^{\mathbf{\sim}}}

\newcommand{\lra}{\longrightarrow}

\newcommand{\et}{{\rm {\acute et}}}

\def\be{\kern -.1em}
\def\le{\kern 0.03em}
\def\lle{\kern 0.015em}
\def\lbe{\kern -.025em}

\newcommand{\Z}{{\mathbb Z}}

\newcommand{\N}{{\mathbb N}}

\newcommand{\spec}{\mathrm{ Spec}\,}

\newcommand{\krn}{\mathrm{Ker}\e}
\newcommand{\img}{\mathrm{Im}\e}
\newcommand{\cok}{\mathrm{Coker}\,}

\def\e{\kern 0.08em}

\newcommand{\s}{\mathscr }

\newcommand{\fl}{{\rm fl}}

\newcommand{\ks}{k^{\e\rm{s}}}

\newtheorem{lemma}{Lemma}[section]
\newtheorem{theorem}[lemma]{Theorem}
\newtheorem{proposition-definition}[lemma]{Proposition-Definition}
\newtheorem{corollary}[lemma]{Corollary}
\newtheorem{proposition}[lemma]{Proposition}
\theoremstyle{definition}
\newtheorem{definition}[lemma]{Definition}

\theoremstyle{remark}
\newtheorem{remark}[lemma]{Remark}
\newtheorem{remarks}[lemma]{Remarks}
\newtheorem{example}[lemma]{Example}

\definecolor{labelkey}{rgb}{1,0,0}

\begin{document}

\input xy     %%for diagrams
\xyoption{all}%%

\title[The units\e-Picard complex of a reductive group scheme]{The 
units\e-Picard complex of a reductive group scheme}

\author{Cristian D. Gonz\'alez-Avil\'es}
\address{Departamento de Matem\'aticas, Universidad de La Serena, Cisternas 1200, La Serena 1700000, Chile}
\email{cgonzalez@userena.cl}
\thanks{The author is partially supported by Fondecyt grant 1160004.}
\date{\today}

\subjclass[2010]{Primary 20G35, 14L15}
\keywords{Units, Picard groups, Brauer groups, Reductive group schemes, torsors}

\maketitle

\topmargin -1cm

\smallskip

\maketitle

\dedicatory{\it In memory of Jost van Hamel}

\topmargin -1cm

\begin{abstract} Let $S$ be a locally noetherian regular scheme. We compute the units-Picard complex of a reductive $S$-group scheme $G$ in terms of the dual algebraic fundamental complex of $G$. To do so, we establish a units-Picard-Brauer exact sequence for a torsor under a smooth $S$-group scheme. 
\end{abstract}

\section{Introduction}

Let $S$ be a non-empty scheme and let $\dbs$ denote the derived category of the category of bounded complexes of abelian \'etale sheaves on $S$. A morphism of schemes $f\colon X\to S$ induces a morphism $f^{\nat}\colon \bg_{m,\le S}\to \tau_{\leq\e  1}\mathbb Rf_{\lbe *}\bg_{m,X}$ in $\dbs$ that factors through the canonical morphism (of \'etale sheaves on $S$\e) $f^{\le\flat}\colon \bg_{m,\le S}\to f_{*}\bg_{m,\le X}$. The (relative) {\it units-Picard complex of $X$ over $S$} is the object $\upicxs=C^{\e\bullet}\lbe(\le f^{\nat}\le)[1]$ of $\dbs$, where $C^{\e\bullet}\lbe(\le f^{\nat}\le)$ denotes the mapping cone of any morphism of complexes that represents $f^{\nat}$. There exists a distinguished triangle in $\dbs$
\[
\picxs[-1]\to\ruxs[1]\to\upicxs\to \picxs,
\]
where $\picxs$ denotes the \'etale relative Picard functor of $X$ over $S$ and $\ruxs=C^{\e\bullet}\lbe(\le f^{\le\flat}\le)$ is the complex of relative units of $X$ over $S$. Except for a shift, $\upicxs$ was originally introduced by Borovoi and van Hamel in the case where $S$ is the spectrum of a field of characteristic zero \cite{bvh1}. The more general object $\upicxs$ discussed in this paper was introduced in \cite[\S3]{ga5}, where the reader can find proofs of some of its main properties. Borovoi and van Hamel, and later Harari and Skorobogatov \cite{hsk}, showed that $\upicxs$ is well-suited for simplifying and generalizing a number of classical constructions \cite{san, cts}. Further, we have shown in \cite{ga5} that $\upicxs$ is important in the study of the Brauer group of $X$.

We now explain the contents of the paper.

Let $k$ be a field, $G$ a (connected) reductive $k$-group scheme and $\Gtil$ the simply connected central cover of $G$. Let $T$ be a maximal $k$-torus in $G$ and set $\Ttil=T\times_{G}\Gtil$. The {\it dual algebraic fundamental complex of $G$} is the object $\pi_{1}^{\le D}\lbe(G)$ of $D^{\e b}\lbe(k_{\e\et})$ represented by the mapping cone of the morphism $T^{\e *}\to \Ttil^{\e *}$ of $k$-group schemes of characters  induced by the canonical morphism $\Gtil\to G$. Up to isomorphism, $\pi_{1}^{\le D}\lbe(G)$ is independent of the choice of $T$. The main theorem of \cite{bvh1} establishes the existence of an isomorphism in $D^{\e b}\lbe(k_{\e\et})$ which is functorial in $G$:
\begin{equation}\label{upp}
{\rm UPic}_{\e G/\le k}\isoto\pi_{1}^{\le D}\lbe(G)
\end{equation}
(if the characteristic of $k$ is $0$, the reductivity assumption on $G$ can be dropped). Over a general base scheme $S$, $\pi_{1}^{\le D}\lbe(G)$ can be defined in terms of a {\it $t$-resolution} of $G$, i.e., a central extension of $S$-group schemes $1\to T\to H\to G\to 1$, where $T$ is an $S$-torus and $H$ is a reductive $S$-group scheme whose derived subgroup $H^{\lbe\der}$ is simply connected. Namely, if $R=H^{\tor}=H/H^{\lbe\der}$ is the largest quotient of $H$ that is an $S$-torus, then $\pi_{1}^{\le D}\lbe(G)$ is (represented by) the cone of the morphism of \'etale twisted constant $S$-group schemes $R^{\e *}\to T^{\e *}$ induced by $T\to H$. If $G$ contains a maximal $S$-torus $T$, then there exists a $t$-resolution $1\to\Ttil\to H\to G\to 1$ of $G$ where $H^{\tor}$ is canonically isomorphic to $T$ and we recover the Borovoi-van Hamel definition of $\pi_{1}^{\le D}\lbe(G)$. Up to isomorphism, $\pi_{1}^{\le D}\lbe(G)$ is independent of the choice of a $t$-resolution of $G$. The aim of this paper is to establish the following generalization of \eqref{upp}:

\begin{theorem} \label{1.1} {\rm (=Theorem \ref{ct})} Let $S$ be a locally noetherian regular scheme and let $G$ be a reductive $S$-group scheme. Then there exists an isomorphism in $\dbs$
\[
\upicgs\isoto\pi_{1}^{\le D}\be(G)
\]
which is functorial in $G$.
\end{theorem}

As in \cite{bvh1}, the main ingredients of the proof of Theorem \ref{1.1} are certain long exact sequences of abelian groups of the form
\begin{equation}\label{long}
\dots\to G_{2}^{\e *}\to G_{1}^{\e *}\to \pic G_{3}\to \pic G_{2}\to\dots
\end{equation}
induced by short exact sequences of $S$-group schemes $1\to G_{1}\to G_{2}\to G_{3}\to 1$. In this paper we derive all the sequences of the form \eqref{long} that we need for the proof of Theorem \ref{1.1} from the {\it units-Picard-Brauer exact sequence of a torsor} given in Proposition \ref{upb}. The latter is a broad generalization of an exact sequence of Sansuc \cite[Proposition 6.10, p.~43]{san}, which itself generalizes a well-known exact sequence of Fossum and Iversen \cite[Proposition 3.1]{fi}. On the other hand, in constrast to \cite{bvh1} (and \cite[Appendix B]{ct08}), we avoid working with explicit presentations of $\upicgs$ (in terms of rational functions and divisors on $G\e$) since such presentations are not needed for the proof of Theorem \ref{1.1} (as already suggested in \cite[Remark 4.6]{bvh1}).

\smallskip

The theorem has interesting applications to the arithmetic of torsors under reductive $S$-group schemes, as we hope to show in a subsequent publication.

\section*{Acknowledgements}

I thank Mikhail Borovoi for helpful comments.

\section{Preliminaries}

\subsection{Generalities}\label{gen}

If $X$ is an object of a category, the identity morphism of $X$ will be denoted by $1_{\be X}$. An {\it exact and commutative diagram} is a commutative diagram with exact rows and columns. The category of abelian groups will be denoted by $\abe$.

The {\it maximal points} of a topological space are the generic points of its irreducible components. If $f\colon X\to S$ is a morphism of schemes, a {\it maximal fiber of $f$} is a fiber of $f$ over a maximal point of $S$. Recall also that $f$ is called {\it schematically dominant} if the canonical morphism (of Zariski sheaves on $S\e$) $f^{\#}\!\colon\! \s O_{\be S}\to f_{*}\s O_{\be X}$  is injective \cite[\S5.4]{ega1}. By \cite[Proposition 52, p.~10]{pic}, a faithfully flat morphism is schematically dominant.

\smallskip

\begin{lemma}\label{ker-cok}  Let  $A\overset{\!u}{\to}B\overset{\!v}{\to}C$ be  morphisms in an abelian category $\s A$. Then there exists a canonical exact sequence in $\s A$
\[
0\to\krn u\to\krn\lbe(\e v\lbe\circ\lbe\be u\e)\to\krn v\to\cok u\to\cok\be(\e v\lbe\circ\lbe u\e)\to\cok v\to 0.
\]
\end{lemma}
\begin{proof} See, for example, \cite[\S1.2]{bey}.
\end{proof}

\smallskip

If $S$ is a scheme (which is tacitly assumed to be non-empty throughout the paper), we will write $S_{\fl}$ for the small fppf site over $S$. Thus $S_{\fl}$ is the category of $S$-schemes that are flat and locally of finite presentation over $S$ equipped with the fppf topology, i.e., a covering of $(X\to S\e)\in S_{\fl}$ is a surjective family $U_{i}\to X$ of flat $S$-morphisms locally of finite presentation. We will also need the small \'etale site $S_{\et}$, which is defined as above by writing {\it \'etale} in place of {\it flat and locally of finite presentation}.

\smallskip

A sequence of $S$-group schemes $1\to G_{1}\to G_{2}\to G_{3}\to 1$ will be called {\it exact} if the corresponding sequence of (representable) sheaves of groups on $S_{\fl}$
is exact.

\smallskip

We will write $\ses$ for the (abelian) category of abelian sheaves on $S_{\et}$ and $\cbs$ for the category of bounded complexes of objects of $\ses$. The corresponding derived category will be denoted by $\dbs$. The {\it (mapping) cone} of a morphism  $u\colon A^{\bullet}\to B^{\le\bullet}$ in $\cbs$ is the complex $C^{\e\bullet}(u)$ whose $n$-th component is $C^{\e n}(u)=A^{n+1}\oplus B^{\le n}$, with differential $d_{\lle C(u)}^{\e n}\lbe(a,b)=(-d^{\e n+1}_{\be A}\lbe(a), u(a)+d_{\lbe B}^{\e n}(b))$, where $n\in\Z,a\in A^{n+1}$ and $b\in B^{\le n}$. If $u\colon A\to B$ is a morphism in $\ses$ and $A$ and $B$ are regarded as complexes concentrated in degree 0, then $C^{\e\bullet}\lbe(u)=(\e A\overset{\!u}\to B\e)$, where $A$ and $B$ are placed in degrees $-1$ and $0$, respectively. Thus $H^{-1}(C^{\e\bullet}\lbe(u))=\krn\e u$ and $H^{0}(C^{\e\bullet}\lbe(u))=\cok u$. The distinguished triangle corresponding to $u$ is
\begin{equation}\label{v}
C^{\e\bullet}\lbe(u)[-1]\overset{v}{\to} A\overset{u}\to B\to C^{\e\bullet}\lbe(u),
\end{equation}
where the map $v\colon C^{\e\bullet}\lbe(u)[-1]=(\e A\overset{\!u}\to B\e)[-1]\to A$ is the negative of the canonical projection. We also recall that, if $A^{\bullet}$ is a bounded below complex of objects of $\ses$ and $n\geq 0$ is an integer, then the {\it $n$-th truncation} of $A^{\bullet}$ is the object of $\cbs$
\[
\tau_{\leq\le n}\le A^{\bullet}=\dots \to A^{n-2}\to A^{n-1}\to\krn\lbe[\le A^{n}\to A^{n+1}]\to 0.
\]
For every $n\in\N$, there exists a distinguished triangle in $\dbs$:
\begin{equation}\label{n}
\tau_{\leq\e  n-1}A^{\bullet}\overset{i^{n}}\to \tau_{\leq\e  n}\le A^{\bullet}\to H^{n}(A^{\bullet})[-n]\to(\tau_{\leq\e  n-1}A^{\bullet})[1].
\end{equation}

Next we recall the definition of the separable index of a scheme over a field. If $k$ is a field with fixed separable algebraic closure $\ks$ and $X$ is a $k$-scheme such that $X\lbe(k^{\e\prime}\le)\neq \emptyset$ for some finite subextension $k^{\e\prime}\be/k$ of $\ks\be/k$, then the {\it separable index of $X$ over $k$} is the integer
\[
I(X)=\gcd\{[k^{\e\prime}\be\colon\be k\e]\colon \text{$k^{\e\prime}\!/\le k\subset \ks\be/k$ finite and $X\!\left(k^{\e\prime}\e\right)\neq \emptyset$}\}.
\]
The positive integer $I(X)$ is defined if $X$ is geometrically reduced and locally of finite type over $k$ (see \cite[\S3.2, Proposition 2.20, p.~93]{liu} for the finite type case and note that the more general case can be obtained by applying the indicated reference to a nonempty open affine subscheme of $X$). The above is a particular case of the following definition: if $f\colon X\to S$ is a morphism of schemes, the {\it \'etale index} $I(\le f\le)$ of $f$ is the greatest common divisor
of the degrees of all finite \'etale quasi-sections of $f$ of constant degree, if any exist. Note that $I(\le f\le)$ is defined (and is equal to 1) if $f$ has a section.

\subsection{Torsors} \label{tsec} If $G$ is an $S$-group scheme, a  {\it right $G$-scheme} (over $S\e$) is an $S$-scheme $X$ equipped with a right action of $G$ over $S$, i.e., a morphism of schemes $X\times_{S}\le G\to X$ such that, for any $S$-scheme $T$, the corresponding maps of sets $X(\le T\e)\times G(\le T\e)\to X(\le T\e), (x,g)\mapsto xg,$ form a functorial system of right actions of the ordinary group $G(\le T\e)$ on the set $X(\le T\e)$. A left 
$G$-scheme is defined similarly, and the preceding definitions also apply if $X$ is a (possibly non-representable) fppf sheaf of sets on $S$. For example, if $u\colon G\to G^{\e\prime}$ is a morphism of $S$-group schemes, then $X=G^{\e\prime}$ is a left $G$-scheme via $u$, i.e., relative to the left action $G\times_{S}G^{\e\prime}\to G^{\e\prime}, (\e g,g^{\e\prime}\e)\mapsto u(\e g)g^{\e\prime}$.

If $G$ and $G^{\e\prime}$ are $S$-groups schemes and $X$ is both a left $G$-scheme (or fppf sheaf) and a right $G^{\e\prime}$-scheme, then $X$ is called a {\it $(G,G^{\e\prime}\e)$-scheme} (over $S\e$) if the given actions are compatible, i.e., if $g(\lle xg^{\e\prime}\e)=(\e g\le x\e)g^{\e\prime}$ for all $g\in G(\le T\e), x\in X(\le T\e)$ and $g^{\e\prime}\in G^{\e\prime}(\le T\e)$, where $T$ is any $S$-scheme. For example, if $u\colon G\to G^{\e\prime}$ is a morphism of $S$-group schemes and $G^{\e\prime}_{\lbe r}$ denotes $G^{\e\prime}$ regarded as a left $G$-scheme via $u$ and as a right $G^{\e\prime}$-scheme via right translations, then $G^{\e\prime}_{\lbe r}$ is a $(G,G^{\e\prime}\e)$-scheme.

Now let $X$ (respectively, $Y$\e) be a right (respectively, left\e) $G$-scheme. Then $X\times_{S} Y$ is a right $G$-scheme under the action $(\e X\be\times_{S}\lbe Y\e)\times_{S}G\to X\!\times_{S}\be Y, (x,y,g)\mapsto (\e xg,g^{-1}y)$. We will write $X\wedge^{\lbe G}\e Y$ for the corresponding quotient fppf sheaf of sets $X\times_{S}Y/G$. A particular instance of the preceding construction arises as follows: if $u\colon G\to G^{\e\prime}$ and $v\colon G\to G^{\e\prime\prime}$ are morphisms of $S$-group schemes and $G^{\e\prime}$ (respectively, $G^{\e\prime\prime}$) is regarded as a right (respectively, left) $G$-scheme via $u$ (respectively, $v$), then the sheaf $G^{\e\prime}\!\wedge^{\be u,\le G,\le v} G^{\e\prime\prime}$ is called the {\it pushout} of $u$ and $v$. We also note that, if $X$ is any right $G$-scheme and $Y$ is a $(G,G^{\e\prime}\e)$-scheme, then the right action of $G^{\e\prime}$ on $Y$ endows $X\be\wedge^{\lbe G}\e Y$ with the structure of a right $G^{\e\prime}$-sheaf. In particular, if $u\colon G\to G^{\e\prime}$ is a morphism of $S$-group schemes and $Y=G^{\e\prime}_{\lbe r}$ is the $(G,G^{\e\prime}\e)$-scheme considered previously, then the right $G^{\e\prime}$-sheaf $X\wedge^{\lbe G,\le u} G^{\e\prime}_{\be r}$ is said to be {\it obtained from $X$ by extending its structural group from $G$ to $G^{\e\prime}$ via $u$} \cite[Proposition 1.3.6, p.~116]{gi}.

Now recall that a right $G$-scheme (or fppf sheaf\e) $X$ over $S$ is called a {\it right} (fppf) {\it $G$-torsor over $S$} if the following conditions hold: (a) the canonical morphism $X\times_{S}\le G\to X\times_{S}X, (\e x,g)\to (x,xg),$ is an isomorphism, and (b) locally for the fppf topology on $S$, the structural morphism $X\to S$ has a section. We will write $H^{\le 1}(S_{\fl},G\e)$ for the pointed set of isomorphism classes of right $G$-torsors over $S$, where the base point is the isomorphism class of the trivial right $G$-torsor $G_{r}$.

For example, if $H$ is a subgroup scheme of $G$ such that the quotient fppf sheaf $G\be/\be H$ is represented by an $S$-scheme $F$, then $G$ is an $H_{\lbe F}\e$-\e torsor over $F$ if $H_{F}$ acts on $G$ via right translations. In this case condition (a) above is a generalization of the following ordinary group\e-\le theoretic fact: if $G$ is an ordinary group, $H$ is a subgroup of $G$ and $(\e g,g^{\e\prime}\le)\in G\times_{(\le G/H\le)}G$, i.e., $gH=g^{\e\prime}H$, then $g^{\e\prime}=gh$ for a unique element $h\in H$, namely $h=g^{-1}g^{\e\prime}$\,\e\footnote{\e That $G\to F$ satisfies condition (b) of the definition follows from the fact that the map of sets $G(\le S\e)\to (G/H)(\le S\e)\simeq F(\le S\e)$ is locally surjective for the fppf topology on $S\e$.}\e.
More generally (in fact, equivalently), if 
\begin{equation}\label{pri}
1\to H\overset{\!i}{\to} G\overset{\!q}{\to} F\to 1
\end{equation}
is an exact sequence of $S$-group schemes, then the right action 
$G\lbe\times_{ F}\be(\e H\be\times_{\lbe S}\be F\e)\to G,\, (\e g, (\e h,q\le(\e g\le)))\mapsto g\le\le i\lbe(\lle h\le)$, endows $G$ with the structure of a right $H_{F}\e$-\e torsor over $F$. In this case we will write $[\e G\,]$ for the isomorphism class of $G$ in $H^{\le 1}\be(F_{\fl},H_{\lbe F}\be)$.

If $u\colon G\to G^{\e\prime}$ is a morphism of $S$-group schemes and $X$ is a right $G$-torsor over $S$, then the right $G^{\e\prime}$-sheaf  $X\wedge^{\be G,\le u} G^{\e\prime}_{\lbe r}$ defined above is a right $G^{\e\prime}$-torsor over $S$ \cite[Proposition 1.4.6(i), p.~119]{gi}. Thus $u$ induces a map of pointed sets
\begin{equation}\label{u1}
u^{(1)}\colon   H^{\le 1}\be(S_{\fl},G\e)\to H^{\le 1}\be(S_{\fl},G^{\e\prime}\e), [\e X\e]\mapsto [\e X\!\wedge^{\be\lle G,\le u}\! G^{\e\prime}_{\lbe r}\,]\e.
\end{equation}
If $T\to S$ is any morphism of schemes, then the following diagram of pointed sets commutes
\[
\xymatrix{H^{\le 1}\be(S_{\fl},G\e)\ar[d]\ar[r]^{u^{(1)}}& H^{\le 1}\be(S_{\fl},G^{\e\prime}\e)\ar[d]\\
H^{\le 1}\be(T_{\fl},G_{T}\e)\ar[r]^{u_{T}^{(1)}}& H^{\le 1}\be(T_{\fl},G_{\lbe T}^{\e\prime}\e),
}
\]
where the vertical maps are induced by $X\mapsto X_{T}$ for any $G$ (respectively, $G^{\e\prime}\e$)-torsor $X$ over $S$. In other words, there exists a (canonical) isomorphism of $G_{T}\e$\e-torsors over $T$
\begin{equation}\label{bcf}
(X\!\wedge^{\be\lle G,\le u}\! G^{\e\prime}_{\le r}\e)\times_{S}T\simeq X_{T}\!\wedge^{\be\lle G_{T},\le u_{T}}\! G_{T,\e r}^{\e\prime}.
\end{equation}
See \cite[1.5.1.2, p.~316]{gi}.

Now let $u\colon H\to K$ be a morphism of $S$-group schemes, where $H=\krn\e i$ is the left-hand term in \eqref{pri}, and let $P=G\be\wedge^{ i,\e H\lbe,\e u}\be\lbe K$ be the pushout of $i$ and $u$. Then there 
exists an exact and commutative diagram of fppf sheaves of groups on $S$
\begin{equation}\label{pdia}
\xymatrix{1\ar[r]& H\ar[d]_{u}\ar[r]^{i}& G\ar[d]\ar[r]^{q}& F\ar@{=}[d]\ar[r]&1\\
1\ar[r]& K\ar[r]^{j}& P\ar[r]^{p}& F\ar[r] & 1,
}
\end{equation}
where the maps $j$ and $p$ are defined as follows: if $\pi_{\lbe K}\colon K\to S$ is the structural morphism of $K$ and $\varepsilon_{\lbe G}\colon S\to G$ and $\varepsilon_{\lbe F}\colon S\to F$ are the unit sections of $G$ and $F$, respectively, then $j$ and $p$ are induced, respectively, by $(\varepsilon_{\lbe G}\circ\pi_{\be K},1_{\lbe K}\lbe)_{S}\colon K\to G\times_{S}K$ and $m\circ (\, q\times_{S}(\varepsilon_{\lbe F}\circ\pi_{\be K}))\colon G\times_{S}K\to F\times_{S}F\to F$, where $m\colon F\times_{S}F\to F$ is the product morphism of $F$. As explained above, the bottom row of diagram \eqref{pdia} equips $P$ with the structure of a right $K_{\be F}\e$-\e torsor over $F$. Using the explicit definitions of the maps $j$ and $p$ given above, it is not difficult to check that the canonical isomorphism of $S$-schemes $G\times_{ S} K\isoto G\times_{ F} K_{\lbe F}$ induces an isomorphism of right $K_{\be F}\e$\e-torsors over $F$
\begin{equation}\label{tbc}
G\be\wedge^{ i,\e H\lbe,\e u}\! K\isoto G\be\wedge^{\be H_{\lbe F}\lbe,\e u_{\lbe F}}\! K_{\be F,\e r}\le.
\end{equation}
Thus $[\e P\e]=u_{F}^{(1)}\lbe(\lle [\e G\e]\lle)$, where $u_{F}^{(1)}\colon H^{\le 1}\lbe(F_{\fl},H_{\lbe F}\lbe)\to H^{\le 1}\lbe(F_{\fl},K_{\lbe F}\lbe)$ is the map \eqref{u1} induced by the morphism of $F$-group schemes $u_{F}\colon H_{F}\to K_{F}$.

\medskip

\subsection{Units, Picard groups and Brauer groups} Let $f\colon X\to S$ be a morphism of schemes and let $f^{\e\flat}\colon \bg_{m,\le S}\to f_{\lbe *}\bg_{m,\le X}$ be the morphism of abelian sheaves on $S_{\et}$ induced by $f$. The {\it complex of relative units of $X$ over $S$} is the complex
\begin{equation}\label{ruxs}
\ruxs=C^{\e\bullet}\lbe(\e f^{\e\flat})=(\bg_{m,\le S}\overset{f^{\le\flat}}{\to} f_{\lbe *}\bg_{m,\le X}),
\end{equation}
where $\bg_{m,\le S}$ and $f_{\lbe *}\bg_{m,\le X}$ are placed in degrees $-1$ and $0$, respectively.

If $X\overset{\!g}{\to}Y\overset{\!h}{\to}S$ are morphisms of schemes, then
$(\e h\!\circ\! g\e)^{\le\flat}\colon \bg_{m,S}\to h_{\le *}(\e g_{\le *}\bg_{m,\le X}\lbe)$ factors as
\begin{equation}\label{pair}
\bg_{m,S}\overset{h^{\flat}}{\to}h_{\le *}\bg_{m,Y}\overset{h_{\lle *}\lbe\lbe(\le g^{\le\flat}\lbe)}{\lra}h_{\le *}\lbe(\e g_{\le *}\bg_{m,X}\lbe),
\end{equation}
Consequently, $h_{\le *}\lbe(\e g^{\e\flat})$ induces a morphism ${\rm RU}_{\lbe S}(\e g)\colon\ruys\to\ruxs$ in $\cbs$. Thus \eqref{ruxs} defines a contravariant functor
\[
{\rm RU}_{\lbe S}\colon (\textrm{Sch}/S\le)\to\cbs, (X\to S\e)\mapsto\ruxs.
\]
Now set
\begin{equation}\label{uxs2}
\uxs=H^{\e 0}(\ruxs)=\cok f^{\le\flat}
\end{equation}
and consider the contravariant functor
\[
{\rm U}_{\be S}\colon (\textrm{Sch}/S\le)\to\ses, (X\to S\e)\mapsto\uxs.
\]
If $f$ is schematically dominant, then $f^{\e\flat}$ is injective \cite[Lemma 2.4]{ga4} and $\ruxs=\uxs$ in degree 0.

\begin{lemma}\label{use} Let $X\overset{\!g}{\to}Y\overset{\!h}{\to}S$ be morphisms of schemes, where $g$ is schematically dominant. Then the canonical morphism ${\rm U}_{\be S}(\e g)\colon\uys\to\uxs$ is injective.	
\end{lemma}
\begin{proof} Apply Lemma \ref{ker-cok} to the pair of morphisms \eqref{pair} using the injectivity of
$h_{*}\lbe\lbe(\le g^{\le\flat}\lbe)\colon	\bg_{m,Y}\to h_{*}\lbe(\e g_{*}\bg_{m,X}\lbe)$.	
\end{proof}

If $G$ is an $S$-group scheme, the presheaf of groups
\begin{equation}\label{gc}
G^{\e *}=\underline{\hom}_{\, S\text{-gr}}(G,\bg_{m,\e S}\e)
\end{equation}
is a sheaf on $S_{\et}$ such that
\begin{equation}\label{gc2}
G^{\le\lle *}\be(\lbe S\le)=\hom_{S\text{-gr}}(G,\bg_{m,\e S}\e).
\end{equation}
See \cite[IV, Corollary 4.5.13 and Proposition 6.3.1(iii)]{sga3}.

\begin{lemma}\label{gros} Let $S$ be a reduced scheme and let $G$ be a 
flat $S$-group scheme locally of finite presentation with smooth and connected maximal fibers. Then there exists a canonical isomorphism of \'etale sheaves on $S$
\begin{equation}\label{us}
\omega_{\le G}\colon \ugs\isoto G^{\e *},
\end{equation}
where $\ugs$ and $G^{\e *}$ are given by \eqref{uxs2} and \eqref{gc}, respectively.
\end{lemma}
\begin{proof} See \cite[Lemma 4.8]{ga4}. The isomorphism of abelian groups $\omega_{G}(S\e)$ is induced by the map
\begin{equation}\label{qmap}
\hom_{S\text{-sch}}(G,\bg_{m,\e S}\e)\to\hom_{S\text{-gr}}(G,\bg_{m,\e S}\e), v\mapsto v/v\be\circ\be\varepsilon\be\circ\be \pi,
\end{equation}
where $\varepsilon\colon S\to G$ and $\pi\colon G\to S$ and the unit section and structural morphism of $G$, respectively.
\end{proof}

If $X$ is a scheme, the \'etale cohomology group $\pic X=H^{\le 1}(X_{\et},\bg_{m,\le X})$ will be identified with the group of isomorphism classes of right $\bg_{m,\le X}$-torsors over $X$ with respect to either the \'etale or fppf topologies on $X$. See \cite[Theorem 4.9, p.~124]{mi1}. A morphism of schemes $g\colon X\to Y$ induces a morphism of abelian groups
\begin{equation}\label{pmap}
\pic g\colon \pic Y\to \pic X, [\e E\e]\mapsto [\e E\!\times_{Y}\!\be X\e],
\end{equation}
where $[\e E\e]$ denotes the isomorphism class of the right $\bg_{m,\le Y}$-torsor $E$ over $Y$ and $[\e E\!\times_{Y}\!\be X\e]$ denotes the isomorphism class of the right $\bg_{m,\le X}$-torsor $E\!\times_{Y}\!\be X$ over $X$.

Next we write $f^{\le p}\colon\bg_{m,\e S}\to f_{*}\bg_{m,\e X}$ for the morphism of abelian presheaves on $S_{\fl}\e$ such that, for any object $T \to S$ of $S_{\fl}$, $f^{\le p}\lbe(\le T\le)\colon \bg_{m,\e S}(\le T\le)\to(\e f_{*}\bg_{m,\e X})(\le T\le)=\bg_{m,\e S}(X_{T})$ is the canonical map induced by the projection $X_{T}\to T$. Now consider the following abelian presheaf on $S_{\fl}\e$
\[
\mathcal U_{X\be/\lbe S}=\cok\be[\e \bg_{m,\e S}\overset{\!f^{\le p}}{\lra} f_{*}\bg_{m,\e X}\e]
\]
and set
\begin{equation}\label{pu}
\mathcal U_{\le S}\lbe(\lbe X)=\mathcal U_{X\be/S}(S\e)=\cok\be[\e \bg_{m,\e S}(S\e)\to \bg_{m,\e S}(X)\e].
\end{equation}
Thus we obtain an abelian presheaf on $(\textrm{Sch}/S\e)$
\begin{equation}\label{puf}
\mathcal U_{\le S}\colon (\textrm{Sch}/S\e)\to\abe,\, (X\to S\e)\mapsto \mathcal U_{\le S}\lbe(\lbe X).
\end{equation}

\begin{remark} \label{too} If $\mathcal U_{X\be/\lbe S}$ is regarded as an abelian presheaf on $S_{\et}$, then the \'etale sheaf on $S$ associated to $\mathcal U_{X\be/\lbe S}$ is $\uxs$ \eqref{uxs2}.
\end{remark}

If $f$ is schematically dominant, then $f^{\le p}$ is an injective morphism of abelian presheaves on $S_{\fl}$ \cite[proof of Lemma 2.4]{ga4}. Thus, in this case, there exists a canonical exact sequence of abelian presheaves on $S_{\fl}$
\begin{equation}\label{cane}
1\to \bg_{m,\e S}\overset{\!f^{\le p}}{\lra} f_{*}\bg_{m,\e X}\to \mathcal U_{X\be/\lbe S}\to 1.
\end{equation}
Further, by \cite[comments after Lemma 2.5]{ga4}, there exists a canonical exact sequence of abelian groups
\[
0\to\bg_{m,\le S}(S\e)\to \bg_{m,\le S}(X\lbe)\to \uxs(S\e)\to\pic \e S\overset{\!\pic\be f}{\lra} \pic X,
\]
where $\uxs$ is the \'etale sheaf \eqref{uxs2}. The preceding sequence induces an exact sequence of abelian groups
\[
0\to\mathcal U_{\le S}(X)\to \uxs(S\e)\to\pic \e S\overset{\!\pic\be f}{\lra} \pic X.
\]	
Thus, if $\pic\be f$ is injective (e.g., $f$ has a section), then $\mathcal U_{\le S}(X)=\uxs(S\e)$. In particular, if $G$ is an $S$-group scheme, then $\mathcal U_{\le S}(G\e)=\ugs(S\e)$ and the following statement is immediate from Lemma \ref{gros}.

\begin{lemma}\label{pug} Let $S$ be a reduced scheme and let $G$ be a flat $S$-group scheme locally of finite presentation with smooth and connected maximal fibers. Then there exists a canonical isomorphism of abelian groups
\[
\mathcal U_{\le S}(G\e)\isoto G^{\e *}\be(S\e),
\]
where the groups $\mathcal U_{\le S}(G\e)$ and $G^{\e *}\be(S)$ are given by \eqref{pu} and \eqref{gc2}, respectively.
\end{lemma}

Now let $\picxs$ be the (\'etale) {\it relative Picard functor of $X$ over $S$}, i.e., the \'etale sheaf on $S$ associated to the abelian presheaf $(\textrm{Sch}/S\e)\to \mathbf{Ab},(\le T\to S\e)\mapsto\pic X_{T}\e$. Then
\[
\picxs=R^{\le 1}_{\et}\le f_{\lbe *}\bg_{m,X}.
\]
If $g\colon X\to Y$ is a morphism of $S$-schemes and $T$ is an $S$-scheme, then the canonical maps $\pic\lbe g_{\e T}\colon \pic Y_{T}\to \pic X_{T}$ \eqref{pmap} induce a morphism $\pic_{\! S}(\le g)\colon \picys\to\picxs$ in $\ses$. Thus we obtain a contravariant functor
\[
\pic_{\! S}\colon (\textrm{Sch}/S\e)\to\ses, (X\to S\e)\mapsto\picxs.
\]
We will also need to consider the abelian group
\begin{equation}\label{npic}
\npicxs=\cok\be[\e \pic S\overset{\!\pic\be f}{\to}\pic X\e]
\end{equation}
and the associated abelian presheaf on $(\textrm{Sch}/S\e)$
\begin{equation}\label{pnp}
\nps\colon (\textrm{Sch}/S\e)\to\abe, (\e X\to S\e)\mapsto \npicxs.
\end{equation}

\begin{remark} \label{sl} If $S$ is a strictly local scheme and $X\to S$ is quasi-compact and quasi-separated, then $\picxs(S\e)=\pic X$. See \cite[Lemma 6.2.3, p. ~124, and Theorem 6.4.1, p.~128]{t}.
\end{remark}

Next we recall from \cite[\S3]{ga5} the definition of the units-Picard complex of a morphism of schemes $f\colon X\to S$. Let $A^{\bullet}$ be any representative of $\mathbb Rf_{\lbe *}\bg_{m,X}\in D(S_{\et})$ and consider the following composition of morphisms in $\cbs$:
\[
f^{\nat}\colon \bg_{m,\le S}\overset{\,f^{\le\flat}}\to f_{\lbe *}\bg_{m,\le X}\simeq\tau_{\leq\e  0}\e A^{\bullet}\overset{\! i^{1}}{\lra}
\tau_{\leq\e  1}\e A^{\bullet},
\]
where $i^{1}$ is the first map in the distinguished triangle \eqref{n} for $n=1$.
The {\it relative units-Picard complex of $X$ over $S\,$} is the following (well-defined) object of $\dbs$: 
\begin{equation}\label{upic}
\upicxs=C^{\e\bullet}\lbe(\le f^{\nat}\le)[1].
\end{equation}
There exists a distinguished triangle in $\dbs$
\begin{equation}\label{t5}
\picxs[-1]\to\ruxs[1]\to\upicxs\to \picxs,
\end{equation}
where $\ruxs$ is the complex \eqref{ruxs}.

If $f\colon X\to S$ is schematically dominant, then $f^{\le\flat}$ is injective and therefore $\ruxs=\uxs$ in $\dbs$. In this case \eqref{t5} is a distinguished triangle
\begin{equation}\label{t6}
\picxs[-1]\to\uxs[1]\to\upicxs\to\picxs.
\end{equation}
Consequently, $H^{\le r}\lbe(\e\upicxs)=0$ for $r\neq -1, 0$,
\begin{equation}\label{up-1}
H^{-1}\lbe(\e\upicxs)=\uxs
\end{equation}
and
\begin{equation}\label{up-0}
H^{\e 0}\lbe(\e\upicxs)=\picxs.
\end{equation}
If $X\overset{\!g}{\to}Y\overset{\!h}{\to}S$ are morphisms of schemes, then $g^{\e\flat}\colon\bg_{m,\e Y}\to g_{\le*}\bg_{m,\e X}$ induces a morphism  $g^{\e\prime}\colon \tau_{\leq\e  1}{\mathbb R}\e h_{*}\lbe\bg_{m,\e Y}\to \tau_{\leq\e  1}{\mathbb R}\e (h\circ g)_{*}\lbe\bg_{m,\e X}$ in $\dbs$ such that 
$g^{\e\prime}\circ h^{\nat}=(h\circ g)^{\nat}$. Thus $g^{\e\prime}$ induces a morphism
$\upic_{S}(\e g)\colon \upicys\to \upicxs$ in $\dbs\e$ and we obtain a contravariant functor
\[
\upic_{S}\colon ({\rm Sch}\lbe/\be S\e)\to\dbs, (X\to S)\to\upicxs,
\]
such that $H^{-1}\lbe(\e\upic_{S})=H^{\e 0}(\e{\rm RU}_{S}\lbe)={\rm U}_{S}$ by \eqref{t5}. Further, if $h$ and $h\!\circ\! g$ are schematically dominant, then the following diagram in $\dbs$ commutes
\begin{equation}\label{cderr}
\xymatrix{\uys[1]\ar[d]\ar[rr]^{{\rm U}_{\lbe S}(\le g)[1]}&&\uxs[1]\ar[d]\\
\upicys\ar[rr]^{\upic_{S}(\le g)}&&\upicxs,
}
\end{equation}
where the vertical arrows are those in \eqref{t6}.

\medskip

Next let $\brp X=H^{2}(X_{\et},\bg_{m,\le X}\lbe)$ be the cohomological Brauer group of $X\e$. We will write $\brxs$ for the \'etale sheaf on $S$ associated to the presheaf $(\le T\to S\e)\mapsto \brp X_{T}$, i.e., 
\[
\brxs=R^{\le 2}_{\et}\le f_{\lbe *}\bg_{m,\le X}.
\]
There exists a canonical morphism of abelian groups $\brp X\!\to\be\brxs(S\le)$\,\e\footnote{\e This is an instance of the canonical adjoint morphism $P(S\le)\to P^{\#}\be(S\le)$, where $P$ is a presheaf of abelian groups on $(\textrm{Sch}/S\le)$ and $P^{\le\#}$ is its associated \'etale sheaf \cite[Remark, p.~46]{t}.} and we set
\begin{equation}\label{br1}
\bro=\krn\!\be\left[\brp X\be\to\be\brxs(S\le)\e\right]\!.
\end{equation}
Note that, if $f=1_{\lbe S}\colon S\to S$, then $\br_{\be 1}^{\prime}(S/S\e)=\brp S$. We obtain an abelian presheaf on $(\textrm{Sch}/S\e)$
\[
\br_{\be 1,\e S}^{\prime}\colon(\textrm{Sch}/S\e)\to\abe,(\e X\to S\e)\mapsto \bro.
\]
If $h$ is a morphism of $S$-schemes, we will write $\br_{\be 1}^{\prime} h$ for $\br_{\be 1,\e S}^{\prime}(\e h\e)$.

Next, since $f\colon X\to S$ is a morphism of $S$-schemes and $\br_{\be 1}^{\prime}(S/S\e)=\brp S$, we may consider 
\begin{equation}\label{bra}
\bra=\cok\!\be\left[\e \brp\lle S\overset{\!\br_{\!\be 1}^{\prime}\lbe f}{\to}\bro\e\right]
\end{equation}
and the corresponding abelian presheaf on $(\textrm{Sch}/S\e)$
\begin{equation}\label{braf}
\bras\colon (\textrm{Sch}/S\e)\to\abe, (\e X\to S\e)\mapsto \bra.
\end{equation}
If $h$ is a morphism of $S$-schemes, we will write $\br_{\be \rm a}^{\prime}\e h$ for $\br_{\be \rm a,\e S}^{\prime}(\e h\e)$.

The groups  \eqref{br1} and \eqref{bra} are related by an exact sequence
\begin{equation}\label{br-seq}
\brp S\overset{\!\br_{\be 1}^{\prime}\lbe f}{\lra}\bro\overset{\! c_{\le(\lbe X\lbe)}}{\lra}\bra\to 0,
\end{equation}
where $c_{\le(\lbe X)}$ is the canonical projection. 

\smallskip

If $f$ has a section $\sigma\colon S\to X$, then $\br_{\be 1}^{\prime}\sigma\colon\bro\to\br_{\be 1}^{\prime}(S/S\e)=\brp S$ is a retraction of $\br_{\! 1}^{\prime}\lbe f$ that splits \eqref{br-seq}. Thus, if we define
\begin{equation}\label{brs}
\br_{\be \sigma}^{\prime}(X\be/\lbe S\le)=\krn[\e\bro\!\overset{\br_{\! 1}^{\prime}\sigma}{\lra}\brp S\,],
\end{equation}
then
\begin{equation}\label{spe}
\bro=\img \br_{\! 1}^{\prime}f\e\oplus\e \br_{\be \sigma}^{\prime}(X\be/\lbe S\le)
\end{equation}
and the restriction of $c_{\le(\lbe X)}$  \eqref{br-seq} to $\br_{\be \sigma}^{\prime}(X\lbe/\lbe S\le)\subseteq\bro$ is an isomorphism of abelian groups
\begin{equation}\label{csig0}
c_{\le(\lbe X),\e\sigma}\colon \br_{\be \sigma}^{\prime}(X\lbe/\lbe S\le)\isoto \bra.
\end{equation}
Next let $g\colon Y\!\to\be S$ be another morphism of schemes with section $\tau\colon S\to Y$ and let $h\colon X\to Y$ be an $S$-morphism such that $h\e\circ\e\sigma=\tau$. Then the restriction of  $\br_{\! 1}^{\prime}\e h$ to $\img\e \br_{\! 1}^{\prime}\e g \subseteq\br_{\be 1}^{\prime}(\le Y\!/\lbe S\le)$ is an isomorphism of abelian groups
\begin{equation}\label{im}
\img\e \br_{\! 1}^{\prime}\le g\isoto \img\e \br_{\! 1}^{\prime}\le f
\end{equation} 
which fits into a commutative diagram
\[
\xymatrix{\img\e \br_{\! 1}^{\prime}\e g\ar[rr]^(.5){\simeq}&& \img\e \br_{\! 1}^{\prime}\lbe f\\
&\brp S\ar[ul]^{\simeq}\ar[ur]_(.5){\simeq}&.
}
\]
On the other hand, the restriction of $\br_{\! 1}^{\prime}\e h$ to $\br_{\be \tau}^{\prime}(\le Y\be/\lbe S\e)\subseteq \br_{\be 1}^{\prime}(\le Y\be/\lbe S\le)$ induces a map
\begin{equation}\label{csig1}
\br_{\!\sigma,\e\tau}^{\prime}\e h\colon \br_{\be \tau}^{\prime}(\le Y\be/\lbe S\e)\to \br_{\be \sigma}^{\prime}(X\be/\lbe S\e)
\end{equation}
such that (by \eqref{spe} and \eqref{im})
\begin{equation}\label{eker}
\krn\e\br_{\!\sigma,\e\tau}^{\prime}\e h=\krn\e \br_{\! 1}^{\prime}\e h
\end{equation}
and the following diagram commutes
\begin{equation}\label{fty}
\xymatrix{\br_{\be \tau}^{\prime}(\le Y\be/\lbe S\e)\ar[rr]^{\br_{\!\sigma,\e\tau}^{\prime}\le h}\ar[d]_{c_{\le( Y),\e\tau}}^{\simeq}&& \br_{\be \sigma}^{\prime}(\lbe X\be/\lbe S\e)\ar[d]^{c_{\le( X),\e\sigma}}_{\simeq}\\
\bray\ar[rr]^{\br_{\!\rm a}^{\lbe\prime}\e h}&&\bra,
}
\end{equation}
where the vertical maps are the isomorphisms \eqref{csig0}.

\medskip

\smallskip

We now discuss products. If $f\colon X\to S$ and $g\colon Y\to S$ are morphisms of schemes, then the canonical projection morphisms $p_{X}\colon X\times_{S}Y\to X$ and $p_{\e Y}\colon X\!\times_{S}\! Y\to Y$ define a morphism of abelian groups  
\begin{equation}\label{cmap0}
\bg_{m,\le S}(X\le)\be\oplus\be\bg_{m,\le S}(\le Y\le)\to \bg_{m,\le S}(\lbe X\be\times_{\be S}\be Y\le), (u,v)\mapsto p_{\be X}^{(0)}\lbe(u)\be\cdot\be p_{\lle Y}^{(0)}\lbe(v),
\end{equation}
where $p_{\lbe X}^{(0)}\colon \bg_{m,S}(X\e)\to \bg_{m,\e S}(X\!\times_{\be S}\! Y\e)$ is the pullback map and $p_{\le Y}^{(0)}$ is defined similarly. Further, there  exist canonical morphisms of abelian groups
\[
\pic X\be\oplus\be\pic Y\to\pic(X\be\times_{\be S}\be Y\le), ([E],[F]\e)\mapsto (\le\pic p_{X})[E\e]+(\le\pic p_{\e Y})[F\e],
\]
and
\[
\brp X\be\oplus\be\brp\, Y\to\brp(X\be\times_{\be S}\be Y\le), (a,b\e)\mapsto (\brp p_{X})(a)+(\brp p_{\e Y})(b\e).
\]
The preceding maps induce morphisms of abelian groups
\begin{equation}\label{cmap}
\mathcal U_{S}\lbe(\lbe X\le)\be\oplus \mathcal U_{S}\lbe(\le Y\e)\to \mathcal U_{S}\lbe(\lbe X\be\times_{\be S}\be Y\le),
\end{equation}
\begin{equation}\label{c1}
\npic\be(X\be/\be S\le)\oplus\npic\be(Y\be/\be S\le)\to\npic\be(X\!\times_{\be S}\be Y\be/\be S\le),
\end{equation}
\begin{equation}\label{c2}
\bro\oplus\broys\to\broxys
\end{equation}
and
\begin{equation}\label{c3}
\bra\oplus\bray\to\braxy.
\end{equation}
Further, there exists a canonical morphism in $\dbs$ (see \cite[beginning of \S4]{ga5})
\begin{equation}\label{upicm}
\upicxs\be\oplus\be\upicys\to\upicxys.
\end{equation}
Now, if $g\colon Y\!\to\! S$ has a section $\tau\colon S\!\to\! Y$ and $\br_{\be \tau}^{\prime}(Y\be/\lbe S\le)\subseteq \broys$ is the group \eqref{brs} associated to $\tau$, we write 
\begin{equation}\label{c0}
\zeta_{\le X,\le Y}\colon \bro\oplus\br_{\be \tau}^{\prime}(Y\be/\lbe S\le)\to\broxys
\end{equation}
for the restriction of \eqref{c2} to $\bro\oplus\br_{\be \tau}^{\prime}(Y\be/\lbe S\le)\subseteq\bro\oplus\br_{\be 1}^{\prime}(Y\be/\lbe S\le)$. If  $i_{\le\br_{\be 1}^{\prime}(X/S\le)}\colon \bro\to \bro\oplus\br_{\be \tau}^{\prime}(Y\be/\lbe S\le), a\mapsto(a,0),$ is the canonical embedding, then
\begin{equation}\label{c00}
\br_{\be 1}^{\prime}\e p_{X}=\zeta_{\e X,\le Y}\be\circ\be i_{\le\br_{\be 1}^{\prime}(X/S\le)}.
\end{equation}
Now we define 
\begin{equation}\label{c4}
\zeta_{\le X,\le Y}^{\le 0}\colon \bro\oplus\bray\to\broxys
\end{equation}
by the commutativity of the diagram
\[
\xymatrix{\bro\oplus\bray\ar[rr]^(.52){\zeta_{\lle X,\le Y}^{\le 0}}&&\broxys,\\
\ar[u]^{\big(1_{\br_{\be 1}^{\prime}\lbe(\lbe X\be/\lbe S\le)},\e c_{\le(\le Y),\e\tau}\big)}_{\simeq}\bro\oplus\br_{\be \tau}^{\prime}(Y\be/\lbe S\le)\ar[urr]_(.52){\zeta_{\lle X,\le Y}}&&
}
\]
where $c_{\le(\le Y),\e\tau}\colon \br_{\be \tau}^{\prime}(Y\be/\lbe S\le)\isoto \bray$ is the map \eqref{csig0} defined by $\tau$. Clearly, $\zeta_{\le X,\le Y}^{\le 0}$ is an isomorphism if, and only if, $\zeta_{\le X,\le Y}$ is an isomorphism. If this is the case, then \eqref{c00} shows that
\begin{eqnarray}
p_{\e \br_{\be 1}^{\prime}\lbe(X\be/S)}\circ\zeta_{\le X,\le Y}^{\e -1}\circ\br_{\be 1}^{\prime}\e p_{X}&=&1_{\br_{\be 1}^{\prime}(X/S)}\label{toi}\\
p_{\e \br_{\be \tau}^{\prime}\lbe(\le Y\be/S)}\circ\zeta_{\le X,\le Y}^{\e -1}\circ\br_{\be 1}^{\prime}\e p_{X}&=&0,\label{toi2}
\end{eqnarray}
where $p_{\e \br_{\be 1}^{\prime}\lbe(X\be/S)}\colon \bro\oplus\br_{\be \tau}^{\prime}(Y\be/\lbe S\le)\to \bro$ is the canonical projection and $p_{\e \br_{\be \tau}^{\prime}\lbe(\le Y\be/S)}$ is defined similarly. 

\begin{lemma}\label{brai} If the map $\zeta_{\le X,\le Y}^{0}$ \eqref{c4} is an isomorphism, then the map \eqref{c3}
\[
\bra\oplus\bray\to\braxy
\]
is an isomorphism as well.	
\end{lemma}
\begin{proof} The composition 
\[
\brp S\overset{\!(\br_{\be 1}^{\prime}\lbe f\lbe,\le 0)}{\to} \bro\oplus\be\bray\overset{\zeta_{\le X,\le Y}^{\le 0}}{\to}\broxys
\]
equals $(\br_{\be 1}^{\prime}\e p_{\lbe X})\be\circ\be(\br_{\be 1}^{\prime}\lbe f\le)=\br_{\be 1}^{\prime}\lbe(\e f\be\times_{\be S}\be g)$. Thus Lemma \ref{ker-cok} applied to the above pair of maps yields an exact sequence of abelian groups
\[		
\krn\e \zeta_{\le X,\le Y}^{\le 0}\to \bra\oplus\bray\to\braxy\to\cok \zeta_{\le X,\le Y}^{\le 0},
\]
where the middle arrow is the map \eqref{c3}. The lemma is now clear.
\end{proof}

\subsection{Additivity theorems}
\begin{proposition}\label{padd} Let $S$ be a reduced scheme and let $f\colon X\to S$ and $g\colon Y\to S$ be faithfully flat morphisms locally of finite presentation with reduced and connected maximal geometric fibers. Assume that $f$ has an \'etale quasi-section and $g$ has a section or, symmetrically, that $f$ has a section and $g$ has an \'etale quasi-section. Then the canonical map \eqref{cmap}
\[
\mathcal U_{S}\lbe(\lbe X\le)\be\oplus \mathcal U_{S}\lbe(\le Y\e)\to \mathcal U_{S}\lbe(\lbe X\be\times_{\be S}\be Y\le)
\]
is an isomorphism of abelian groups.
\end{proposition}
\begin{proof} Since $f, g$ and $f\!\times_{\lbe S}\be g$ are schematically dominant, the following diagram of abelian groups, whose exact rows are induced by \eqref{cane}, commutes:	
\[
\xymatrix{1\ar[r]&\bg_{m,\le S}\le(S\le)\be\oplus\be \bg_{m,\le S}\le(S\le)\,\ar[r]^(.5){\iota}\ar@{->>}[d]^{(\e\cdot\e)}& \bg_{m,\le S}(X\le)\oplus\bg_{m,\le S}(\le Y\le) \ar[r]\ar[d]& \mathcal U_{S}\lbe(\lbe X\le)\be\oplus \mathcal U_{S}\lbe(\le Y\e)\ar[d]\ar[r]&1\\
1\ar[r]&\bg_{m,\le S}\le(S\le)\,\ar[r]& \bg_{m,\e S}(X\!\times_{\be S}\! Y\le)\ar[r]& \mathcal U_{S}\lbe(\lbe X\be\times_{\be S}\be Y\le)\ar[r]&1.
}
\]
The left-hand vertical map in the above diagram is the (surjective) multiplication homomorphism. The middle and right-hand vertical arrows are the maps \eqref{cmap0} and \eqref{cmap}, respectively. Now, since $f\lbe\times_{\lbe S}\lbe g$ has an \'etale quasi-section, the map labeled $\iota$ in the above diagram induces an isomorphism between the kernels of the first two vertical arrows. See \cite[Corollary 4.5 and its proof]{ga4}. Further, the cokernel of the middle vertical map is canonically isomorphic to $\krn\pic\be f\le\cap\le \krn\pic\lbe g\subseteq\pic S$. Thus the diagram yields an exact sequence of abelian groups
\[
1\to \mathcal U_{S}\lbe(\lbe X\le)\be\oplus \mathcal U_{S}\lbe(\le Y\e)\to \mathcal U_{S}\lbe(\lbe X\be\times_{\be S}\be Y\le)\to \krn\pic\be f\le\cap\le \krn\pic\lbe g\to 1.
\]
Since either $\krn\pic\be f$ or $\krn\pic\lbe g$ is zero by hypothesis, the proposition follows.
\end{proof}

The next statement alludes to the separable indices $I(X_{\eta}),I(Y_{\eta})$ and the \'etale index $I(\e f\!\times_{S}\!g)$ defined at the end of Subsection \ref{gen}.

\begin{proposition}\label{upadd} Let $S$ be a locally noetherian normal scheme and let $f\colon X\to S$ and $g\colon Y\to S$ be faithfully flat morphisms locally of finite type. Assume that the following conditions hold:
\begin{enumerate}
\item[(i)] $f\be\times_{\be S}\be g$ has an \'etale quasi-section,
\item[(ii)] for every \'etale and surjective morphism $T\to S$, $X_{T}$, $Y_{T}$ and $X_{T}\be\times_{T}\be Y_{T}$ are locally factorial,
\item[(iii)] for every point $s\in S$ of codimension $\leq 1$, the fibers $X_{\lbe s}$ and $Y_{\lbe s}$ are geometrically integral, and
\item[(iv)] for every maximal point $\eta$ of $S$, ${\rm gcd}\le(\le I(X_{\eta}\le),I(\le Y_{\eta}\le))=1$ and
\begin{equation}\label{kid}
\pic\be(X_{\lbe\eta}^{\le\rm s}\be\times_{k(\eta)^{\rm s}}\be Y_{\!\eta}^{\e\rm s}\le)^{\g(\eta)}=(\e\pic X_{\lbe\eta}^{\le\rm s})^{\g(\eta)}\be\oplus\be(\e\pic Y_{\!\eta}^{\e\rm s})^{\g(\eta)},
\end{equation}
where $\g(\eta)={\rm Gal}(k(\eta)^{\rm s}\be/k(\eta))$.
\end{enumerate}
Then
\begin{enumerate}
\item[(a)] the canonical map \eqref{c1}
\[
\npic\be(X\be/\be S\le)\oplus\npic\be(Y\be/\be S\le)\to\npic\be(X\!\times_{\be S}\be Y\be/\be S\le)
\]
is an isomorphism of abelian groups,
\item[(b)] the canonical morphism \eqref{upicm}
\[
\upicxs\oplus\upicys\to\upicxys
\]
is an isomorphism in $\dbs$, and
\item[(c)] if, in addition, either
\begin{enumerate}
\item[(v)] $H^{\le 3}\lbe(S_{\et},\bg_{m,S})=0$, or
\item[($\text{v}^{\e\prime}\e$)] the \'etale index $I(\le f\!\times_{\be S}\be g\le)$ is defined and is equal to $1$, or
\item[($\text{v}^{\e\prime\prime}\e$)] $g$ has a section,
\end{enumerate}
then the canonical map \eqref{c3}
\[
\bra\be\oplus\be\bray\to\braxy
\]
is an isomorphism of abelian groups.
\end{enumerate}
\end{proposition}
\begin{proof} (Cf. \cite[proof of Lemma 6.6]{san}) For (a) and (b), see \cite[Propositions 2.7 and 4.4, respectively]{ga5}. By \cite[Corollary 4.5]{ga5}, is suffices to check that \eqref{c3} is an isomorphism of abelian groups when ($\text{v}^{\e\prime\prime}\e$) holds, i.e., $g$ has a section $\tau\colon S\to Y$. In this case we will show that the map $\zeta_{\le X,\le Y}^{\le 0}$ \eqref{c4} is an isomorphism, which will show that\eqref{c3} is an isomorphism by Lemma \ref{brai}.

By \cite[Proposition 3.6(iv)]{ga5}, there exists a canonical exact sequence of abelian groups
\begin{equation}\label{need1}
\picys(S\e)\to H^{2}(S_{\et},\uys)\to \bray\to H^{1}(S_{\et},\picys)\to H^{3}(S_{\et},\uys).
\end{equation}
Further, if $(h,Z)=(f,X)$ or $(f\times_{S}g,X\times_{S}Y\e)$, then the Cartan-Leray spectral sequence associated to $h\colon Z\to S$
\[
H^{\le r}\be(S_{\et}, R^{\e s} h_{\lbe *}\bg_{m,Z})\Rightarrow H^{\le r+s}\lbe(Z_{\et}, \bg_{m,Z})
\]
induces an exact sequence of abelian groups \cite[p.~309, line 8]{mi1}
\begin{equation}\label{need2}
{\rm Pic}_{Z/S}(S\le)\to
H^{\le 2}\lbe(S_{\et},h_{ *}\bg_{m,Z})\to\br_{\!1}^{\prime}(Z/S\e)\to H^{\le 1}\lbe(S_{\et},{\rm Pic}_{Z/S})\to H^{\le 3}\lbe(S_{\et},h_{*}\bg_{m,Z}).
\end{equation}
Next, by \cite[Corollary 4.4]{ga4}, the given section $\tau$ of $g$ induces an isomorphism of \'etale sheaves on $S\e$
\begin{equation}\label{need3}
f_{\lle *}\le\bg_{m,\le X}\oplus\uys\overset{\!\sim}{\to}
(f\times_{\be S} g)_{*}\bg_{m,\e X\times_{\be S} Y}.
\end{equation}
We now use the sequences \eqref{need1} and \eqref{need2} to form the following $5$-column diagram of abelian groups with exact rows
\[
\xymatrix{\picxs(S\e)\oplus \picys(S\e)\ar[d]_{\alpha}^{\simeq}\ar[r]& H^{\le 2}\lbe(S_{\et},f_{\lbe *}\bg_{m,X})\oplus H^{\le 2}\lbe(S_{\et},\uys)\ar[r]\ar[d]_{\beta}^{\simeq}&\bro\oplus\be\bray\ar[d]^{\zeta_{\lle X,\le Y}^{\le 0}}\\
\picxys(S\e)\ar[r]&H^{\le 2}\lbe(S_{\et},(f\!\times_{\be S}\! g)_{*}\bg_{m,\e X\times_{\be S} Y})\ar[r]& \broxys
}
\]	
\[
\xymatrix{\ar[r]& H^{\le 1}\lbe(S_{\et},\picxs)\oplus H^{\le 1}\lbe(S_{\et},\picys)\ar[d]_{\gamma}^{\simeq}\ar[r]& H^{\le 3}\lbe(S_{\et},f_{\lbe *}\bg_{m,X})\oplus H^{\le 3}\lbe(S_{\et},\uys)\ar[d]_{\delta}^{\simeq}\\
\ar[r]&	H^{\le 1}\lbe(S_{\et},\picxys)\ar[r]& H^{\le 3}\lbe(S_{\et},(f\!\times_{\be S}\! g)_{*}\bg_{m,\e X\times_{\be S} Y}).
}
\]
The maps $\alpha$ and $\gamma$ are isomorphisms since, by (b), $\picxs\oplus\picys\to\picxys$, i.e., $H^{\le 0}\lbe(\e\upicxs\le\oplus\le\upicys)\to H^{\le 0}\lbe(\le\upicxys)$, is an isomorphism of \'etale sheaves on $S$. The maps $\beta$ and $\delta$ are induced by the isomorphism of \'etale sheaves \eqref{need3}. It follows from the definition of $\zeta_{\le X,\le Y}^{\le 0}$ \eqref{c4} and the proof of \cite[Corollary 4.4]{ga4} that the above diagram commutes. Thus the five lemma applied to the diagram shows that $\zeta_{\le X,\le Y}^{\le 0}$ is an isomorphism, as claimed.
\end{proof}

\begin{remark}\label{comm}\indent
\begin{enumerate}
\item[(a)] The proof shows that if hypotheses (i)-(iv) and ($\text{v}^{\e\prime\prime}\e$) of the proposition hold, then the fact that 
$\bra\be\oplus\be\bray\to\braxy$ \eqref{c3} is an isomorphism is a consequence of the fact that $\zeta_{\le X,\le Y}^{\le 0}$ \eqref{c4} (or, equivalently, $\zeta_{\le X,\le Y}$ \eqref{c0}) is an isomorphism.
\item[(b)] If $S$ is locally noetherian and regular and $f\colon X\to S$ and $g\colon Y\to S$ are smooth and surjective, then hypothesis (ii) of the proposition holds. See \cite[Remark 4.10(a)]{ga5}. Further, since in this case $f\!\times_{S}\! g$ is also smooth and surjective, hypothesis (i) also holds by \cite[${\rm IV}_{4}$, Corollary 17.16.3(ii)]{ega}. The identity (or, more precisely, canonical isomorphism) \eqref{kid} holds in many cases of interest. See \cite[Examples 5.9 and 5.12]{ga4}. Finally, hypothesis (v) holds in the following cases (see \cite[Remark II.2.2(a), p.~165]{adt}):
\begin{itemize}
\item[(1)] $S$ is the spectrum of a global field.
\item[(2)] $S$ is a proper nonempty open subscheme of the spectrum of the ring of integers of a number field.
\item[(3)] $S$ is a nonempty open affine subscheme of a smooth, complete and irreducible curve over a finite field.	
\end{itemize}
\end{enumerate}
\end{remark}

\smallskip

Recall that, if $k^{\e\prime}\be/k$ is a field extension, a geometrically integral $k$-scheme $X$ is called {\it $k^{\e\prime}$-rational} if the function field of $X\be\times_{k}\be\spec k^{\e\prime}$ is a purely transcendental extension of $k^{\e\prime}$.

\begin{corollary}\label{xg} Let $S$ be a locally noetherian normal scheme and let $X$ and $G$ be faithfully flat $S$-schemes locally of finite type, where $G$ is a group. Let $F\colon ({\rm Sch}/\lbe S\e)\to\abe$ denote either $\npic_{\!S}$ \eqref{pnp} or $\bras$ \eqref{braf}. Assume that
\begin{enumerate}
\item[(i)] the structural morphism $X\to S$ has an \'etale quasi-section,
\item[(ii)] for every integer $i\geq 1$ and every \'etale and surjective morphism $T\to S$, $X_{T}$, $G^{\e i}_{\lbe T}$ and $X_{T}\be\times_{T}\be G^{\e i}_{T}$ are locally factorial, 
\item[(iii)] for every point $s\in S$ of codimension $\leq 1$, the fibers $X_{\lbe s}$ and $G_{\lbe s}$ are geometrically integral, and
\item[(iv)] for every maximal point $\eta$ of $S$, $G_{\lbe\eta}$ is  $k(\eta)^{\rm s}$-rational.
\end{enumerate}
Then the maps $F(\le G^{\e i}\le)\oplus F(\le G\e)\to F(\le G^{\e i+1}\le)$ and $F(\lbe X\le)\oplus F(\le G^{\e i}\le)\to F(\lbe X\be\times_{\be S} G^{\e i}\e)$ \eqref{c1}, \eqref{c3} are isomorphisms of abelian groups for every $i\geq 1$.
\end{corollary}
\begin{proof} By (ii) and (iii), $X_{\eta}$ is normal, geometrically integral and locally of finite type over $k(\eta)$. Further, by (iv), $G_{\lbe\eta}^{\e i}$ is $k(\eta)^{\rm s}$-rational for every integer $i\geq 1$. Thus, by \cite[Example 5.9(a)]{ga4}, $\pic\be(\e\xs_{\lbe\eta}\!\times_{k(\eta)^{\rm s}}\!(G_{\lbe\eta}^{\e i})^{\rm s}\le)=\pic\xs_{\lbe\eta}\be\oplus\be \pic((G_{\lbe\eta}^{\e i})^{\rm s})$, i.e., \eqref{kid} holds. Similarly, $\pic\be(\e(G_{\lbe\eta}^{\e i})^{\rm s}\be\times_{k(\eta)^{\rm s}}\be G_{\lbe\eta}^{\e\rm s}\le)=\pic((G_{\lbe\eta}^{\e i})^{\rm s})\be\oplus\be \pic(G_{\lbe\eta}^{\e\rm s})$. Therefore hypotheses (i)-(iv) and ($\text{v}^{\e\prime\prime}\e$) of the proposition hold for $X,G^{\e i}$ and $G^{\e i},G$. The corollary for $F=\npic_{\!S}$ (respectively, $F=\bras$) now follows from part (a) (respectively, (c)) of the proposition.
\end{proof}

\section{The units-Picard-Brauer sequence of a torsor}

Let $S$ be a scheme and let $G$ and $Y$ be flat $S$-schemes locally of finite presentation, where $G$ is a group. A basic problem is to compute the Picard group of a $G_{\le Y}$-torsor $X$ over $Y$ in terms of the Picard groups of $Y$ and $G$. When $S=\spec k$, where $k$ is a field, this problem was discussed by Sansuc in \cite[pp.~43-45]{san}, who used a simplicial method to obtain a units-Picard-Brauer exact sequence that relates the groups mentioned above. In this Section we generalize Sansuc's method to deduce, under appropriate conditions, a similar exact sequence over any locally noetherian normal scheme $S$.

\smallskip

\subsection{A simplicial lemma} Let $\s C$ be a full subcategory of $S_{\fl}$ which is stable under products and contains $1_{\lbe S}$. Further, let $F\colon \s C\to \abe$ be an abelian presheaf on $\s C\e$ such that $F(1_{\lbe S})=0$. If $X\to S$ is an object of $\s C$ (respectively, if $\xi\colon X\to Y$ is a morphism in $\s C\e$), we will write $F(X)$ (respectively, $F(\xi)$) for the abelian group $F(\le X\to S\le)$ (respectively, morphism of abelian groups $F(\le Y\le)\to F(X)$). The canonical $\s C$-morphisms  $p_{X}\colon X\times_{S}Y\to X$ and $p_{\e Y}\colon X\times_{S} Y\to Y$ define a morphism in $\abe$
\begin{equation}\label{fxy}
\psi_{X,\e Y}\colon F(X)\oplus F(Y\le)\to F(X\be\times_{S}\be Y\le), (a,b)\mapsto F(\e p_{\lbe X})(a)+F(\e p_{\e Y})(b).
\end{equation}
If $Z\to S$ is a third object of $\s C$ and we identify $(X\times_{S}Y\e)\times_{S} Z$ and $X\times_{S}(\e Y\times_{S}Z\e)$ (and write $X\times_{S}Y\times_{S}Z\e$ for the latter), then the following diagram commutes:
\begin{equation}\label{big}
\xymatrix{F(X)\oplus F(\le Y\le)\oplus F(Z\e)\ar[d]_{\psi_{X,Y}\oplus 1_{F(Z\le)}}
\ar[rr]^(.5){1_{F\lbe(\lbe X)}\le\oplus\e \psi_{\e Y,Z}}&& F(X\e)\oplus F(\e Y\be\times_{S}\be Z\e)\ar[d]^{\psi_{X,Y\lbe\times_{\be S}\lbe Z}}\\
F(X\be\times_{S}\be Y\e)\oplus F(Z\le)\ar[rr]^(.5){\psi_{X\lbe\times_{\be S}\lbe Y,Z}}&& F(X\be\times_{S}\be Y\be\times_{S}\be Z\e).
}
\end{equation}

Let $\pi\colon G\to S$ be a group object of $\s C$ with unit section $\varepsilon\colon S\to G$. We will need the following fact: since $F(S\e)=0$, both $F(\pi)$ and $F(\varepsilon)$ are the zero morphism of $\abe$. Now let $h\colon Y\to S$ be an object of $\s C$, write $\s C_{\be/Y}$ for the category of $Y$-objects of $\s C$ and let $\xi\colon X\to Y$ be a faithfully flat morphism locally of finite presentation. Thus $\xi$ is an object of $\s C_{\be/Y}$. Further, locally for the fppf topology on $Y$, $\xi$ has a section \cite[${\rm IV}_{4}$, Corollary 17.16.2]{ega}.

Assume now that the $Y$-scheme $X$ is equipped with a right action of $G_{Y}$ over $Y$
\begin{equation}\label{vsig}
\varsigma^{\le\prime}\colon X\be \times_{\lbe Y}\be G_{Y}\to X
\end{equation}
such that $X$ is a right (fppf\e) $G_{Y}$-torsor over $Y$, i.e., the $Y$-morphism
\[
(\e p_{\lbe X},\varsigma^{\le\prime}\e)_{\lbe Y}\colon X \times_{Y} (G\times_{S}Y\e)\to X\be \times_{Y}\be X, (x,(\e g,\xi(x)))\mapsto (x,xg),
\]
is an isomorphism. We will write 
\begin{equation}\label{smor}
\varsigma\colon X\be \times_{\lbe S}\le G\to X, (x,g)\mapsto xg,
\end{equation}
for the composite $S$-morphism
\[
X\times_{\lbe S} G\isoto X\!\times_{\lbe Y}\be G_{Y}\overset{\!\varsigma^{\lle\prime}}{\to}X,
\]
where the first map is the canonical $S$-isomorphism of \cite[Corollary 1.3.4, p.~32]{ega1} and the second map is the $Y$-morphism \eqref{vsig} regarded as an $S$-morphism. Note that the morphism of $S$-schemes
\begin{equation}\label{no}
\beta_{\le 1}=(\e p_{\lbe X},\varsigma\le)_{\lbe S}\colon X\times_{\lbe S} G\to X\!\times_{Y}\! X, (x,g)\mapsto (x,xg),
\end{equation}
is an isomorphism. Further, the induced $S$-morphism
\begin{equation}\label{si}
\vartheta=p_{\le G}\circ \beta_{\e 1}^{\e -1}\colon X\times_{Y}X\to G
\end{equation}
is defined as follows: for every $S$-scheme $T$, $\vartheta(\le T\le)$ maps $(x_{0},x_{1})\in X(\le T\le)\be\times_{Y\lbe(\le T\le)}\be X(\le T\le)$ to the unique element $g\in G(\le T\le)$ such that $x_{1}=x_{0}\e g$.

\begin{example} \label{2ex} Let $1\to H\overset{\!i}{\to} G\to F\to 1$ be an exact sequence of groups of $\s C$. Then $G\to S$ factors as $G\to F\to S$, where both $G\to F$ and $F\to S$ are faithfully flat and locally of finite presentation \cite[${\rm VI_{B}}$, Proposition 9.2, (xi) and (xii)]{sga3}. Further, $G$ is an $H_{F}\e$-\e torsor over $F$ with action $\varsigma\colon G\be \times_{\lbe S}\le H\to G, (g,h)\mapsto g\e i(h)$. See Subsection \ref{tsec}.
\end{example}

For every integer $i\geq 0$ and every $S$-scheme $X$, we define $X^{\le i}$ recursively by $X^{\le i+1}=X\be\times_{\be S}\be X^{\le i}$, where $X^{\e 0}=S$. If $X\to S$ factors as $X\to Y\overset{\!h}{\to} S$, we define $X^{\le i}_{\!/Y}$ similarly, i.e., 
$X^{\le 0}_{\!/Y}=Y$ and $X^{\le i+1}_{\!/Y}=X\be\times_{Y}X^{\le i}_{\!/Y}$. It is well-known \cite[\S10.2, p.~97]{may} that $\{X_{\!/Y}^{\le i+1}\}_{\e i\e\geq\e 0}$ is a simplicial $Y$-scheme with degeneracy maps
\begin{equation}\label{deg1}
X_{\!/Y}^{\le i+1}\to X_{\!/Y}^{\le i+2}, (x_{0},\dots,x_{i})\mapsto (x_{0},\dots,x_{j},x_{j},\dots, x_{i}), \quad 0\leq j\leq i,
\end{equation}
and face maps\,\footnote{\e In several formulas below, ordered tuples of the form $y_{\le k},\dots,y_{\le r}$ with $k>r$ must be omitted for the formulas to make sense.}
\begin{equation}\label{fa1}
\partial^{\,j}_{i+1}\colon X_{\!/Y}^{\le i+2}\to X_{\!/Y}^{\le i+1}, (x_{0},\dots,x_{i+1})\mapsto (x_{0},\dots,x_{j-1},x_{j+1},\dots, x_{i+1}),\quad 0\leq j\leq i+1.
\end{equation}

Note that
\begin{equation}\label{b1}
\partial^{\e 0}_{\le 1}\circ\beta_{1}=\varsigma
\end{equation}
and
\begin{equation}\label{b2}
\partial^{\e1}_{\le 1}\circ\beta_{1}=p_{X},
\end{equation}
where $\varsigma$ and $\beta_{\le 1}$ are given by \eqref{smor} and \eqref{no}, respectively.

\smallskip

Now let $h^{\lbe*}\be F\colon \s C_{\be/Y}\to \abe$ be the restriction of $F$ to $\s C_{\be/Y}$, i.e.,
\[
(\e h^{\lbe*}\be F\e)(Z\to Y\e)=F(Z\to Y\overset{\!h}{\to} S\e).
\]
Further, let $\check{H}^{i}(X/Y,h^{\lbe*}\be F\e)$ be the \v{C}ech cohomology groups of the abelian presheaf $h^{\lbe*}\be F$ on $\s C_{\be/Y}$ with respect to the fppf covering $\xi\colon X\to Y$. The groups $\check{H}^{i}(X/Y,h^{\lbe*}\be F\e)$ are the cohomology groups of the complex of abelian groups $\big\{F\lbe\big(X_{\!/Y}^{\le i+1}\big), \partial^{\, i+1}\big\}_{\e i\e\geq\e 0}$, where $\partial^{\, i+1}=\sum_{\e j=0}^{\e i+1}\e (-1)^{\le j}\lbe F\lbe\big(\partial_{\le i+1}^{\,j}\lbe\big)\colon F\lbe\big(X_{\!/Y}^{\le i+1})\to F\lbe\big(X_{\!/Y}^{\le i+2})\,$\footnote{\e Here we regard $X_{\!/Y}^{\le i+1}$ and $\partial_{\le i+1}^{\,j}$ as $S$-schemes and $S$-morphisms, respectively}\,.

\smallskip

The next statement was obtained by Sansuc \cite[Lemma 6.12]{san} in the case $S=\spec k$, where $k$ is a field. Unfortunately, his proof omits many details and we cannot claim that it extends without difficulty to the case of an arbitrary base scheme $S$. Therefore we include below a detailed proof of the general case.

\begin{lemma}\label{gsan} Let $F\colon \s C\to \abe$ with $F(1_{\lbe S})=0$, $G\to S$, $h\colon Y\to S$ and $\xi\colon X\to Y$ be as above. Assume that, for every integer $i\geq 1$, the maps $\psi_{\le G^{\le\lle i}\lbe,\e G}\colon F(G^{\e i}\le)\oplus F(G\e)\to F(G^{\e i+1}\le)$ and $\psi_{X,\e G^{\e i}}\colon F(X)\oplus F(G^{\e i})\to F(X\be \times_{\lbe S} G^{\e i}\e)$ \eqref{fxy} are isomorphisms of abelian groups. Then $\check{H}^{i}(X/Y,h^{\lbe*}\be F\le)=0$ for every $i\geq 2$ and there exists a canonical exact sequence of abelian groups
\begin{equation}\label{vphi}
0\to \check{H}^{\le 0}(X/Y,h^{\lbe*}\be F\e)\to F(X\le)\overset{\!\varphi}{\to}F(G\e)\to\check{H}^{\le 1}(X/Y,h^{\lbe*}\be F\e)\to 0,
\end{equation}
where the map $\varphi$ is the composition
\begin{equation}\label{vphi2}
\varphi\colon F(X)\overset{\!\!\be F(\varsigma)}{\lra}F(X\be \times_{\lbe S}\le G)\overset{\!\psi_{X,\le G}^{-1}}{\to}F(X)\be\oplus\be F(G\le)\overset{\be p_{\lbe F\lbe(G\lle)}}{\lra}F(G\le).
\end{equation}
Above, $\varsigma$ is the $S$-morphism \eqref{smor} and $p_{\lbe F\lbe(G\lle)}$ is the canonical projection.
\end{lemma}
\begin{remark} The proof of the lemma will show that $F(\vartheta)\colon F(G\le)\to F(X\!\times_{\lbe Y}\! X\le)$ \eqref{si} factors as
\[
F(G\le)\overset{\!F\lbe(\lbe\vartheta)^{\prime}}{\lra} 
\krn\partial^{\e 2}\subseteq F(X\!\be\times_{\lbe Y}\be\! X\le)
\]
and the maps $\check{H}^{\le 0}(X/Y,h^{\lbe*}\be F\e)\to F(X\le)$ and $F(G\e)\to\check{H}^{\le 1}(X/Y,h^{\lbe*}\be F\e)$ in \eqref{vphi} are, respectively, the inclusion and the composition
\[
F(G\le)\overset{\!F\lbe(\lbe\vartheta)^{\prime}}{\lra}\krn\partial^{\e 2}\twoheadrightarrow \krn\partial^{\e 2}\be/\e \img\partial^{\e 1}=\check{H}^{\le 1}(X/Y,h^{\lbe*}\be F\e).
\]
\end{remark}
\begin{proof} It is well-known \cite[\S10]{may} that $\{X\!\times_{\be S}\be G^{\e i}\e\}_{\e i\e\geq\e 0}$ is canonically endowed with a simplicial $S$-scheme structure whose degeneracy and face maps are defined as follows. For every integer $j$ such that $0\leq j\leq i$, define $v_{\lbe i}^{\e j}=(\e 1_{\le G^{\le j}},\varepsilon,1_{\le G^{\e i-j}})_{S}\colon G^{\e i}\to G^{\e i+1}$, where $\varepsilon\colon S\to G$ is the unit section of $G$, i.e., 
\begin{equation}\label{vij}
v_{\lbe i}^{\e j}(g_{\le 1},\dots, g_{\le i})=(\e g_{\le 1},\dots, g_{\le j}, 1, g_{\le j+1},\dots, g_{\le i}).
\end{equation}
Then the degeneracy maps of $\{X\be\times_{\be S}\lbe G^{\e i}\e\}_{\e i\e\geq\e 0}$ are the maps $s^{\e j}_{\lbe i}=(1_{\lbe X},v_{\lbe i}^{\e j})_{S}\colon X\be\times_{\be S} G^{\e i}\to X\be\times_{\be S}\lbe G^{\e i+1}$ given by
\begin{equation}\label{deg2}
s_{i}^{\e j}(x,g_{1},\dots,g_{i})=(x, g_{\le 1},\dots, g_{\le j}, 1, g_{\le j+1},\dots, g_{\le i}).
\end{equation}
Next, for $i\geq 1$ and $1\leq j\leq i+1$, define
\begin{equation}\label{wij}
w_{i+1}^{\e j}\colon G^{\e i+1}\to G^{\e i}, (\e g_{\le 1},\dots,g_{\le i+1})\mapsto (\e g_{\le 1},\dots,g_{\le j-1},g_{\le j}g_{\le j+1},g_{\le j+2}\dots, g_{\le i+1}).
\end{equation}
We extend the above definition to the pair $(i,j)=(0,1)$ by setting
\begin{equation}\label{w11}
w_{1}^{\lle 1}=\pi\colon G\to S,
\end{equation}
where $\pi$ is the structural morphism of $G\e$. Then the face maps of $\{X\be\times_{\be S}\lbe G^{\e i}\e\}_{\e i\e\geq\e 0}$ are given by $d_{\le i+1}^{\,j}=(1_{\lbe X},w_{i+1}^{\e j}\e)_{S}\colon X\be\times_{\be S}\lbe G^{\e i+1}\to X\be\times_{\be S}\lbe G^{\e i}$ if $1\leq j\leq i+1$ and
\begin{equation}\label{fa2}
d_{\e i+1}^{\,\le 0}(x,g_{1},\dots,g_{\le i+1})=(xg_{1},g_{\le 2},\dots,g_{\le i+1}).
\end{equation}
If $i\geq 1$ and $1\leq j\leq i+1$, then
\begin{equation}\label{fa3}
d_{\le i+1}^{\,\le j}(x,g_{\le 1},\dots,g_{\le i+1})=(x, g_{\le 1},\dots,g_{\le j-1},g_{\le j}g_{\le j+1},g_{\le j+2}\dots, g_{\le i+1}).
\end{equation}
Now, for every integer $i\geq 1$, \eqref{no} induces (by iteration) an isomorphism of $S$-schemes
\[
\beta_{\le i}\colon X\!\times_{\lbe S}\be G^{\e i}\isoto  X_{\!/Y}^{\le i+1}, (x,g_{1},\dots,g_{i})\mapsto (x,xg_{1},\dots, xg_{1}\cdots g_{i}).
\]
In conjunction with $\beta_{\e 0}=1_{\lbe X}$, the preceding maps define an isomorphism of simplicial $S$-schemes
$\{\beta_{\le i}\}_{\e i\e\geq\e 0}\colon \{X\!\times_{S}\le G^{\e i}\e\}_{\e i\e\geq\e 0}\isoto \{X_{\!/Y}^{\le i+1}\}_{\e i\e\geq\e 0}$, i.e., the maps $\beta_{\le i}$ commute with the operators \eqref{deg1}, \eqref{fa1}, \eqref{deg2}, \eqref{fa2} and \eqref{fa3}, in the sense of \cite[Definition 9.1, p.~83]{may}. 
The above map induces an isomorphism of cosimplicial abelian groups 
\begin{equation}\label{iso}
\{\le F\lbe(\le \beta_{\le i}^{\lle -1})\}_{\e i\e\geq\e 0}\colon \big\{F(X\!\times_{\be S} G^{\e i}\le)\big\}_{\e i\e\geq\e 0}\!\isoto\big\{F\be\le\big(X_{\!/Y}^{\le i+1}\big)\big\}_{\e i\e\geq\e 0}.
\end{equation}
We will now define a cosimplicial abelian group structure on $\big\{F(X)\lbe\oplus\lbe F(G^{\, i})\e\big\}_{\e i\e\geq\e 0}$ so that the map
\begin{equation}\label{soi}
\{\le \psi_{X,\e G^{\le i}}\}_{\e i\e\geq\e 0}\colon \big\{F(X)\!\oplus\! F(G^{\, i})\e\big\}_{\e i\e\geq\e 0}\isoto  \big\{F(X\!\times_{\be S} G^{\e i}\le)\big\}_{\e i\e\geq\e 0}
\end{equation}
is an isomorphism of cosimplicial abelian groups.

The codegeneracy maps are
\[
1_{F(X)}\!\oplus\! F(v_{\lbe i}^{\, j})\colon F(X)\!\oplus\! F(G^{\, i+1})\to F(X)\!\oplus\! F(G^{\, i}),
\]
where $0\leq j\leq i$ and the maps $v_{\lbe i}^{\e j}\colon G^{\e i}\to G^{\e i+1}$ are given by \eqref{vij}. The coface maps are 
\[
(\e p_{\lbe F(X)},\alpha_{j}^{\le i+1})\colon F(X)\oplus F(G^{\, i})\to F(X)\be\oplus\be F(G^{\, i+1}), (a,b)\mapsto (a,\alpha_{j}^{\e i+1}(a,b)),
\]
where $\e 0\leq j\leq i+1\e$, $p_{ F\lbe(X)}\colon F(X)\oplus F(G^{\, i})\to F(X)$ is the canonical projection and the maps $\alpha_{j}^{\le i+1}\colon F(X)\oplus F(G^{\, i})\to F(G^{\, i+1})$ are defined by
\begin{equation}\label{boi2}
\alpha_{j}^{\le i+1}(a,b)=\begin{cases}
F(\e p_{\le G}^{\le 1})(\varphi(a))+F(\e p_{\le G^{\le i}})(b\le)\quad \text{if $j=0$}\\
F(w_{i+1}^{\e j})(b)\hskip 3.06cm \text{if $1\leq j\leq i+1$},
\end{cases}
\end{equation}
where $a\in F(X),b\in F(G^{\, i})$, $p_{\le G}^{\le 1}\colon G^{\e i+1}\to G$ is the first projection, $\varphi$ is the map \eqref{vphi2}, $p_{\le G^{\le i}}\colon G^{\e i+1}\to G^{\e i}$ is the projection onto the last $i$ factors (when $i\geq 1$) and the maps $w_{i+1}^{\e j}\colon G^{\e i+1}\to G^{\e i}$ are given by \eqref{wij} and \eqref{w11}. To check that \eqref{soi} is, indeed, an  isomorphism of cosimplicial abelian groups, we need to check that the following diagrams commute:
\begin{equation}\label{d1}
\xymatrix{F(X)\oplus F(G^{\e i}\le)\ar[d]_{(\e p_{\lbe F(\lbe X)},\e\alpha_{j}^{\le i+1})}\ar[rr]^(.5){\psi_{X,\le G^{\e i}}}&& F(X\times_{S}G^{\e i}\e)\ar[d]^{F(\le d_{i+1}^{\,j}\le)}\\
F(X)\oplus F(G^{\e i+1}\le)\ar[rr]^(.5){\psi_{X,\le G^{\e i+1}}}&& F(X\times_{S}G^{\e i+1}\e),
}
\end{equation}
where $0\leq j\leq i+1$, and
\[
\xymatrix{F(X)\oplus F(G^{\e i+1}\le)\ar[d]_{1_{F(X)}\e\oplus\, F(v_{\lbe i}^{\e j})}\ar[rr]^(.5){\psi_{X,\le G^{\le i+1}}}&& F(X\times_{S}G^{\e i+1}\e)\ar[d]^{F(\le s_{i}^{\e j}\le)}\\
F(X)\oplus F(G^{\e i}\le)\ar[rr]^(.5){\psi_{X,\le G^{\le i}}}&& F(X\times_{S}G^{\e i}\e),
}
\]
where $0\leq j\leq i$. Except when $j=0$ in \eqref{d1}, the commutativity of the preceding diagrams follows without difficulty from the definitions \eqref{fxy}, \eqref{vij}-\eqref{w11}, \eqref{fa3} and \eqref{boi2}. The commutativity of \eqref{d1} when $j=0$ follows from the definitions \eqref{fa2}, \eqref{boi2}, the commutativity of diagram \eqref{big} for $(X,Y,Z)=(X,G,G^{\e i})$ and the following equality of maps $F(X)\to F(X\times_{S}G\e)$:
\begin{equation}\label{ana0}
\psi_{X,\e G}\circ (\e i_{\le F(X)}+i_{\le F(G\le)}\circ\varphi\e)=F(\varsigma),
\end{equation}
where $i_{\le F(X)}\colon F(X)\to F(X)\oplus F(G\le)$ and $i_{\le F(G\le)}\colon F(G\le)\to F(X)\oplus F(G\le)$ are the canonical embeddings. Given the definition of $\varphi$, \eqref{ana0} follows from the equality 
\[
p_{F(X)}\circ \psi_{X,\e G}^{-1}\circ F(\varsigma)=1_{F(X)},
\]
which is formula \cite[(25), p.~480]{ga4}.

\smallskip

Thus we have checked that \eqref{soi} is an  isomorphism of cosimplicial abelian groups.

\smallskip

Next, composing \eqref{iso} and \eqref{soi}, we obtain an isomorphism of cosimplicial abelian groups 
\begin{equation}\label{isos}
\{\le F\lbe(\le \beta_{\le i}^{\e -1})\circ\psi_{X,\e G^{\le i}} \}_{\e i\e\geq\e 0}\colon \big\{F(X)\!\oplus\! F(G^{\, i})\e\big\}_{\e i\e\geq\e 0}\isoto\big\{F\le\be\big(X_{\!/Y}^{\le i+1})\}_{\e i\e\geq\e 0}.
\end{equation}
We will now define a cosimplicial abelian group structure on $\big\{F(X)\!\oplus\! F(G\e)^{\le i}\le\big\}_{\e i\e\geq\e 0}$ and an isomorphism 
of cosimplicial abelian groups
\begin{equation}\label{som}
\{\le 1_{F\lbe(\lbe X)}\be\oplus\be\gamma_{G}^{(i)}\}_{\e i\e\geq\e 0}\colon \big\{F(X)\be\oplus\be F(G\e)^{\le i}\e\big\}_{\e i\e\geq\e 0}\!\isoto\big\{F(X)\be\oplus\be F(G^{\e i}\le)\e\big\}_{\e i\e\geq\e 0}.
\end{equation}
For every integer $i\geq 0$, let $\gamma_{G}^{(i)}\colon F(G\e)^{i}\to F(G^{\e i}\e)$ be defined (recursively) by $\gamma_{G}^{(0)}=0$, $\gamma_{G}^{(1)}=1_{\lbe F(G\le)}$ and $\gamma_{G}^{(i+1)}(b_{\le 1},\dots,b_{\e i+1})=\psi_{\e G^{\le i}\lbe,\e G}\lle(\gamma_{G}^{(i)}(b_{\le 1},\dots,b_{\le i}),b_{\le i+1})$ for $i\geq 1$. Since the maps $\psi_{\e G^{\le i}\lbe,\e G}$ are isomorphisms of abelian groups for every $i\geq 1$, so also is $\gamma_{G}^{(i)}$. Further, \eqref{fxy} yields the formula
\begin{equation}\label{for}
\gamma_{G}^{(i)}(b_{\le 1},\dots,b_{\le i})=\sum_{k\e=\e 1}^{i}F(\e p_{G}^{\e k})(b_{\le k}),
\end{equation}
where $p_{G}^{\e k}\colon G^{\e i}\to G$ is the $k$-th projection morphism. We now observe that the maps $\gamma_{G}^{(i)}$ satisfy the following inversion formula \eqref{inv}\,\footnote{\e This inversion formula is the key to obtaining without difficulty the formulas \eqref{cij} below.}\,: for every integer $k$ such that $1\leq k\leq i$, let $\theta^{\e k}\colon G\to G^{\e i}, g\mapsto (1,\dots,g_{(k)},\dots,1),$ where the subscript indicates position, i.e., $\theta^{\e k}$ is the unique $S$-morphism such that $p_{G}^{\e k}\lbe\circ\lbe \theta^{\e k}=1_{G}$ and $p_{G}^{\e j}\lbe\circ\lbe \theta^{\, k}=\varepsilon\be\circ\be \pi$ for $j\neq k$. Then, for every $k$ such that $1\leq k\leq i$, we have
\begin{equation}\label{inv}
b_{\le k}=F(\theta^{\e k}\le)\be\circ\be \gamma_{G}^{(i)}(b_{\le 1},\dots,b_{\le i}).
\end{equation}
Next, for every $i\geq 0$, we define maps $\delta_{\lbe j}^{\e i+1}\colon F(X)\oplus F(G\e)^{\le i}\to F(X)\oplus F(G\e)^{\le i+1}$ (where $\e 0\leq j\leq i+1\e$) and $\sigma_{\be j}^{\e i}\colon F(G\e)^{\le i+1}\to F(G\e)^{\le i}$ (where $\e 0\leq j\leq i\e$) by the commutativity of the diagrams
\begin{equation}\label{d3}
\xymatrix{F(X)\be\oplus\be F(G\e)^{\e i}\ar@{..>}[d]_(.45){\delta_{\lbe j}^{\e i+1}}\ar[rr]^(.5){1_{F\lbe(\lbe X)}\oplus\e\gamma_{G}^{(i)}}_{\sim}&& F(X)\be\oplus\be F(G^{\e i}\le)\ar[d]^(.45){(\e p_{\lbe F(X)},\e\alpha_{j}^{\le i+1})}\\
F(X)\be\oplus\be F(G\e)^{\e i+1}\ar[rr]^(.5){1_{F(X)}\oplus\e\gamma_{G}^{(i+1)}}_{\sim}&& F(X)\be\oplus\be F(G^{\e i+1}\le),
}
\end{equation}
where $0\leq j\leq i+1$, and
\[
\xymatrix{F(G\e)^{\le i+1}\ar@{..>}[d]_(.45){\sigma_{\be j}^{\le i}}\ar[rr]^(.5){\gamma_{G}^{(i+1)}}_(.47){\sim}&&  F(G^{\e i+1}\le)\ar[d]^{F(v_{\lbe i}^{\, j})}\\
F(G\e)^{\le i}\ar[rr]^(.48){\gamma_{G}^{(i)}}_(.47){\sim}&& F(G^{\e i}\e),
}
\]
where $0\leq j\leq i$. Then $\big\{F(X)\!\oplus\! F(G\e)^{\e i}\e\big\}_{\e i\e\geq\e 0}$ is a cosimplicial abelian group with codegeneracy maps $1_{\lbe F\lbe(X)}\be\oplus\be\sigma_{\! j}^{\le i}$ and coface maps $\delta_{\lbe j}^{\e i+1}$. Further, \eqref{som} is an isomorphism of cosimplicial abelian groups. Composing \eqref{isos} and \eqref{som}, we obtain an isomorphism of cosimplicial abelian groups
\[
\big\{\le F\lbe(\le \beta_{\le i}^{\lle -1})\circ\psi_{X,\e G^{\le i}}\circ\big( 1_{F\lbe(\lbe X)}\be\oplus\lbe\gamma_{G}^{(i)}\le\big)\big\}_{\e i\e\geq\e 0}\colon\big\{F(X)\lbe\oplus\lbe F(G\e)^{\le i}\e\big\}_{\e i\e\geq\e 0}\!\isoto\big\{F\le\be\big(X_{\!/Y}^{\le i+1}\big)\big\}_{\e i\e\geq\e 0}.
\]
The latter map induces isomorphisms of the associated cochain complexes of abelian groups
\begin{equation}\label{dura}
\big\{F(X)\be\oplus\be F(G\e)^{\le i}, \delta^{\e i}\big\}_{\e i\e\geq\e 0}\isoto \big\{F\le\be\big(X_{\!/Y}^{\le i+1}\big), \partial^{\, i+1}\big\}_{\e i\e\geq\e 0}
\end{equation}
and the corresponding cohomology groups
\begin{equation}\label{Fmap0}
H^{\le i}(F(X)\be\oplus\be F(G\e)^{\bullet})\isoto\check{H}^{\le i}(X/Y,h^{\lbe *} \be F\e)\qquad(\e i\geq 0\e),
\end{equation}
where $\delta^{\e i}=\sum_{\e j=0}^{\e i+1}\e (-1)^{\le j}\lbe \delta_{\lbe j}^{\e i+1}\colon F(X)\be\oplus\be F(G\e)^{\le i}\to F(X)\be\oplus\be F(G\e)^{\le i+1}$.

Note that, since $\beta_{\e 0}=1_{X}, \psi_{X,G^{\le\lle 0}}=1_{F(X)}$ and $\gamma_{G}^{(0)}=0$, the map \eqref{Fmap0} for $i=0$ is an {\it equality}
$\check{H}^{0}(X/Y,h^{\lbe*}\be F\e)=\krn\delta^{\e 0}$.

We will now compute the maps $\delta^{\e i}$ explicitly. 

\smallskip

Using definitions \eqref{wij}, \eqref{w11}, \eqref{boi2}, formulas \eqref{for}, \eqref{inv} and the commutativity of diagram \eqref{d3}, the maps $\delta_{\lbe j}^{\e i+1}$ (where $\e 0\leq j\leq i+1\e$) are given by $\delta_{\lbe j}^{\e i+1}(a,b_{\le 1},\dots, b_{\le i})=(a, c_{\le 1}^{(i,\e j)},\dots,c_{\le i+1}^{(i,\e j)})$, where the elements $c_{\le k}^{(i,\e j)}\in F(G\e)$ for $1\leq k\leq i+1$ are defined by
\begin{equation}\label{cij}
c_{\le k}^{(i,\e j)}=\begin{cases} \varphi(a)\quad\text{if $k=1$ and $j=0$}\\
b_{\le k-1}\hskip .5cm\text{if $2\leq k\leq i+1$ and $0\leq j\leq k-1$}\\
b_{\le k}\hskip .9cm\text{if $1\leq k\leq i$ and $k\leq j\leq i+1$}\\
0\hskip 1.05cm\text{if $(i,j)=(0,1)$ or $k=j=i+1\geq 2$}.
\end{cases}
\end{equation}
The preceding formulas yield\,\e\footnote{\e These formulas generalize those stated (without proof) in \cite[p.~45]{san}.}:
\[
\begin{array}{rcl}
\delta^{\e 0}\lbe(a)&=&(0,\varphi(a))\\
\delta^{\e 1}\lbe(a,b_{\le 1})&=&(a,\varphi(a),0)\\
\delta^{\e 2r}\be(a,b_{\le 1},\dots,b_{\le 2\le r})&=&(0,\varphi(a)\!-\be b_{\le 1},0,b_{\le 2}\!-\! b_{\le 3},0,\dots,b_{\e 2\le r-2}-b_{\e 2\le r-1},0,b_{\e 2\le r})\\
\delta^{\e 2\le r+1}\be(a,b_{\e 1},\dots,b_{\e 2\le r+1})&=&(a,\varphi(a),b_{\e 2}, b_{\e 2},\dots,b_{\e 2\le r},b_{\e 2\le r},0),
\end{array}
\]
where $r\geq 1$.

It is now clear that $\check{H}^{0}(X/Y,h^{\lbe *}\be F\e)=\krn\delta^{\e 0}=\krn\varphi$. Next we observe that the first component of the map \eqref{dura}, i.e., $F\be\lle(\le\beta_{\le 1}^{\le-1})\circ\psi_{X,\le G}\colon F(X)\oplus F(G\le)\isoto F(X\be\times_{\lbe Y}\be X\le)$, induces (via restriction of domain) a map $\krn \delta^{\e 1}\to \krn\partial^{\e 2}$. The composition of the latter map and the canonical isomorphism $F(G\le)\isoto\krn\delta^{\e 1}\subseteq F(X)\oplus F(G\le), b\mapsto(0,b)$, is a map $F\lbe(\lbe\vartheta)^{\prime}\colon F(G\le)\to \krn\partial^{\e 2}$ such that the composition
\[
F(G\le)\overset{\!F\lbe(\lbe\vartheta)^{\prime}}{\lra}\krn\partial^{\e 2}\hookrightarrow F(X\!\times_{\lbe Y}\!\be X\lle)
\]
is the map $F(\vartheta)\colon F(G\le)\to F(X\!\times_{\lbe Y}\! X\le)$ \eqref{si}. Further, $F\lbe(\lbe\vartheta)^{\prime}$ induces an isomorphism of abelian groups $F(G\le)\be/\e \img\varphi\isoto \check{H}^{\le 1}(X/Y,h^{\lbe *}\be F\e)$, namely the composition 
\[
F(G\le)\be/\e \img\varphi\isoto\krn \delta^{\e 1}\!/\e \img\delta^{\e 0}\isoto  \krn\partial^{\e 2}\be/\e \img\partial^{\e 1}=\check{H}^{\le 1}(X/Y,h^{\lbe *}\be F\e),
\]
where the second isomorphism is the map \eqref{Fmap0} for $i=1$. Thus we can define the map $F(G\le)\twoheadrightarrow \check{H}^{\le 1}(X/Y,h^{\lbe *}\be F\e)$ appearing in \eqref{vphi} as the composition 
\[
F(G\le)\overset{\!F\lbe(\lbe\vartheta)^{\prime}}{\lra}\krn\partial^{\e 2}\twoheadrightarrow \krn\partial^{\e 2}\be/\e \img\partial^{\e 1}=\check{H}^{\le 1}(X/Y,h^{\lbe *}\be F\e).
\]
It remains only to check that $\check{H}^{\le i}(X/Y,h^{\lbe *}\be F\e)=0$ for $i\geq 2$. Following \cite[p.~45]{san}, we define the map $\lambda_{\e i}\colon F(X)\lbe\oplus\lbe F(G\e)^{\le i}\to F(X)\lbe\oplus\lbe F(G\e)^{\le i-1}$ by $\lambda_{\e i}(a,b_{\le 1},\dots,b_{\le i})=(a,-b_{\le 1},b_{\le 3},\dots,b_{\le i})$, where $i\geq 2$. Then $\delta^{\e i-1}\circ \lambda_{\e i}+\lambda_{\e i+1}\circ\delta^{\e i}=1_{F(X)\le\oplus\le F(G\e)^{\le i}}$
for every $i\geq 2$, whence $\krn \delta^{\e i}=\img \delta^{\e i-1}$ for all $i\geq 2$, i.e., $\check{H}^{i}(X/Y,h^{\lbe *}\be F\e)=0$ for every $i\geq 2$ by \eqref{Fmap0}. The proof is now complete.
\end{proof}

\begin{remark}\label{gps} Let $1\to H\overset{\!i}{\to} G\to F\to 1$ be an exact sequence of groups of $\s C$. Then $G$ is an $H_{F}\e$-\e torsor over $F$ with action $\varsigma\colon G\times_{S}H\to G,\e (g,h\le)\mapsto g\e i(h)$ (see Example \ref{2ex}).
Now assume that all the conditions of the lemma hold in this setting. Then 
$\varphi\colon F(G\e)\to F(H)$ \eqref{vphi2} is the map $F(i\e)$. In effect, applying the  argument at the beginning of the proof of \cite[Lemma 2.7, p.~480]{ga4} to the unit section $\varepsilon$ of $G$ (rather than to the unit section of $H$), we obtain the formula $F(\varepsilon\be\times_{\be S}\be 1_{\lbe H}\le)=p_{\le F(H\le)}\circ \psi_{G,\e H}^{-1}$, whence
\[
\varphi=p_{\le F(H\le)}\circ \psi_{G,\e H}^{-1}\circ F(\varsigma\le)=F(\varepsilon\be\times_{\be S}\be 1_{\lbe H}\le)\circ F(\varsigma\le)=F(i),
\]
as claimed.
\end{remark}

\subsection{A complement} Let $S$ be a locally noetherian normal scheme, let $\s C=S_{\fl}$ and assume that the objects $X\to S$ and $G\to S$ considered above satisfy the conditions of Corollary \ref{xg}. Then an application of Lemma \ref{gsan} to the functor
$F=\br_{\lbe\rm a,\e S}^{\prime}\colon S_{\fl}\to\abe$ \eqref{braf} yields an isomorphism of abelian groups
\begin{equation}\label{anag}
\psi_{X,\e G}\colon \bra\oplus\br_{\be\rm a}^{\prime}(G/S)\to \br_{\be \rm a}^{\e\prime}(X\!\times_{\be S}\be G\lbe/\lbe S\le) 
\end{equation}
and a map  \eqref{vphi2}
\begin{equation}\label{vphi3}
\varphi=p_{\e\br_{\be\rm a}^{\prime}(G/S)}\circ \psi_{X,\e G}^{-1}\circ \br_{\be\rm a}^{\prime}\e\varsigma\colon\bra\to\br_{\be\rm a}^{\prime}(G/S\e).
\end{equation}
On the other hand, the functor $F=\br_{\be 1,\e S}^{\prime}\colon S_{\fl}\to\abe$ does {\it not} satisfy the conditions of Lemma \ref{gsan} since $F(1_{\lbe S})=\brp S$ is nonzero in general. However, by the proof of Corollary \ref{xg} and Remark \ref{comm}, the map \eqref{c0}
\[
\zeta_{\le X,\le G}\colon \bro\oplus\br_{\be \varepsilon}^{\prime}(G\lbe/\lbe S\le)\to\br_{\be 1}^{\e\prime}(X\!\times_{\be S}\be G\lbe/\lbe S\le)
\]
(which may be regarded as an analog of \eqref{anag} for $\br_{\be 1,\e S}^{\prime}$) is an isomorphism of abelian groups. Thus we may define the following variant of \eqref{vphi3}
\begin{equation}\label{vphi2p}
\varphi^{\e\prime}=p_{\e\br_{\be \varepsilon}^{\lbe\prime}(G\lbe/\lbe S\le)}\circ\zeta_{\le X,\le G}^{\e -1}\circ\br_{\be 1}^{\prime}\e\varsigma \colon \bro\to\br_{\be \varepsilon}^{\prime}(G\lbe/\lbe S\le).
\end{equation}
We now note that the following diagram has commutative squares
\[
\xymatrix{\bro\ar[r]^(.42){\br_{\be 1}^{\prime}\varsigma}\ar[d]_{c_{\e(X)}}&\br_{\be 1}^{\e\prime}(X\!\times_{\be S}\be G\lbe/\lbe S\le)\ar[d]_(.45){c_{(X\times_{\be S}G)}}\ar[r]^(.42){\zeta_{\le X,\le G}^{-1}}&  \bro\oplus\br_{\be \varepsilon}^{\prime}(G\lbe/\lbe S\le)\ar[d]_{(c_{\le(X)},\e c_{\le(G\le),\e \varepsilon})}\ar[rr]^(.6){p_{\le\br_{\be \varepsilon}^{\lbe\prime}(G\lbe/\lbe S\le)}}&&\br_{\be \varepsilon}^{\prime}(G\lbe/\lbe S\le)\ar[d]^{c_{\le(G\le),\e \varepsilon}}\\
\bra\ar[r]^(.42){\br_{\be \rm a}^{\prime}\varsigma}&\br_{\be \rm a}^{\e\prime}(X\!\times_{\be S}\be G\lbe/\lbe S\le)\ar[r]^(.42){\psi_{X,\e G}^{-1}}&\bra\oplus\br_{\be\rm a}^{\prime}(G/S)\ar[rr]^(.6){p_{\lle\br_{\be\rm a}^{\prime}(G/S)}}&&\br_{\be\rm a}^{\prime}(G/S\e),
}
\]
where $c_{\e(G\e),\e\varepsilon}$ is the map \eqref{csig0}. We conclude that the following diagram commutes
\begin{equation}\label{fif}
\xymatrix{\bro\ar[r]^{\varphi^{\le\prime}}\ar[d]_{c_{\e(X)}}&\br_{\be \varepsilon}^{\prime}(G\lbe/\lbe S\le)\ar[d]^{c_{\e(G\e),\e\varepsilon}}\\
\bra\ar[r]^{\varphi}&\br_{\be\rm a}^{\prime}(G/S\e).
}
\end{equation}
Next we claim that the following analog of \eqref{ana0} holds:
\begin{equation}\label{ana1}
\zeta_{\le X,\e G}\circ (i_{\le \br_{\be 1}^{\prime}\lbe(X/S\le)}+i_{\le \br_{\be \varepsilon}^{\prime}(G\lbe/\lbe S\le)}\circ\varphi^{\e\prime}\e)=\br_{\be 1}^{\prime}\e\varsigma.
\end{equation}
Given the definition of $\varphi^{\prime}$ \eqref{vphi2p}, we need only check that
\begin{equation}\label{ana2}
p_{\e\br_{\be 1}^{\prime}(X\be/\lbe S\e)}\circ\zeta_{\le X,\e G}^{\le -1}\circ \br_{\be 1}^{\prime}\e\varsigma=1_{\br_{\be 1}^{\prime}(X\be/\lbe S\e)}.
\end{equation}
To this end, recall $s_{\le 0}^{\e 0}=(1_{X},\varepsilon)_{S}$. Since $p_{\e G}\be\circ\be s_{0}^{\le 0}=\varepsilon\circ h\circ\xi$ and $\br_{\be 1}^{\prime}(\varepsilon)\circ p_{\e\br_{\be \varepsilon}^{\prime}(G\lbe/\lbe S\e)}=0$ by definition of $\br_{\be \varepsilon}^{\prime}(G\be/\lbe S\le)$ \eqref{brs}, we have
\[
(\br_{\be 1}^{\prime}\e s_{0}^{\le 0}\le)\be\circ\be \zeta_{\le X,\e G}=\br_{\be 1}^{\prime}(\e p_{\lbe X}\!\circ\!s_{0}^{\le 0}\le)\be\circ \lbe p_{\e\br_{\be 1}^{\prime}\lbe(X\be/\lbe S\e)}+\br_{\be 1}^{\prime}(\e p_{\e G}\be\circ\be s_{0}^{\le 0}\le)\be\circ\lbe p_{\e\br_{\be \varepsilon}^{\prime}(G\be/\lbe S\e)}=p_{\e\br_{\be 1}^{\prime}\lbe(X\be/\lbe S\e)},
\]
whence $p_{\e\br_{\be 1}^{\prime}\lbe(X\be/\lbe S\e)}\circ\zeta_{\le X,\le G}^{-1}\circ \br_{\be 1}^{\prime}\e\varsigma=(\br_{\be 1}^{\prime}\e s_{0}^{\le 0}\le)\!\circ\! (\br_{\be 1}^{\prime}\e\varsigma)=\br_{\be 1}^{\prime}(1_{X})=1_{\br_{\be 1}^{\prime}\lbe(X\be/\lbe S\e)}$, as claimed.

\begin{remark}\label{gps2} Let $1\to H\overset{\!i}{\to} G\to F\to 1$ be as in Remark \ref{gps} and assume that $G\to S$ and $H\to S$ satisfy the conditions of Corollary \ref{xg}. Write $\varepsilon_{\lbe H}$ and $\varepsilon_{\lbe G}$ for the unit sections of $H$ and $G$, respectively, and consider the diagram
\[
\xymatrix{\br_{\be 1}^{\prime}(G/S\e)\ar@/_2.7pc/[dd]_(.5){c_{(G\le)}}\ar[rr]^{\varphi^{\e\prime}} &&\br_{\! \varepsilon_{\lbe H}}^{\prime}\be(H\lbe/\lbe S\le)\ar@{=}[d]\ar@/^2.7pc/[dd]^(.5){c_{\le(H\le),\le \varepsilon_{\lbe H}}}\\
\br_{\be \varepsilon_{\lbe G}}^{\prime}\be(\le G\be/\lbe S\e)\ar[d]^{c_{\e(G\le),\e \varepsilon_{\lbe G}}}_{\simeq}\ar@{_{(}->}[u]\ar[rr]^{\br_{\!\varepsilon_{\be H}\lbe,\e\varepsilon_{\lbe G}}^{\prime}\le i}&&\br_{\! \varepsilon_{\lbe H}}^{\prime}\be(H\lbe/\lbe S\le)\ar[d]_{c_{\le(H\le),\le \varepsilon_{\lbe H}}}^{\simeq}\\
\br_{\be \rm a}^{\prime}\be(\le G\be/\lbe S\e)\ar[rr]^{\br_{\be \rm a}^{\prime} i\e=\e\varphi}&&\br_{\be \rm a}^{\prime}\be(\le H\be/\lbe S\e),
}
\]
where
\begin{equation}\label{vphi3p}
\varphi^{\e\prime}=p_{\e\br_{\be \varepsilon_{\lbe H}}^{\lbe\prime}\be(H\lbe/\lbe S\le)}\circ\zeta_{\e G,\le H}^{\e -1}\circ\br_{\be 1}^{\prime}\e\varsigma
\end{equation}
is the map \eqref{vphi2p}. By definition of $c_{\e(G\le),\e \varepsilon_{\lbe G}}$ \eqref{csig0}, the left-hand semicircle commutes. Further, the bottom square commutes by the commutativity of diagram \eqref{fty} and the equality on the bottom horizontal arrow follows from Remark \ref{gps}. Now the commutativity of diagram \eqref{fif} shows that the outer diagram commutes as well. Finally, since the map $c_{\le(H\le),\e \varepsilon_{\lbe H}}$ is an isomorphism, we conclude that the top square above commutes, i.e., the restriction of $\varphi^{\e\prime}$ to $\br_{\be \varepsilon_{\lbe G}}^{\prime}\be(\le G\be/\lbe S\e)$ is the map $\br_{\!\varepsilon_{\lbe H}\lbe,\e\varepsilon_{\lbe G}}^{\prime}\le i$ \eqref{csig1}.

\end{remark}

\subsection{The units-Picard-Brauer sequence} 
The following proposition generalizes \cite[(6.10.1), p.~43]{san}\,\footnote{In particular, the smoothness hypothesis on $Y$ in [loc.cit.] is unnecessary.}\,.

\begin{proposition}	{\rm (The units\e-\le Picard\e-\e Brauer sequence of a torsor)} \label{upb} Let $S$ be a locally noetherian normal scheme, $G$ a flat $S$-group scheme and $X\to Y$ a $G_{\le Y}$-torsor over $Y\!$, where $Y\to S$ is faithfully flat and locally of finite type. Assume that
\begin{enumerate}
\item[(i)] the structural morphism $X\to S$ has an \'etale quasi-section,
\item[(ii)] for every integer $i\geq 1$ and every \'etale and surjective morphism $T\to S$, $X_{T}$, $G^{\e i}_{\lbe T}$ and $X_{T}\be\times_{T}\be G^{\e i}_{T}$ are locally factorial, 
\item[(iii)] for every point $s\in S$ of codimension $\leq 1$, the fibers $X_{\lbe s}$ and $G_{\lbe s}$ are geometrically integral, and
\item[(iv)] for every maximal point $\eta$ of $S$, $G_{\lbe\eta}$ is $k(\eta)^{\rm s}$-rational.
\end{enumerate}
Then there exists a canonical exact sequence of abelian groups
\[
\begin{array}{rcl}
0\to \mathcal U_{\le S}(\le Y\le)\to \mathcal U_{\le S}(\lbe X)\to G^{\e *}\be\lle(S\e)\to\pic Y&\to&\pic X\to\npic\be(\lbe G\be/\be S\e)\\
&\to&\brp\, Y\to\brp X,
\end{array}
\]
where $\mathcal U_{\le S}$ is the functor \eqref{pu} and $\npic\be(\lbe G\be/\be S\e)$ is the group \eqref{npic}.
\end{proposition}
\begin{proof} Let $\s C$ be the full subcategory of $S_{\fl}$ whose objects are the schematically dominant morphisms. Then $\s C$ contains $1_{\lbe S}$ and is stable under products \cite[$\textrm{IV}_{3}$, Theorem 11.10.5(ii)]{ega}. Now consider the restrictions of the abelian presheaves \eqref{puf} and \eqref{pnp} to $\s C$:
\[
\mathcal U_{\e S}\colon \s C\to\abe, (Z\to S\e)\mapsto \mathcal U_{\le S}(Z),
\]
and
\[
\nps\colon \s C\to {\mathbf{Ab}}, (Z\to S\e)\mapsto \npiczs.
\]
If $F$ denotes either of the above functors, then $F(1_{\lbe S}\e)=0$ and Proposition \ref{padd} and Corollary \ref{xg} show that the maps \eqref{fxy} $\psi_{\le G^{\le i}\lbe,\e G}\colon F(G^{\e i}\le)\oplus F(G\e)\to F(G^{\e i+1}\le)$ and $\psi_{X,\e G^{\e i}}\colon F(X)\oplus F(G^{\e i})\to F(X\be \times_{\lbe S} G^{\e i}\e)$ are isomorphisms of abelian groups for every $i\geq 0$. Thus Lemma \ref{gsan} yields 
$\check{H}^{\lle i}(X/\le Y,h^{\lbe *}\lle \mathcal U_{\e S}\e)=\check{H}^{\lle i}(X/Y,h^{\lbe *}\lle \nps\e)=0$ for every $i\geq 2$, where $h\colon Y\to S$ is the structural morphism of $Y$. Further, by Lemmas \ref{pug} and \ref{gsan}, there exist canonical exact sequences of abelian groups
\begin{equation}\label{an}
0\to \check{H}^{\lle 0}\lbe(X/\e Y,h^{\lbe *}\lle \mathcal U_{\le S}\e)\to \mathcal U_{\le S}(X\le)\to G^{\le\le*}\be(S\e)\to \check{H}^{\lle 1}\lbe(X/\e Y,h^{\lbe *}\lle\mathcal U_{\le S}\e)\to 0
\end{equation}
and
\begin{equation}\label{de}
0\to \check{H}^{\le 0}\lbe(X/Y,h^{\lbe *} \nps\e)\to \npicxs\to\npic(G/S)\to\check{H}^{\le 1}\lbe(X/Y,h^{\lbe *} \nps\e)\to 0.
\end{equation}
We now compute $\check{H}^{\lle i}(X/\le Y,h^{\lbe *} \lle \mathcal U_{\e S}\e)$ and $\check{H}^{\lle i}(X/Y,h^{\lbe *}  \nps\e)$ for $i=0$ and 1.

Let $P_{\e\bg_{m,\lle S}(S\e)}$ be the constant presheaf on $\s C_{\be/Y}$ with value $\bg_{m,\e S}(S\e)$. Further, for every integer $j\geq 0$, let $\s H^{\le j}\lbe(\bg_{m,\e Y})$ be the presheaf on $\s C_{\be/Y}$ defined by $(\e Z\to Y\e)\mapsto H^{\le j}(Z_{\et},\bg_{m,\e Z})$ . Since every object $Z\to Y$ of $\s C_{\be/Y}$ is schematically dominant over $S$, the canonical map $\bg_{m,\e S}(S\e)\to\bg_{m,\e S}(Z\e)$ is injective. Consequently, there exists a canonical exact sequence of abelian presheaves on $\s C_{\be/Y}$
\[
1\to P_{\e\bg_{m,\lle S}(S\e)}\to \s H^{\e 0}\lbe(\bg_{m,\e Y})\to h^{\lbe *} \le\mathcal U_{\le S}\to 1.
\]
The preceding sequence induces isomorphisms in $\abe$
\begin{eqnarray}
\check{H}^{\le 0}(X\be/\e Y,h^{\lbe *} \le\mathcal U_{\e S})&=&\mathcal U_{\le S}(\le Y\le)\label{ein}\\
\check{H}^{\lle 1}(X\be/\e Y,h^{\lbe *} \le\mathcal U_{\le S})&=&\check{H}^{\lle 1}\lbe(X\be/\e Y,\s H^{\e 0}\lbe(\bg_{m,\e Y}))\label{zwei}\\
\check{H}^{\lle i}(X\be/\e Y,\s H^{\e 0}\lbe(\bg_{m,\e Y}))&=&0\quad (\e i\geq 2\e).\label{drei}
\end{eqnarray}
Next we consider the contravariant functor
\[
Q_{\lbe S}\colon\s C\to\abe, (\e Z\!\overset{\! g}{\to} S\e)\mapsto\img\be[\e\pic S\overset{\!\!\pic\lbe\lbe g}{\lra}\pic Z\e].
\]
Then there exists a canonical exact sequence of abelian presheaves on $\s C_{\be/Y}$
\begin{equation}\label{nseq}
1\to h^{\lbe *} \be Q_{\lbe S}\to \s H^{\le 1}\be(\bg_{m,\e Y}\be)\to h^{\lbe *}  \nps\to 1.
\end{equation}
We now compute the groups $\check{H}^{\lle i}(X\be/\e Y,h^{\lbe *}  Q_{\lbe S})$. Recall the $Y$-morphism $\partial_{\le i+1}^{\,j}\colon  X_{\!/Y}^{\le i+2}\to  X_{\!/Y}^{\le i+1}$ \eqref{fa1}, where $i\geq 0$ and $0\leq j\leq i+1$. Then $(\le h^{\lbe *} \lbe Q_{\lbe S}\e)(\partial_{\le i+1}^{\,j})=Q_{\lbe S}(\partial_{\le i+1}^{\,j})$ is the restriction of $\pic\partial_{\le i+1}^{\,j}\colon\pic X_{\!/Y}^{\le i+1}\to\pic X_{\!/Y}^{\le i+2}$ to $\img\pic\be\big(\e h\le\circ\le p_{\e Y}^{\e [\le i\le]}\e\big)$, where $p_{\e Y}^{\e [\le i\le]}\colon X_{\!/Y}^{\le i+1}\to Y$ is the structural morphism of $X_{\!/Y}^{\le i+1}$. Since $h\circ p_{\le Y}^{\e [\le i\le]}\circ \partial_{\le i+1}^{\,j}=h\circ p_{\le Y}^{\e [\le i\le]}\circ \partial_{\le i+1}^{\,k}$ for every pair of integers $j,k$ such that $0\leq j,k\leq i+1$, we have $Q_{\lbe S}(\partial_{\le i+1}^{\,j})=Q_{\lbe S}(\partial_{\le i+1}^{\,k})$ for all $j,k$ as above. Further, since $\partial_{\le i+1}^{\,j}$ has a section, namely \eqref{deg1}, the map $\pic \partial_{\le i+1}^{\,j}$ is injective for every $j$. Thus the complex $\big\{Q_{\lbe S}\lbe\big(X_{\!/Y}^{\le i+1}\big), \partial^{\, i+1}\big\}_{\e i\e\geq\e 0}$, where $\partial^{\, i+1}=\sum_{\e j=0}^{\e i+1}\e (-1)^{\le j} Q_{\lbe S}\lbe\big(\partial_{\le i+1}^{\,j}\lbe\big)$, is quasi-isomorphic to the complex whose only term is the group
$\img\pic\be(\e h\circ \xi\le)$ placed in degree 0, where $\xi=p_{\e Y}^{\e [\le 0\le]}\colon X\to Y$. Consequently, $\check{H}^{\le i}\lbe(X/Y,h^{\lbe *} \lbe Q_{\lbe S}\e)=0$ for every $i\geq 1$ and $\check{H}^{\le 0}\lbe(X/Y,h^{\lbe *} \lbe Q_{\lbe S}\e)=\img\pic\be(\e h\circ \xi\le)$. Now \eqref{nseq} yields
isomorphisms
\begin{equation}\label{vier}
\check{H}^{\le i}\lbe(X/Y,\s H^{\le 1}\be(\bg_{m,\e Y}\be)\e)=\check{H}^{\le i}\lbe(X/Y,h^{\lbe *}  \nps)\quad (\e i\geq 1\e)
\end{equation}
and an exact sequence
\begin{equation}\label{funf}
0\to \img\pic\be(\le h\!\circ\be \xi\e)\to \check{H}^{\le 0}\lbe(X/Y,\s H^{\le 1}\be(\bg_{m,\e Y}\be)\e)\to \check{H}^{\le 0}\lbe(X/Y,h^{\lbe *}  \nps)\to 0,
\end{equation}
where
\[
\check{H}^{\le 0}\lbe(X/Y,\s H^{\le 1}\be(\bg_{m,\e Y}\be)\e)=\krn[\e\pic\partial_{\le 1}^{\e 0}\!-\!\pic\partial_{\le 1}^{\e 1}\colon \pic X \to\pic(X\!\times_{Y}\!\be X)].
\]
Combining \eqref{an}-\eqref{zwei} and \eqref{vier}-\eqref{funf}, we obtain the exact sequences
\begin{equation}\label{sechs}
0\to \mathcal U_{\le S}(\le Y\le)\to \mathcal U_{\le S}(\be X)\to G^{\e *}\be(S\e)\to\check{H}^{\le 1}(X\be/\e Y,\s H^{\le 0}\lbe(\bg_{m,\e Y}))\to 0
\end{equation}
and
\begin{equation}\label{sieben}
0\to \pic X/\check{H}^{\le 0}\lbe(X/Y,\s H^{\le 1}\be(\bg_{m,\e Y}\be)\e)\to\npic(G/S\e)\to\check{H}^{\le 1}\lbe(X/Y,\s H^{\le 1}\be(\bg_{m,\e Y}\be)\e)\to 0.
\end{equation}
We now endow $\s C_{\be/Y}$ with the fppf topology and consider the spectral sequence for \v{C}ech cohomology associated to the fppf covering $\xi\colon X\to Y$ and the abelian sheaf $\bg_{m, Y}$ on $\s C_{\be/Y}$:
\[
\check{H}^{\le i}(X\be/\e Y,\s H^{\le j}\lbe(\bg_{m,\e Y}))\implies H^{\le i+j}(Y_{\fl},\bg_{m,\e Y}\be)=H^{\le i+j}(Y_{\et},\bg_{m,\e Y}\be).
\]
By \eqref{drei} and \cite[Case $E^{\e k}$ for $k=0$, p.~329]{ce}, the above spectral sequence induces exact sequences of abelian groups
\begin{equation}\label{deux}
0\to \check{H}^{\le 1}(X\be/\e Y,\s H^{\e 0}\lbe(\bg_{m,\e Y}))\to\pic Y\to\check{H}^{\le 0}(X\be/\e Y,\s H^{\le 1}\lbe(\bg_{m,\e Y}))\to 0
\end{equation}
and
\begin{equation}\label{trois}
0\to  \check{H}^{\le 1}\lbe(X\be/\e Y,\s H^{\le 1}\be(\bg_{m,\e Y}\be))\to \brp\, Y\to \check{H}^{\e 0}(X\be/\e Y,\s H^{\le 2}\lbe(\bg_{m,\e Y}))\to 0,
\end{equation}
where
\[
\check{H}^{\e 0}(X\be/\e Y,\s H^{\le 2}\lbe(\bg_{m,\e Y}))
=\krn[\e\brp \e\partial_{\e 1}^{\e 0}\!-\!\brp \e\partial_{\e 1}^{\e 1}\colon \brp X\to\brp(X\!\times_{Y}\!\be X\le)].
\]
Note that, since $\beta_{1}\colon X\times_{\lbe S} G\to X\!\times_{Y}\! X$ \eqref{no} is an isomorphism, we have
\[
\check{H}^{\e 0}(X\be/\e Y,\s H^{\le 2}\lbe(\bg_{m,\e Y}))=\krn[\e \brp(\partial^{\e 0}_{\le 1}\be\circ\be\beta_{1})\!-\!\brp (\partial_{\e 1}^{\e 1}\be\circ\be\beta_{1})\,]=\krn[\e \brp \varsigma-\brp p_{X}\e]
\]
by \eqref{b1} and \eqref{b2}. Thus, if
\begin{equation}\label{zeta}
\phi=\brp \varsigma\!-\!\brp p_{X}\colon \brp X\to\brp(X\!\times_{S}\be G\e),
\end{equation}
then \eqref{sechs}-\eqref{trois} yield exact sequences
\begin{equation}\label{sept}
0\to \mathcal U_{\le S}(\le Y\le)\to \mathcal U_{\le S}(\be X)\to G^{\le *}\be(S\e)\to\pic Y\to\check{H}^{\le 0}(X\be/\e Y,\s H^{\le 1}\lbe(\bg_{m,\e Y}))\to 0
\end{equation}
and
\begin{equation}\label{huit}
0\to \pic X/\check{H}^{\le 0}\lbe(X/Y,\s H^{\le 1}\be(\bg_{m,\e Y}\be)\e)\to\npic\be(G/S)\to\brp\, Y\to \brp X\overset{\!\phi}{\to}\brp(X\!\times_{\lbe S}\be G\e).
\end{equation}
The sequences \eqref{sept} and \eqref{huit} can now be assembled to yield the exact sequence of abelian groups
\begin{equation}\label{mg}
\begin{array}{rcl}
0\to \mathcal U_{\le S}(\le Y\le)\to \mathcal U_{\le S}(\be X)\to G^{\le *}\be(S\e)&\to&\pic Y\to\pic X\to\npic\be(\lbe G\be/\be S\e)\\
&\to&\brp\, Y\to\brp X\overset{\!\phi}{\to}\brp(X\!\times_{\lbe S}\be G\e),
\end{array}
\end{equation}
where $\phi$ is the map \eqref{zeta}. The proposition follows.
\end{proof}

\begin{remarks}\label{tter}\indent
\begin{enumerate}
\item[(a)] The sequence \eqref{mg} (and therefore also the sequence of the proposition) is functorial in the following sense. Let $G^{\e\prime}$ be a flat $S$-group scheme and let 
$X^{\le\prime}\to Y^{\le\prime}$ be a (right) $G^{\e\prime}_{\le Y^{\prime}}\e$-\e torsor over $Y^{\le\prime}\!$ such that the $S$-schemes $G^{\e\prime}, X^{\le\prime}, Y^{\le\prime}$ satisfy all the conditions of the proposition. Assume, furthermore, that there exist a morphism of $S$-group schemes $G\to G^{\e\prime}$ and morphisms of $S$-schemes $X\to X^{\le\prime}$ and $Y\to Y^{\le\prime}$ such that the following diagram, whose horizontal arrows are the corresponding group actions, commutes
\[
\xymatrix{X\times_{Y}G_{Y}\ar[d]\ar[r]& X\ar[d]\\
X^{\le\prime}\!\times_{Y^{\lle\prime}}\be G^{\e\prime}_{Y^{\prime}}\ar[r]& X^{\le\prime}.
}
\]
Then the following diagram, whose top and bottom rows are, respectively, the exact sequences \eqref{mg} associated to the triples $(X^{\le\prime}, Y^{\le\prime},G^{\e\prime}\e)$ and $(X,Y,G\e)$, commutes:
\[
\xymatrix{\dots\ar[r]&\pic Y^{\le\prime}\ar[d]\ar[r]&\pic X^{\le\prime}\ar[d]\ar[r]&\npic\be(\lbe G^{\e\prime}\be/\be S\e)\ar[d]\ar[r]&\brp\, Y^{\le\prime}\ar[d]\ar[r]&\dots\\
\dots\ar[r]&\pic Y\ar[r]&\pic X\ar[r]&\npic\be(\lbe G\be/\be S\e)\ar[r]&\brp\, Y\ar[r]&\dots.
}
\]
We do not carry out here the (lengthy) verification of this fact. We do, however, include below (after Remark \ref{ray}) a partial verification of the indicated functoriality in a particular case that will be relevant in the next section.

\item[(b)] The homomorphism of abelian groups $d\colon G^{\e *}\be\lle(S\e)\to\pic Y$ in \eqref{mg} is defined as follows: if $\chi\in G^{\e *}\be\lle(S\e)$, i.e., $\chi\colon G\to\bg_{m,\le S}$ is a morphism of $S$-group schemes, then
\[
d(\e\chi\e)=\chi_{\e Y}^{(1)}\be(\lle[\e X\e]\lle)=[\e X\!\wedge^{\lbe G_{Y}\be,\e\chi_{Y}}\!\bg_{m,\e Y,\e r}\e],
\]
where $\chi_{\e Y}^{(1)}\colon H^{\le 1}\lbe(\e Y_{\fl},G_{Y}\lbe)\to H^{\le 1}\lbe(\e Y_{\fl},\bg_{m,\le Y})=\pic Y$ is the map \eqref{u1} induced by $\chi_{Y}$ and $[\e X\e]$ is the class of the
$G_{Y}$-torsor $X\!\to Y\be$ in $H^{\le 1}\lbe(\e Y_{\fl},G_{Y}\lbe)$.

\item[(c)] In the setting of the proposition, assume in addition that $S$ is noetherian and (for simplicity) irreducible with function field $K$. Assume also that $G$ is of finite type over $S$ and, for every point $s$ of $S$ of codimension $1$, $Y_{\lbe s}$ is integral. By (a), \cite[Corollary 5.3]{ga4}, \cite[Proposition VII.1.3(4), p.~104]{ray} and the proposition, there exists a canonical exact and commutative diagram of abelian groups
\[
\xymatrix{&\pic S\ar[d]\ar@{=}[r]&\pic S\ar[d]&&&\\
G^{\e *}\be(S)\ar[r]\ar[d]^{\simeq}&\pic Y\ar@{->>}[d]\ar[r]^(.45){\pic \xi}&\pic X\ar@{->>}[d]\ar[r]&\npic\be(\lbe G\be/\be S\e)\ar[d]^{\simeq}\ar[r]&\brp\, Y\ar[d]\ar[r]^{\br^{\lle\prime} \xi}&\brp X\ar[d]\\
G^{\e *}\be(K)\ar[r]&\pic Y_{K}\ar[r]^(.45){\pic \xi_{K}}&\pic X_{K}\ar[r]&\pic G_{K}\ar[r]&\brp\, Y_{K}\ar[r]^{\brp \xi_{K}}&\brp\e X_{K}.
}
\]
It follows from the diagram that the canonical map $\krn \br^{\lle\prime} \xi\to\krn\br^{\lle\prime} \xi_{\le K}$ is an isomorphism. Further, if the canonical map $\mathcal U_{\le S}(\be X)\to \mathcal U_{K}\lbe(\be X_{\be K})$ is surjective, then $\krn \pic \xi\to\krn\pic \xi_{\le K}$ is an isomorphism as well. Thus, in this case, a substantial part of the sequence of the proposition is {\it essentially equivalent} to the corresponding part of the sequence over the field $K$. See also Remark \ref{lastb}.
\end{enumerate}
\end{remarks}

\smallskip

Assertion (ii) in the following statement generalizes \cite[(6.10.3), p.~43]{san}\,\footnote{\e Since it seems that the proof of the indicated result in \cite[p.~45, lines 21-26]{san} is incorrect, we provide a modified argument in the setting of this paper.}.

\begin{corollary}\label{last} Let the notation and hypotheses be those of the proposition. \begin{enumerate}
\item[(i)] There exists a canonical exact sequence of \'etale sheaves on $S$
\[
\begin{array}{rcl}
1\to \uys\to \uxs\to G^{\e *}&\to&\picys\to\picxs\to\picgs\\
&\to&\brys\to\brxs\to\brp_{\!\!X\be\times_{\be S}\le G\lbe/\lbe S}
\end{array}
\]
\item[(ii)] If the morphism $\picxs\to\picgs$ in {\rm (i)} is surjective, then there exists a canonical exact sequence of abelian groups
\[
\begin{array}{rcl}
0\to \mathcal U_{\le S}(\le Y\le)\to \mathcal U_{\le S}(\lbe X)\to G^{\e *}\be\lle(S\e)&\to&\pic Y\to\pic X\to\npic\be(\lbe G\be/\be S\e)\\
&\to&\broys\to\bro\overset{\!\varphi^{\le\prime}}{\to} \br_{\be \varepsilon}^{\prime}(G\lbe/\lbe S\le),
\end{array}
\]
where $\varepsilon$ is the unit section of $G\to S$ and $\varphi^{\e\prime}$ is the map \eqref{vphi2p}.
\end{enumerate}
\end{corollary}
\begin{proof} If $T\to S$ is an \'etale and surjective morphism of schemes, then $T, G_{T}, X_{T}$ and $ Y_{T}$ satisfy all the hypotheses of the proposition. Thus there exists an exact sequence of abelian groups \eqref{mg}
\begin{equation}\label{mgs}
\begin{array}{rcl}
0\to \mathcal U_{\e Y/S}(\le T\le)\to \mathcal U_{\e X/S}(\le T\le)\to G^{\le *}\be(T\e)&\to&\pic Y_{T}\to\pic X_{T} \to\npic\be(\lbe G_{T}\be/\e T\e)\\
&\to&\brp\, Y_{T}\to\brp X_{T}\to\brp(X_{T}\!\times_{\lbe T}\be G_{T}\e).
\end{array}
\end{equation}
The above sequence is the sequence of $T$-sections of a complex of abelian presheaves on $S_{\et}$ such that the corresponding complex of associated sheaves on $S_{\et}$ is exact, by the exactness of \eqref{mgs} for every $T\to S$ as above. Assertion (i) now follows, noting that $\uys$ and $\uxs$ are the \'etale sheaves on $S$ associated to $\mathcal U_{\e Y/S}$ and $\mathcal U_{\e X/S}$ (respectively) by Remark \ref{too}, and $\picgs$ is the \'etale sheaf on $S$ associated to the presheaf $T\mapsto\npic\be(G_{T}\lbe/\e T\e)$ by \cite[Definition 9.2.2, p.~252]{klei}. If $\picxs\to\picgs$ is surjective, then (i) yields the exactness of the right-hand column in the following commutative diagram of abelian groups with exact rows:
\[
\xymatrix{1\ar[r] & \broys\ar[r]\ar[d] & \brp\, Y\ar[r]\ar[d] &\brys(S\e)\ar@{^{(}->}[d] \\
1\ar[r] & \bro\ar[r]\ar[d]^{\phi_{\lle 1}}& \brp X \ar[r]\ar[d]^{\phi} & \brxs(S\e)\ar[d]\\
1\ar[r]&\br_{\be 1}^{\e\prime}(X\!\times_{\be S}\be G\lbe/\lbe S)\ar[r]&\brp(X\!
\times_{\be S}\be G\e)\ar[r]&\brp_{\!\!X\lbe\times_{\be S}G\lbe/\lbe S}(S\e),}
\]
where $\phi$ is the map \eqref{zeta} and $\phi_{\lle 1}=\br_{\be 1}^{\prime}\varsigma-\br_{\be 1}^{\prime}\e p_{\lbe X}$ is induced by $\phi$. 
By the exactness of \eqref{mg}, the middle column in the above diagram is also exact. It now follows that the left-hand column is exact and $\krn [\e\brp\, Y\to\brp X\e]=\krn[\e\broys\to\bro\e]$. Thus \eqref{mg} yields the exact sequence
\[
\begin{array}{rcl}
0\to \mathcal U_{\le S}(\le Y\le)\to \mathcal U_{\le S}(\be X)\to G^{\e *}\be(S\e)&\to&\pic Y\to\pic X\to\npic\be(\lbe G\be/\be S\e)\\
&\to&\broys\to\bro\overset{\!\be\phi_{1}}{\to}\br_{\be 1}^{\e\prime}(X\!\times_{\be S}\be G\lbe/\lbe S).
\end{array}
\]
Now, since we work under the hypotheses of Proposition \ref{upb}, which are the same as those of Corollary \ref{xg}, the proof of Corollary \ref{xg} and Remark \ref{comm}(a) together show that the map $\zeta_{\le X,\le G}\colon \bro\oplus\br_{\be \varepsilon}^{\prime}(G\be/\lbe S\le)\to\br_{\be 1}^{\e\prime}(X\be\times_{\be S}\le G\lbe/\lbe S\le)$ \eqref{c0} is an isomorphism of abelian groups. Thus the kernel of $\phi_{1}$ is the same as the kernel of the composite map
\begin{equation}\label{mag}
\bro\overset{\!\be\phi_{1}}{\to}\br_{\be 1}^{\e\prime}(X\!\times_{\be S}\be G\lbe/\lbe S)\overset{\!\zeta_{X,\le G}^{-1}}{\lra}\bro\oplus\br_{\be \varepsilon}^{\prime}(G\be/\lbe S\le).
\end{equation}
Now, by \eqref{toi}, \eqref{toi2} and \eqref{ana2}, we have
\[
\begin{array}{rcl}
p_{\e\br_{\be 1}^{\prime}(X\be/S\e)}\circ\zeta_{\le X,\le G}^{-1}\circ\phi_{\le 1}&=&p_{\e\br_{\be 1}^{\prime}(X\be/\lbe S\e)}\circ\zeta_{X,\le G}^{-1}\circ \br_{\be 1}^{\prime}\e\varsigma-
p_{\e\br_{\be 1}^{\prime}(X\be/\lbe S\e)}\circ\zeta_{\le X,\le G}^{-1}\circ \br_{\be 1}^{\prime}\e p_{X}\\
&=&1_{\br_{\be 1}^{\prime}(X\be/S\e)}-1_{\br_{\be 1}^{\prime}(X\be/S\e)}=0
\end{array}
\]
and
\[
\begin{array}{rcl}
p_{\e\br_{\be \varepsilon}^{\prime}(G\lbe/S\e)}\circ\zeta_{X,\le G}^{-1}\circ\phi_{\le 1}&=&p_{\e\br_{\be \varepsilon}^{\prime}(G\lbe/S\e)}\circ\zeta_{\le X,\le G}^{-1}\circ \br_{\be 1}^{\prime}\e\varsigma-
p_{\e\br_{\be \varepsilon}^{\prime}(G\lbe/S\e)}\circ\zeta_{\le X,\le G}^{-1}\circ \br_{\be 1}^{\prime}\e p_{X}\\
&=&\varphi^{\e\prime}-0=\varphi^{\e\prime},
\end{array}
\]	
where $\varphi^{\e\prime}\colon\bro\to\br_{\be \varepsilon}^{\prime}(G\be/\lbe S\le)$ is the map \eqref{vphi2p}. Thus the kernel of the composition \eqref{mag} is the kernel of $\varphi^{\e\prime}$, which completes the proof.
\end{proof}

\begin{corollary}\label{kcor} Let $1\to H\to G\to F\to 1$ be an exact sequence of smooth $S$-group schemes with connected fibers at all points of $S$ of codimension $\leq 1$, where $S$ is a locally noetherian regular scheme. Assume that, for every maximal point $\eta$ of $S$, $H_{\lbe\eta}$ is $k(\eta)^{\rm s}$-rational. Then the given sequence induces
\begin{enumerate}
\item[(i)] an exact sequence of abelian groups
\[
\begin{array}{rcl}
0\to F^{\le *}\be(\lbe S\le)\to G^{\e *}\be(\lbe S\le)\to H^{*}\be(\lbe S\le)\to\pic F&\to&\pic G \to\npic\be(\lbe H\be/\be S\e)\\
&\to&\brp\le F\to\brp\le G,
\end{array}
\]
\item[(ii)] an exact sequence of \'etale sheaves on $S$
\[
\begin{array}{rcl}
0\to F^{\e *}\to G^{\e *}\to H^{\le *}&\to&\pic_{\! F\lbe/\lbe S}\to\picgs\to \pic_{\! H\lbe/\lbe S}\\
&\to&\brp_{\!\!F/S}\to \brp_{\be\!G/S} \to \brp_{\be\!G\times_{S}H/S}
\end{array}
\]
and,
\item[(iii)] if the morphism $\picgs\to \pic_{\! H\lbe/\lbe S}$ in {\rm (ii)} is surjective, an exact sequence of abelian groups
\[
\begin{array}{rcl}
0\to F^{\le *}\be(\lbe S\le)\to G^{\e *}\be(\lbe S\le)\to H^{*}\be(\lbe S\le)&\to&\pic F\to\pic G \to\npic\be(\lbe G\be/\be S\e)\\
&\to&\br_{\be 1}^{\prime}(F\be/\lbe S\le)\to\br_{\be 1}^{\prime}(G\lbe/\lbe S\le )\overset{\!\varphi^{\e\prime}}{\to}\br_{\be \varepsilon}^{\prime}(H\be/\lbe S\le),
\end{array}
\]
where $\varepsilon$ denotes the unit section of $H\to S$ and $\varphi^{\e\prime}$ is the map \eqref{vphi3p}.
\end{enumerate}
\end{corollary}
\begin{proof} As noted in Example \ref{2ex}\e, the $F$-scheme $G$ is an $H_{F}\e$-\e torsor over $F$. Thus assertion (i) follows by applying Proposition \ref{upb} to $G\to F$, which is justified since all the conditions of that proposition hold true by Remark \ref{comm}(b). Assertion (ii) follows from Corollary \ref{last}(i) and Lemma \ref{gros}. Assertion (iii) follows from Corollary \ref{last}(ii) and Lemma \ref{pug}.
\end{proof}

\begin{remark}\label{ray} Regarding assertion (i) of the preceding corollary, Raynaud constructed in \cite[Proposition VII.1.5, pp.~106-107]{ray} a canonical complex of abelian groups $G^{\le\lle *}\be(\lbe S\le)\to H^{*}\be(\lbe S\le)\to\pic F\to\pic G$, where $S$ is any scheme, $G$ is an $S$-group scheme and $H$ is a subgroup scheme of $G$ such that the quotient fpqc sheaf $G\be/\be H$ is represented by an $S$-scheme $F$. If the maximal fibers of $H$ are not smooth, then the preceding complex may not be exact.
\end{remark}

The maps $H^{*}\lbe\lbe(\lbe S\e)\to\pic F$ and $H^{\le*}\!\!\to \pic_{\! F\be/\lbe S}$ in the sequences of Corollary \ref{kcor}, (i) and (ii), can be defined for {\it any} exact sequence of $S$-group schemes
\begin{equation}\label{short}
1\to H\to G\to F\to 1.
\end{equation}
The first map $d\colon H^{*}\lbe\lbe(\lbe S\e)\to\pic F$ is given by
\begin{equation}\label{dmap}
d(\e \chi)=\chi_{ F}^{(1)}\be(\lle[\e G\,]\lle)=[\e G\be\wedge^{\be H_{\lbe F}\lbe,\e \chi_{F}}\! \bg_{m,\e F,\e r}\le],
\end{equation}
where $\chi_{ F}^{(1)}\colon H^{\le 1}\lbe(\e F_{\fl},H_{F}\lbe)\to H^{\le 1}\lbe(\e F_{\fl},\bg_{m,\le F})=\pic F$ is the map \eqref{u1} induced by $\chi_{F}$ and $[\e G\e]$ is the class  of the
$H_{F}\e$-\e torsor $G\to F$ in $H^{\le 1}\lbe(\e F_{\fl},H_{F}\lbe)$. See Remark \ref{tter}(b). The map \eqref{dmap} is compatible with pullbacks\,\footnote{\e This fact is a particular case of the functoriality assertion in Remark \ref{tter}(a).}, i.e., if $g\colon F^{\e\prime}\to F$ is a morphism of $S$-group schemes, then the exact sequence $1\to H\to G_{\be F^{\le\prime}}\to F^{\e\prime}\to 1$ induced by \eqref{short} defines a map $d^{\e\prime}\colon H^{*}\be(\lbe S\e)\to\pic F^{\e\prime}$, namely $d^{\e\prime}(\e \chi)=\chi_{\lbe F^{\e\prime}}^{(1)}\be(\lle[\e G_{\be F^{\le\prime}}\e]\lle)=[\e G_{\be F^{\e\prime}}\be\wedge^{\be H_{\be F^{\le\prime}}\lbe,\e \chi_{\lbe F^{\le\prime}}}\! \bg_{m,\e F^{\e\prime},\e r}\le]$, such that the following diagram commutes
\begin{equation}\label{tia0}
\xymatrix{ 
H^{\le*}\lbe(\lbe S\le)\ar[dr]_(.45){d^{\e\prime}}\ar[r]^{d}& \pic F\ar[d]^(.45){\pic g}\\
&\pic F^{\e\prime}.
}
\end{equation}
See \eqref{bcf} and \eqref{pmap}. Now, for every morphism of schemes $T\to S$, there exist maps $H^{*}\be( T\e)\to\pic F_{ T}$ defined similarly to \eqref{dmap} which induce a morphism of abelian presheaves on $(\le{\rm Sch}/S\e)$. This morphism of presheaves induces, in turn, a morphism of the associated \'etale sheaves on $S$
\begin{equation}\label{vcm3}
\bm{d}\colon H^{\le*}\!\!\to \pic_{\! F\be/\lbe S},
\end{equation}
which is the map in the sequence of Corollary \ref{kcor}(ii). If $S$ is strictly local, then $\bm{d}(\lle S\e)$ agrees with the map $d$ \eqref{dmap} (see Remark \ref{sl} for the equality $\pic_{\! F\be/\lbe S}(\lbe S\e)=\pic F\e$). Further, there exist commutative diagrams analogous to \eqref{tia0} when $S$ is replaced by any $T$ as above. We conclude that there exists a canonical commutative diagram of \'etale sheaves on $S$:
\begin{equation}\label{tia1}
\xymatrix{ 
H^{\le*}\ar[dr]_(.4){\bm{d}^{\le\prime}}\ar[r]^(.45){\bm{d}}& \pic_{\! F\be/\lbe S}\ar[d]^(.45){\pic_{\!S}(\e g)}\\
&\pic_{\! F^{\le\prime}\be/\lbe S}.
}
\end{equation}

\section{Reductive group schemes}

In this section we establish the main theorem of the paper (Theorem \ref{ct} below), which generalizes the main theorem of \cite{bvh1}.

\smallskip

Let $S$ be a (non-empty) scheme. Henceforth, all $S$-group schemes below are tacitly assumed to be of finite type over $S$.

If $M$ is an $S$-group scheme of multiplicative type, then $M^{\e*}=\underline{\hom}_{\, S\text{-gr}}(M,\bg_{m,S})$ \eqref{gc} and 
$M_{*}=\underline{\hom}_{\, S\text{-gr}}(\bg_{m,S},M\e)$ 
are represented by (finitely generated) twisted constant $S$-group
schemes \cite[X, Corollary 4.5 and Theorem 5.6]{sga3}. If $f\colon M\to N$ is a morphism of $S$-group schemes, we will write $f^{(*)}\colon N^{*}\to M^{*}$ and $f_{(*)}\colon M_{*}\to N_{*}$ for the canonical morphisms induced by $f$.

\smallskip

An $S$-group scheme $G$ is called
\emph{reductive} (respectively, \emph{semisimple}, \emph{simply connected}) if $G$ is affine and smooth over $S$ and its
geometric fibers are \emph{connected} reductive (respectively, semisimple,
simply connected) algebraic groups. By convention, the trivial $S$-group scheme is simply connected. If $G$ is a reductive $S$-group scheme, we will write $\pi_{\lbe G}\colon G\to S$ and $\varepsilon_{G}\colon S\to G$ for the structural morphism and unit section of $G$, respectively. Further, $\rad(G\e)$ will denote the radical of $G$, i.e., the maximal torus in the center of $G$. Now recall that the derived group $G^{\der}$ of $G$ is a normal and semisimple $S$-subgroup scheme of $G$ such that $G^{\tor}=G/G^{\der}$ is the largest quotient of $G$ that is an $S$-torus. If $G$ is a semisimple $S$-group scheme, there exists a simply-connected $S$-group scheme $\Gtil$ and a central isogeny $\varphi\colon\Gtil\to G$. The $S$-group scheme $\Gtil$ is called the {\it simply-connected central cover} of $\,G$ and the group $\mu=\krn\varphi$ is called the {\it fundamental group} of $G$. See \cite[p.~1161]{ga2} for more details and relevant references. If $G$ is an arbitrary reductive $S$-group scheme, $\Gtil$ and $\mu\e$ will denote the simply connected central cover and fundamental group of $G^{\der}$, respectively. Note that $G^{\der}$ is simply connected if, and only if, $\mu=0$. The $S$-group scheme $\mu$ is a (possibly non-smooth) finite subgroup scheme of $\Gtil$ of multiplicative type. Consequently, $\mu^{*}$ is a finite and \'etale $S$-group scheme.

For lack of adequate references, we now present proofs of the following ``well-known" facts over a field.

\begin{proposition} \label{wk} Let $k$ be a field with fixed separable algebraic closure $\ks$ and let $G$ be a (connected) reductive algebraic $k$-group scheme. Then
\begin{enumerate}
\item[(i)] $G$ is a $\ks$-rational variety, and
\item[(ii)] $\pic G=0$ if $G$ is simply connected.
\end{enumerate}
\end{proposition}
\begin{proof} (See \cite{nn-1} and \cite{nn-2}) To prove (i), we may assume that $k=\ks$. Let $T$ be a (split) maximal $k$-torus in $G$, let $B\supset T$ be the Borel $k$-subgroup of $G$ such that the set of $T$-roots on ${\rm Lie}(B)$ is a chosen positive system of $T$-roots for $G$ and let $B^{\e\prime}\supset T$ be the Borel $k$-subgroup of $G$ opposite to $B$. If $U\subset B$ and $U^{\e\prime}\subset B^{\e\prime}$ are the $k$-split $k$-unipotent radicals, then the multiplication morphism $U^{\e\prime}\lbe\times_{k} T\lbe\times_{k}\lbe U\to G$ is an open immersion with $k$-rational source, which yields (i). See \cite[Propositions 2.1.8,(2),(3), p.~53, 2.1.10, p.~58 and 2.2.9, p.~67]{cgp}. Now let $k$ be any field and let $G$ be a simply connected $k$-group scheme. By \cite[Corollary 5.7]{ct08}\,\footnote{\e The proof of this result depends on (i) but not on the assertion that we aim to prove, i.e., (ii).}\,, $\pic G$ is canonically isomorphic to the group of central extensions of $G$ by $\bg_{m,\le k}$. Since any such extension is trivial by \cite[Proposition 2.4]{ga2}, the proof is complete.
\end{proof}

Let again $S$ be any nonempty scheme and let $G$ be a reductive $S$-group scheme.
There exist a canonical central extension of flat $S$-group schemes of finite type
\begin{equation}\label{mgg}
1\to\mu\overset{i}{\to}\Gtil\to G^{\der}\to 1
\end{equation}
and a canonical exact sequence of reductive $S$-group schemes
\begin{equation}\label{ggg}
1\to G^{\der}\overset{\iota}{\to} G\overset{\!q}{\to} G^{\tor}\to 1.
\end{equation}
When reference to $G$ is necessary, we will write
\begin{equation}\label{sg}
q_{\le(G\le)}\colon G\to G^{\tor}
\end{equation}
for the morphism $q$ in \eqref{ggg}. Note that, since $(G^{\der})^{*}=0$, \eqref{ggg} induces an isomorphism of \'etale twisted constant $S$-group schemes
\begin{equation}\label{iso0}
q_{\le(G\le)}^{(*)}\colon (G^{\tor})^{*}\isoto G^{\e *}.
\end{equation}
Further, the exact sequence \eqref{mgg} induces a morphism of abelian groups \eqref{dmap}
\begin{equation}\label{stl}
e\colon \mu^{\lbe*}\be(\lbe S\le)\!\!\to \pic G^{\der}, \, \chi\mapsto \big[\e \Gtil\be\wedge^{\lbe \mu_{G^{\lbe\der}},\e\chi_{G^{\lbe\der}}}\be\bg_{m,\e G^{\der}\be,\e r}\e\big],
\end{equation}
and a morphism of \'etale sheaves on $S$ \eqref{vcm3}
\begin{equation}\label{vcm2}
\bm{e}\colon \mu^{*}\!\!\to \pic_{\be G^{\der}\!/\lbe S}.
\end{equation}
If $S$ is strictly local, then $\bm{e}(\lle S\e)=e$. 

\begin{lemma} \label{bseq} Let $S$ be a scheme and let $1\to G_{1}\to G_{2}\to G_{3}\to 1$ be an exact sequence of	reductive $S$-group schemes. Then the given sequence induces an exact sequence of $S$-group schemes of multiplicative type
\[
1\to\mu_{1}\to\mu_{2}\to\mu_{3}\to G_{1}^{\tor}\to G_{2}^{\tor}\to G_{3}^{\tor}\to 1.
\]
\end{lemma}
\begin{proof} By \cite[Proposition 2.10]{ga2}, there exists an exact and commutative diagram of reductive $S$-group schemes
\[
\xymatrix{1\ar[r] &\Gtil_{1}\ar[r]\ar[d] &\Gtil_{2}\ar[r]
\ar[d] &\Gtil_{ 3}\ar[r]\ar[d] &1 \\
1\ar[r] & G_{1}\ar[r] & G_{2}\ar[r] & G_{3}\ar[r] &1.}
\]
Now, although the category of reductive $S$-group schemes is not abelian, the proof of the snake lemma given in \cite[Lemma 2.3, p.~307]{bey} can be adapted so that it applies to the above diagram (for example, the decomposition \cite[(2.2), p.~306]{bey} is valid if we set $X=G_{2}\!\times_{G_{3}}\be\Gtil_{3}$ there). The sequence of the lemma then follows from the above diagram as in \cite[Lemma 2.3, p.~307]{bey}. 
\end{proof}

\begin{lemma}\label{picr} Let $S$ be a locally noetherian normal scheme and let $G$ be a smooth $S$-group scheme with connected fibers at every point of $S$ of codimension $\leq 1$. Then there exists a canonically split exact sequence of abelian groups
\[
0\to\pic S\to \pic G\to \prod \pic G_{\lbe\eta}\to 0,
\]
where the product runs over the set of maximal points $\eta$ of $S$.
\end{lemma}
\begin{proof} See \cite[Corollary 5.3]{ga4} and note that the unit section of $G$ defines a retraction of the canonical map $\pic S\to \pic G$ which splits the above sequence \cite[Remark 3.1(c)]{ga4}.
\end{proof}

\begin{proposition} \label{qt} Let $S$ be a noetherian strictly local regular scheme and let $G$ be a reductive $S$-group scheme such that $G^{\der}$ is simply connected. Then $\pic G=0$.
\end{proposition}
\begin{proof} Since $G^{\tor}$ splits over a finite \'etale cover of $S$ \cite[X, Corollary 4.6(i)]{sga3}, we have $G^{\tor}\simeq\bg_{m,\le S}^{n}$ for some $n\geq 0$ by \cite[${\rm IV}_{4}$, Proposition 18.8.1(b)]{ega}. Now, since $\pic S=0$, Lemma \ref{picr} shows that $\pic G^{\tor}\simeq \pic\bg_{m,\le S}^{n}\simeq\pic\bg_{m,\le K}^{n}$, where $K$ is the function field of $S$. Thus, since $\pic\bg_{m,\le K}^{n}=H^{1}(K,\Z^{n})=0$ by \cite[Lemma 6.9(ii), p.~41]{san}, we have $\pic G^{\tor}=0$. Now Corollary \ref{kcor}(i) applied to the exact sequence \eqref{ggg} shows that the canonical map $\pic G\to\pic G^{\der}_{\be K}$ is injective. Since $\pic G^{\der}_{\be K}=0$ by Proposition \ref{wk}(ii), the proposition follows.
\end{proof}

\smallskip

A {\em  $t$-resolution} of a reductive $S$-group scheme $G$ is a central extension of reductive $S$-group schemes $(\s R\e)\colon 1\to T\to H\to G\to 1$, where  $T$ is an $S$-torus and $H^{\be\der}$ is simply connected. A {\it morphism of $t$-resolutions of $G$} is a morphism of central extensions of $G$ \cite[Definition 2.4]{bga}. We recall from  \cite[proof of Proposition 2.2]{bga} the following construction of a $t$-resolution of $G$.

The composition $\Gtil\twoheadrightarrow G^{\der}\hookrightarrow G$ and the product in $G\e$ define a faithfully flat morphism of $S$-group schemes $\rad(G\e)\times_{S}\e\Gtil\to G$ whose kernel $\mu^{\le\prime}$ is a finite $S$-group scheme of multiplicative type. Let $i^{\le\prime}\colon \mu^{\le\prime}\hookrightarrow \rad(G\e)\times_{S}\Gtil$ be the corresponding central embedding. Now choose a closed immersion $l^{\e\prime}\colon \mu^{\le\prime}\hookrightarrow T$, where $T$ is an $S$-torus, and let
\begin{equation}\label{pout}
H=\big(\rad(G\e)\!\times_{\be S}\be\Gtil\,\big)\be\wedge^{i^{\le\prime},\e \mu^{\le\prime},\e l^{\e\prime}}\! T
\end{equation}
be the pushout of $i^{\e\prime}$ and $l^{\e\prime}$. Since $H^{\lbe\der}$ is isomorphic to $\Gtil$ \cite[proof of Proposition 3.2]{ga2}, we obtain a $t$-resolution of $G$
\begin{equation} \label{tres}
1\to T\overset{j}\to H\overset{p}\to G\to 1\qquad(\s R\e).
\end{equation}
The above sequence is the bottom row of the following exact and commutative diagram of $S$-group schemes
\begin{equation} \label{tdiag}
\xymatrix{1\ar[r] & \mu^{\le\prime}\ar[r]^(.35){i^{\le\prime}}\ar@{^{(}->}[d]_{l^{\le\prime}} & \rad(G\e)\!\times_{\lbe S}\be\Gtil\ar@{^{(}->}[d]\ar[r] & G\ar[r]\ar@{=}[d]&1 \\
1\ar[r] & T\ar[r]^{j}& H\ar[r]^{p} & G\ar[r] & 1.}
\end{equation}
Now let
\begin{equation}\label{h1}
H_{1}=H\!\times_{G}\be G^{\der}.
\end{equation}
Then \eqref{tres} induces an exact sequence of reductive $S$-group schemes 
\begin{equation} \label{tresd}
1\to T\overset{\!\be j_{1}}\to H_{1}\overset{\!\be p_{1}}\to G^{\der}\to 1
\end{equation}
which fits into an exact and commutative diagram of reductive $S$-group schemes
\[
\xymatrix{1\ar[r] & T\ar[r]^{j_{1}}\ar@{=}[d] & H_{1}\ar@{^{(}->}[d]^{\iota_{\lbe H}}\ar[r]^{p_{\le 1}} & G^{\der}\ar[r]\ar@{^{(}->}[d]^{\iota} &1 \\
1\ar[r] & T\ar[r]^{j}& H\ar[r]^{p} & G\ar[r] & 1.}
\]
The preceding diagram induces an exact sequence of reductive $S$-group schemes
\begin{equation}\label{aseq}
1\to H_{1}\to H\to G^{\tor}\to 1.
\end{equation}
Now set $R=H^{\lbe\tor}$ and $R_{1}=H_{1}^{\tor}$. Then Lemma \ref{bseq} applied to the sequence \eqref{aseq} shows that $H_{1}^{\der}$ is simply connected and, moreover, yields an exact sequence of $S$-tori
\begin{equation} \label{tori}
1\to R_{1}\overset{\! r}\to R\to G^{\tor}\to 1.
\end{equation}
In particular, \eqref{tresd} is a $t$-resolution of $G^{\der}$. Next, by \cite[proof of Proposition 3.2]{ga2}, the $S$-group scheme $\mu^{\le\prime}$ introduced above is related to the fundamental group $\mu$ of $G^{\der}$ via a (canonical) exact sequence of $S$-group schemes of multiplicative type 
\[
1\to\mu\to\mu^{\le\prime}\to \rad(G\e)\to G^{\tor}\to 1.
\]
Let $l\colon\mu\hookrightarrow T$ be the composition of closed immersions $\mu\hookrightarrow\mu^{\le\prime}\overset{l^{\le\prime}}\hookrightarrow T$. Since $i^{\le\prime}\colon \mu^{\le\prime}\to \rad(G\e)\times_{\be S}\Gtil$ can be identified with $\varepsilon_{G}\times_{G}(\e \rad(G\e)\times_{\lbe S}\Gtil\e)$ and $\rad(G\e)\times_{G}G^{\der}$ is the trivial $S$-group scheme, we may identify the morphisms $i^{\le\prime}\times_{G}G^{\der}$ and $i\colon \mu\to\Gtil$. Thus the pullback of diagram \eqref{tdiag} along $\iota\colon G^{\der}\hookrightarrow G$ is (isomorphic to) the pushout diagram
\[
\xymatrix{1\ar[r] &\mu\ar[r]^{i}\ar@{^{(}->}[d]_(.45){l} &\Gtil\ar[r]\ar@{^{(}->}[d]& G^{\der}\ar[r]\ar@{=}[d] &1\\
1\ar[r] & T\ar[r]^{j_{1}} & H_{1} \ar[r]^{p_{1}} & G^{\der} \ar[r] &1,
}
\]
i.e., $H_{1}\simeq\Gtil\wedge^{\lbe i,\e\mu,\e l}\e T$. See \eqref{bcf}. Now \eqref{tbc} yields a canonical isomorphism of right $T_{G^{\der}}\e$-torsors over $G^{\der}$
\begin{equation} \label{semi1}
H_{1}\simeq\Gtil\wedge^{\mu_{G^{\lbe\der}},\, l_{\lbe G^{\lbe\der}}}\! T_{G^{\lbe\der},\, r}.
\end{equation}
Next, if $q_{\le(\lbe H_{1}\lle)}$ and $q_{\le(\lbe H\lle)}$ are the maps \eqref{sg} associated to $H_{1}$ and $H$, respectively, then there exist exact and commutative diagrams of $S$-group schemes
\begin{equation}\label{fund1}
\xymatrix{&1\ar[d] &1\ar[d] &1\ar[d] &\\
1\ar[r] &\mu\ar[r]^{i}\ar[d]_(.45){l} &\Gtil\ar[r]\ar[d]& G^{\der}\ar[r]\ar@{=}[d] &1\\
1\ar[r] & T\ar[r]^{j_{1}}\ar[d]_(.45){\rho_{\le 1}} & H_{1} \ar[r]^{p_{1}}\ar[d]^(.45){q_{\lle(\lbe H_{\lbe 1}\lbe)}} & G^{\der} \ar[r] &1\\
& R_{1}\ar[d]\ar@{=}[r] & R_{1}\ar[d]&&&\\
&1 &1 &&& 
}
\end{equation}
and
\begin{equation}\label{fund}
\xymatrix{ &1\ar[d] &1\ar[d] &1\ar[d] &\\
1\ar[r] &\mu\ar[r]^{i}\ar[d]_(.45){l} &\Gtil\ar[r]\ar[d]& G^{\der}\ar[r]\ar[d]^{\iota} & 1\\
1\ar[r] & T\ar[r]^{j}\ar[d]_(.45){\!\!\rho_{1}}\ar[dr]^{\!\!\rho}\ar[d] & H \ar[r]^{p}\ar[d]^(.44){q_{\lle(\lbe H)}} & G \ar[r]\ar[d] &1\\
1\ar[r] & R_{1}\ar[d]\ar[r]^{r} & R\ar[d]\ar[r] & G^{\tor}\ar[d]\ar[r] &1,\\
&1 &1 &1 & 
}
\end{equation}
where the bottom row in \eqref{fund} is the sequence \eqref{tori} and
\begin{equation}\label{ro}
\rho=r\!\circ\!\rho_{\le 1}=q_{\le(\lbe H\lle)}\circ j\colon T\to R.
\end{equation}
Note that the left-hand column in \eqref{fund1} induces an exact sequence of \'etale twisted constant $S$-group schemes
\begin{equation}\label{fdsd1}
1\to R_{1}^{\e *}\overset{\rho_{1}^{\lle (*)}}{\to} T^{\e *}\overset{\! l^{\le (*)}}{\to} \mu^{*}\to 1.
\end{equation}
Further, by the commutativity of \eqref{fund1}, we have
\begin{equation}\label{sab1}
\rho_{1}^{\lle (*)}\be(\le S\e)=j_{1}^{\lle (*)}\be(\le S\e)\lbe\circ\lbe q_{\le(\lbe H_{\lbe 1}\lbe)}^{\lle (*)}\lbe(\le S\e),
\end{equation}
where $q_{\le(\lbe H_{\lbe 1}\lbe)}^{\lle (*)}\colon R_{1}^{\e *}\isoto H_{1}^{\le *}$ is the isomorphism \eqref{iso0} associated to $H_{1}$.

Now we observe that a $t$-resolution \eqref{tres} of $G$ and the associated $t$-resolution \eqref{tresd} of $G^{\der}$ induce, respectively, morphisms of abelian groups \eqref{dmap}
\begin{equation}\label{e}
d\colon T^{\e*}\be(\lbe S\le)\!\!\to \pic G, \, \chi\mapsto \big[\e H\be\wedge^{T_{G},\e \chi_{G}}\be\bg_{m,\e G,\e r}\e\big]
\end{equation}
and
\begin{equation}\label{e1}
d^{\e\prime}\colon T^{\e*}\be(\lbe S\le)\!\!\to \pic G^{\der}, \, \chi\mapsto \big[\e H_{1}\be\wedge^{T_{G^{\der}},\e \chi_{G^{\der}}}\be\bg_{m,\e G^{\der}\be,\e r}\e\big]
\end{equation}
such that the following diagram commutes \eqref{tia0}
\begin{equation}\label{tia}
\xymatrix{ 
T^{\e*}\lbe(\lbe S\e)\ar[dr]_(.45){d^{\e\prime}}\ar[r]^{d}& \pic G\ar[d]^(.45){\pic \iota}\\
&\pic G^{\der}.
}
\end{equation}
Further, there exists a canonical morphism of \'etale sheaves on $S$ \eqref{vcm3}
\begin{equation}\label{bme}
\bm{d}\colon T^{\e*}\to \picgs
\end{equation}
such that $\bm{d}(S\e)=d$ if $S$ is strictly local and
\begin{equation}\label{tia2}
\xymatrix{ 
T^{\e*}\ar[dr]_(.45){\bm{d}^{\e\prime}}\ar[r]^(.45){\bm{d}}& \picgs\ar[d]^(.45){\pic_{\!S}(\le \iota)}\\
&\pic_{\be G^{\der}\!/\lbe S}
}
\end{equation}
is a commutative diagram of \'etale sheaves on $S$ \eqref{tia1}.

\begin{proposition} \label{key1} Let $S$ be a locally noetherian regular scheme and let $G$ be a reductive $S$-group scheme. Then there exists a canonical isomorphism of \'etale sheaves on $S$
\[
\picgs\simeq \mu^{\lbe *}.
\]
\end{proposition}
\begin{proof} We will show that the canonical maps $\pic_{\!S}(\le \iota)\colon \pic_{\be G\lbe/\lbe S}\to \pic_{\be G^{\der}\!/\lbe S}$ (induced by $\iota\colon G^{\der}\hookrightarrow G\e$) and $\bm{e}\colon \mu^{*}\!\!\to \pic_{\be G^{\der}\!/\lbe S}$ \eqref{vcm2} are isomorphisms of \'etale sheaves on $S$. By \cite[Tag 07QL, Lemma 15.42.10]{sp} and standard considerations \cite[Theorem 5.6(i), p.~118, Lemma 6.2.3, p. ~124, and Theorem 6.4.1, p.~128]{t}, we may assume that $S$ is noetherian, strictly local and regular, in which case the proof reduces to checking that the induced maps $e=\bm{e}(\le S\e)\colon\mu^{*}\be(\lbe S\le)\to \pic G^{\der}$ \eqref{stl} and $\pic \iota\colon \pic G\to\pic G^{\der}$ are isomorphisms of abelian groups. Let $H$ be given by \eqref{pout} and let $H_{1}=H\times_{G}\e G^{\der}$. Since $H_{1}^{\der}$ is simply connected, we have $\pic H_{1}=0$ by Proposition \ref{qt}. Thus Corollary \ref{kcor}(i) applied to \eqref{tresd} yields the bottom row of the following diagram of abelian groups with exact rows
\begin{equation}\label{rs}
\xymatrix{0\ar[r]& R_{1}^{\e *}\lbe(\lbe S\le)\ar[rr]^{\rho_{1}^{\lle (*)}\be(S\le)}\ar[d]_(.47){q_{(\lbe H_{\lbe 1}\lbe)}^{\lle (*)}\lbe(\le S\e)}^(.45){\simeq}&& T^{\e *}\be(\lbe S\le)\ar[rr]^{l^{\le (*)}\be(\lle S\lle)}\ar@{=}[d]&&\mu^{\le *}\be(\lbe S\le)\ar[d]^(.45){e}\ar[r]& 0\\
0\ar[r]& H_{1}^{\le *}\be(\lbe S\le)\ar[rr]^{j_{1}^{\lle (*)}\be(\le S\e)}&& T^{\e *}\be(\lbe S\le)\ar[rr]^{d^{\e\prime}}&&\pic G^{\der}\ar[r]& 0,
}
\end{equation}
where the top row is induced by the exact sequence \eqref{fdsd1} using \cite[II, Lemma 6.2.3, p.~124]{t}, the left-hand square commutes by \eqref{sab1} and $d^{\e\prime}$ is the map \eqref{e1}. We will show that the right-hand square in \eqref{rs} commutes, which will show that the right-hand vertical map $e$ is an isomorphism. Let $\chi\colon T\to \bg_{m,\e S}$ be a morphism of $S$-group schemes. By \eqref{semi1} and \cite[III, 1.3.1.3, p.~115, and 1.3.5, p.~116]{gi}, there exist isomorphisms of right $\bg_{m,\e G^{\lbe\der}}\e$-\e torsors over $G^{\der}$
\[
\begin{array}{rcl}
H_{1}\be\wedge^{T_{G^{\lbe\der}},\e \chi_{G^{\lbe\der}}}\be\bg_{m,\e G^{\lbe\der},\e r}&\simeq&(\Gtil\wedge^{\mu_{G^{\lbe\der}},\, l_{\lbe G^{\lbe\der}}}\! T_{G^{\lbe\der},\e r})\wedge^{T_{G^{\lbe\der}},\, \chi_{G^{\lbe\der}}}\be\bg_{m,\e G^{\lbe\der},\e r}\\
&\simeq&\Gtil\be\wedge^{\lbe \mu_{G^{\lbe\der}},\e(\e\chi\e\circ\e l\e)_{G^{\lbe\der}}}\be\bg_{m,\e G^{\lbe\der}\be,\e r}.
\end{array}
\]
Thus, by definitions \eqref{stl} and \eqref{e1}, we have
\[
d^{\e\prime}(\chi)=[\e H_{1}\be\wedge^{T_{\be G^{\lbe\der}},\e \chi_{G^{\lbe\der}}}\be\bg_{m,\e G^{\lbe\der},\e r}\e]=[\e \Gtil\be\wedge^{\lbe \mu_{G^{\lbe\der}},\e(\e\chi\e\circ\e l\e)_{G^{\lbe\der}}}\be\bg_{m,\e G^{\lbe\der}\be,\e r}\e]=(e\circ l^{\le (*)}\be(S\lle))(\chi),
\]
whence $d^{\e\prime}=e\circ l^{\le (*)}\be(\lle S\lle)$, as claimed. Thus $e$ is an isomorphism. Now, by Proposition \ref{qt} and Corollary \ref{kcor}(i) applied to the sequence \eqref{ggg}, the map $\pic \iota\colon \pic G\to\pic G^{\der}$ is injective. On the other hand, the commutative diagram \eqref{tia} and the surjectivity of $d^{\e\prime}$ \eqref{rs} show that $\pic \iota$ is surjective as well, which completes the proof.
\end{proof}

\begin{remarks} \indent
\begin{enumerate}
\item[(a)] In the case $S=\spec k$, where $k$ is a separably closed field, the preceding argument yields a new proof of the ``well-known" facts $\mu^{\le *}\be(k)=\pic G^{\der}=\pic G$ \cite[Lemma 6.9(iii), p.~41, and (6.11.4), p.~43]{san}.
\item[(b)] Under the hypotheses of the proposition, the above proof and the commutativity of diagram \eqref{tia2} show that the following diagram of \'etale sheaves on $S$ commutes
\begin{equation}\label{tia3}
\xymatrix{ 
T^{\e*}\ar[d]_(.45){l^{\le (*)}}\ar[r]^(.4){\bm{d}}&\picgs\ar[d]^(.45){\pic_{\!S}(\le \iota)}_(.45){\simeq}\\
\mu^{*}\ar[r]^(.35){\bm{e}}_(.35){\simeq}&\pic_{\be G^{\der}\!/\lbe S},
}
\end{equation}
where $\bm{d}$ and $\bm{e}$ are the maps \eqref{bme} and \eqref{vcm2}, respectively.
\end{enumerate}
\end{remarks}

\begin{corollary}\label{cder} Let $S$ be a locally noetherian regular scheme and let $G$ be a reductive $S$-group scheme such that $G^{\der}$ is simply connected. Then
\[
\picgs=0.
\]
Consequently, there exists a canonical isomorphism in $\dbs$
\[
\upicgs\isoto\ugs[1].
\]
\end{corollary}
\begin{proof} The first assertion of the corollary is immediate from the proposition since $\mu=0$. The triangle \eqref{t6} yields an isomorphism $\ugs[1]\isoto\upicgs$ whose inverse is the isomorphism of the corollary.
\end{proof}

We now apply Lemma \ref{bseq} to the middle row of diagram \eqref{fund}, i.e., the given $t$-resolution $(\s R\e)$ of $G$ \eqref{tres}, and obtain 
an exact sequence of $S$-group schemes of multiplicative type
\begin{equation}\label{fds}
1\to\mu\overset{\! l}{\to}  T\overset{\!\rho}{\to} R\to G^{\tor}\to 1,
\end{equation}
where $\rho$ is the map \eqref{ro}. The latter sequence induces an exact sequence of \'etale twisted constant $S$-group schemes
\begin{equation}\label{fdsd}
1\to (G^{\tor}\le)^{*}\to R^{\e *}\overset{\!\rho^{\lle (*)}}{\to} T^{\e *}\overset{\! l^{\le (*)}}{\to} \mu^{*}\to 1.
\end{equation}
We now consider the following object of $\cbs$ 
\[
\pi_{1}^{\le D}\be(\s R\e)=C^{\e\bullet}\be\big(\lle\rho^{\le (*)}\big)=(\e R^{\e*}\overset{\rho^{\le(*)}}\to T^{\e *}),
\]
where $R^{\e*}$ and $T^{\e*}$ are placed in degrees $-1$ and $0$, respectively.

\begin{lemma} Let $G$ be a reductive $S$-group scheme. Then a morphism $\phi\colon \s R_{1}\to \s R_{2}$ of $t$-resolutions of $G$ induces a quasi-isomorphism $\pi_{1}^{\le D}\be(\s R_{2}\e)\isoto\pi_{1}^{\le D}\be(\s R_{1}\e)$ in $\cbs$ such that, for $i=-1$ and $0$, the induced isomorphisms of \'etale sheaves on $S$
\[
H^{\le i}(\pi_{1}^{\le D}\be(\s R_{2}\e))\isoto H^{\le i}(\pi_{1}^{\le D}\be(\s R_{1}\e))
\]
are independent of the choice of $\phi$.
\end{lemma}
\begin{proof} Let $\s R_{\le i}\colon 1\to T_{i}\to H_{i}\to G\to 1$, where $i=1$ and $2$, be the given $t$-resolutions of $G$. By \eqref{fds}, the morphism of complexes $(\phi_{(\le T\le)},\phi_{(\le R\le)})\colon (T_{1}\to R_{1})\to (T_{2}\to R_{2})$ is a quasi-isomorphism. Further, if $\psi\colon \s R_{1}\to \s R_{2}$ is another morphism of $t$-resolutions, then the morphisms $\phi_{(H)}, \psi_{(H)}\colon H_{1}\to H_{2}$ differ by a morphism of $S$-group schemes $\alpha\colon H_{1}\to T_{2}$ that factors through $R_{1}$ \cite[proof of Lemma 2.7]{bga}. Thus, since $\alpha$ is trivial on $\krn(T_{1}\to R_{1})$, we conclude that the two isomorphisms $\krn(T_{1}\to R_{1})\isoto\krn(T_{2}\to R_{2})$ (respectively, $\cok(T_{1}\to R_{1})\isoto \cok(T_{2}\to R_{2})$) induced by $\phi_{(H)}$ and $\psi_{(H)}$ are equal. The lemma now follows from the above by duality, i.e., by applying to the preceding considerations the exact functor $M\to M^{\e*}$, where $M$ denotes an $S$-group scheme (of finite type and) of multiplicative type.
\end{proof}

\begin{lemma} Let $G$ be a reductive $S$-group scheme and let $\s R_{1}$ and $\s R_{2}$ be two $t$-resolutions of $G$. Then $\pi_{1}^{\le D}\be(\s R_{2}\e)$ and $\e\pi_{1}^{\le D}\be(\s R_{1}\e)$ are canonically isomorphic in $\dbs$.
\end{lemma}
\begin{proof} The proof is similar to the proof of \cite[Proposition 2.10]{bga}\e, using the previous lemma in place of \cite[Lemma 2.7]{bga}.
\end{proof}

\begin{definition} \label{def} Let $G$ be a reductive $S$-group scheme. Using the preceding lemma, we shall henceforth identify the objects $\pi_{1}^{\le D}\be(\s R\e)\in\dbs$ as $\s R$ ranges over the family of all $t$-resolutions of $G$. Their common value will be denoted by $\pi_{1}^{\le D}\be(G\e)$ and called the {\it dual algebraic fundamental complex of $G$}. Thus $\pi_{1}^{\le D}\be(G\e)=\pi_{1}^{\le D}\be(\s R\e)\in \dbs$ for any $t$-resolution $\s R$ of $G$.
\end{definition}

A morphism of reductive $S$-group schemes $\varphi\colon G^{\e\prime}\to G$ induces a morphism $\pi_{1}^{\le D}\be(\varphi)\colon\pi_{1}^{\le D}\be(G\e)\to\pi_{1}^{\le D}\be(G^{\e\prime}\e)$ in $\dbs$. Thus we obtain a contravariant functor $\pi_{1}^{\le D}$ from the category of reductive $S$-group schemes to $\dbs$. We will show next that $\pi_{1}^{\le D}$ is {\it exact}, i.e., transforms short exact sequences of reductive $S$-group schemes into distinguished triangles in $\dbs$. To this end, we prove

\begin{lemma}\label{trl} Let $S$ be a scheme and let $1\to G_{1}\to G_{2}\to G_{3}\to 1$	be an exact sequence of reductive $S$-group schemes. Then there exists an exact and commutative diagram of reductive $S$-group schemes 
\begin{equation}\label{trl0}
\xymatrix{ &1\ar[d] &1\ar[d] &1\ar[d] &\\
1\ar[r] & T_{1}\ar[r]\ar[d] & H_{1}\ar[r]\ar[d] & G_{1}\ar[r]\ar[d] &1\\
1\ar[r] & T_{2}\ar[r]\ar[d] & H_{2}\ar[r]\ar[d]  & G_{2}\ar[r]\ar[d]  &1\\
1\ar[r] & T_{3}\ar[r]\ar[d] & H_{3}\ar[r]\ar[d]& G_{3}\ar[r]\ar[d]&1,\\
&1 &1 &1 & 
}
\end{equation}
where the top, middle and bottom rows are $t$-resolutions of $G_{1}$, $G_{2}$ and $G_{3}$, respectively.
\end{lemma}
\begin{proof} Let $1\to T_{3}\to H_{3}\to G_{3}\to 1$ be a $t$-resolution $G_{3}$
and let
\begin{equation} \label{arb}
1\to T_{1}\to H\overset{\!p}\to G_{2}\to 1
\end{equation}
be a $t$-resolution of $G_{2}$. Set $T_{2}=T_{1}\times_{S}T_{3}$ and $H_{2}=H_{3}\times_{G_{3}}H$, where $H\to G_{3}$ is the composition $H\to G_{2}\to G_{3}$. By \cite[Proposition 2.8 and its proof\e]{ga2}, there exists a commutative diagram with exact rows
\begin{equation}\label{aux}
\xymatrix{1\ar[r] & T_{1}\ar[r] & H\ar[r]^{p} & G_{2}\ar[r]\ar@{=}[d] &1\\
1\ar[r] & T_{2}\ar[u]\ar[r]^{j_{\le 2}}\ar[d] & H_{2}\ar[u]^{q}\ar[r]^{p_{\le 2}}\ar[d]  & G_{2}\ar[r]\ar[d]  &1\\
1\ar[r] & T_{3}\ar[r]\ar[d] & H_{3}\ar[r]\ar[d]& G_{3}\ar[d]\ar[r]&1,\\
&1 &1 &1 & 
}
\end{equation}
where the middle row is a $t$-resolution of $G_{2}$. Note that, by the definition of $H_{2}$, there exists a canonical exact and commutative diagram
\begin{equation}\label{note}
\xymatrix{&& H_{2}\times_{H_{3}}S\ar@{^{(}->}[d]\ar[r]^(.5){\widetilde{p}_{\le 0}}_(.5){\sim} & H\times_{G_{3}}S\ar@{^{(}->}[d]\\
1\ar[r] & T_{3}\ar@{=}[d]\ar[r] & H_{2}\ar[d]\ar[r]^{\widetilde{p}} & H\ar[r]\ar[d]  &1\\
1\ar[r] & T_{3}\ar[r] & H_{3}\ar[r]& G_{3}\ar[r]&1.
}
\end{equation}
Now set $H_{1}=H\!\times_{G_{2}}\! G_{1}$. Then the pullback of \eqref{arb} along $G_{1}\to G_{2}$ is an exact sequence of reductive $S$-groups schemes
\[
1\to T_{1}\overset{\!\be j_{\lle 1}}\to H_{1}\overset{\!\be p_{\lle 1}}\to G_{1}\to 1.
\]
Next let $\nu\colon H_{1}\isoto H\!\times_{G_{3}}\! S$ be the composition of canonical isomorphisms
\begin{equation}\label{nu}
H_{1}=H\! \times_{G_{2}}\!  G_{1}\isoto H\! \times_{G_{2}}\! (G_{2}\! \times_{G_{3}}\! S\e)\isoto H\! \times_{G_{3}}\! S
\end{equation}
and consider the composition
\begin{equation}\label{psip}
\psi\colon H_{1}\overset{\!\nu}{\underset{\!\sim}{\to}} H\!\times_{G_{3}}\! S\overset{\!\widetilde{p}_{\le 0}^{\e-1}}{\underset{\!\sim}{\lra}}H_{2}\! \times_{H_{3}}\! S\hookrightarrow H_{2}.
\end{equation}
Then $1\to H_{1}\overset{\psi}\to H_{2}\to H_{3}\to 1$ is an exact sequence of reductive $S$-group schemes. Thus the rows and columns of the following diagram of reductive $S$-group schemes are exact:
\begin{equation}\label{trl1}
\xymatrix{ &1\ar[d] &1\ar[d] &1\ar[d] &\\
1\ar[r] & T_{1}\ar @{} [dr]|{\rm (\e \bf{I}\e)}\ar[r]^{j_{\lle 1}}\ar[d]_{r}& H_{1}\ar @{} [dr]|{\rm (\e \bf{II}\e)}\ar[r]^{p_{\le 1}}\ar[d]_{\psi} & G_{1}\ar[r]\ar[d]^{i} &1\\
1\ar[r] & T_{2}\ar[r]^{j_{\lle 2}}\ar[d] & H_{2}\ar[r]^{p_{\le 2}}\ar[d]  & G_{2}\ar[r]\ar[d]  &1\\
1\ar[r] & T_{3}\ar[r]\ar[d] & H_{3}\ar[r]\ar[d]& G_{3}\ar[r]\ar[d]&1.\\
&1 &1 &1 & 
}
\end{equation}
An application of \cite[Corollary 2.11]{ga2} to the middle column above shows at once that $H_{1}^{\der}$ is simply connected. Thus the top, middle and bottom rows of diagram \eqref{trl1} are $t$-resolutions of $G_{1}$, $G_{2}$ and $G_{3}$, respectively. Further, by the commutativity of \eqref{aux}, the lower half of diagram \eqref{trl1} commutes. Thus it remains only to check that the squares labeled ({\bf I}) and ({\bf II}) above commute. By the definitions of $H_{1}$ and $\nu$ \eqref{nu} and the commutativity of diagrams \eqref{aux} and \eqref{note}, there exists an exact and commutative diagram of reductive $S$-group schemes 
\[
\xymatrix{H_{2}\times_{H_{3}}S\ar@{^{(}->}[d]\ar[r]^(.5){\widetilde{p}_{\le 0}}_(.5){\sim} & H\times_{G_{3}}S\ar@{^{(}->}[d]\ar[r]^(.6){\nu^{-1}}_(.55){\sim}& H_{1}\ar@{^{(}->}[d]\ar[r]^{p_{\le 1}}& G_{1}\ar@{^{(}->}[d]^{i}\\
\ar@/_1pc/[rrr]_(.45){p_{\le 2}}H_{2}\ar[d]\ar[r]^{\widetilde{p}} & H\ar@{=}[r]&H\ar[d]\ar[r]^{p}& G_{2}\\
H_{3}\ar[rr]&& G_{3}.&&
}
\]
Thus, by the definition of $\psi$ \eqref{psip}, we have
\[
p_{\le 2}\circ \psi=p_{\le 2}\!\!\mid_{H_{2}\times_{H_{3}}S}\circ\,\widetilde{p}_{\le 0}^{\le -1}\circ \nu= (\e i\circ p_{\le 1}\circ  \nu^{-1}\circ \widetilde{p}_{\le 0}\e)\circ \widetilde{p}_{\le 0}^{\le -1}\circ \nu=i\circ p_{\le 1},
\]
i.e., the square labeled ({\bf II}) in diagram \eqref{trl1} commutes. The commutativity of the square labeled ({\bf I}) in \eqref{trl1} follows from the commutativity of the following diagram:
\[
\xymatrix{T_{1}\ar@/^1.3pc/[rr]^(.5){j_{1}}\ar[ddd]_(.45){r}\ar[r]^{\sim} &H_{1}\! \times_{G_{1}}\! S\,\ar[d]^{\sim}\ar@{^{(}->}[r]& H_{1}\ar[d]^{\nu}_{\sim}\ar@/^2.5pc/[ddd]^(.45){\psi}\\
& H\!\times_{G_{2}}\! S\,\ar@{^{(}->}[d]\ar@{^{(}->}[r]&  H\! \times_{G_{3}}\! S\ar[d]^{\widetilde{p}_{\le 0}^{\, -1}}_{\sim}\\
&(H\! \times_{G_{2}}\! S\e)\! \times_{S}\! (H_{3}\! \times_{G_{3}}\! S\e)\ar[d]^{\sim}&  H_{2}\! \times_{H_{3}}\! S\ar@{^{(}->}[d]\\
T_{2}\ar@/_1.3pc/[rr]_(.5){j_{2}}\ar[r]^{\sim}& H_{2}\! \times_{G_{2}}\! S\,\ar@{^{(}->}[r]&H_{2}.
}
\]
\end{proof}

\begin{proposition}\label{dex} Let $S$ be a scheme and let $1\to G_{1}\to G_{2}\to G_{3}\to 1$ be an exact sequence of reductive $S$-group schemes. Then the given sequence induces a distinguished triangle in $\dbs$
\[
\pi_{1}^{D}\lbe(G_{3}\e)\to \pi_{1}^{D}\lbe(G_{2}\e)\to \pi_{1}^{D}\lbe(G_{1}\e)\to\pi_{1}^{D}\lbe(G_{3}\e)[1].
\]
\end{proposition}
\begin{proof} By Lemma \ref{trl}\e, there exist $t$-resolutions $(\s R_{\le i}\e)\colon 1\to T_{i}\to H_{i}\to G_{i}\to 1$ of $G_{i}$, where $i=1,2$ or 3, and an exact and commutative diagram of reductive $S$-group schemes 
\[
\xymatrix{
1\ar[r] & T_{1}\ar[r]\ar[d] & T_{2}\ar[r]\ar[d] & T_{3}\ar[r]\ar[d] &1\\
1\ar[r] & H_{1}\ar[r] & H_{2}\ar[r]  & H_{3}\ar[r]  &1.
}
\]
The preceding diagram induces an exact and commutative diagram of $S$-tori 
\begin{equation}\label{bdi}
\xymatrix{
1\ar[r] & T_{1}\ar[r]\ar[d]^{\rho_{1}} & T_{2}\ar[r]\ar[d]^{\rho_{\le 2}} & T_{3}\ar[r]\ar[d]^{\rho_{\le 3}} &1\\
1\ar[r] & R_{\le 1}\ar[r] & R_{\le 2}\ar[r]  & R_{\le 3}\ar[r]  &1,
}
\end{equation}
where $R_{\le i}=H_{i}^{\lbe\tor}$, the bottom sequence is obtained by applying Lemma \ref{bseq} to the sequence $1\to H_{1}\to H_{2}\to H_{3}\to 1$ and the maps $\rho_{\le i}$ are the compositions $T_{i}\to H_{i}\to R_{\le i}$ \eqref{ro}. Now \eqref{bdi} induces an exact and commutative diagram of \'etale twisted constant $S$-group schemes
\[
\xymatrix{
1\ar[r] & R_{\le 3}^{\e*}\ar[r]\ar[d]^{\rho_{3}^{(*)}} & R_{\le 2}^{\e*}\ar[r]\ar[d]^{\rho_{2}^{(*)}} & R_{\le 1}^{\e*}\ar[r]\ar[d]^{\rho_{1}^{(*)}} &1\\
1\ar[r] & T_{3}^{\e*}\ar[r] & T_{2}^{\e*}\ar[r]  & T_{1}^{\e*}\ar[r]  &1
}
\]
which induces, in turn, an exact sequence in $\cbs$
\[
1\to \pi_{1}^{\le D}\be(\s R_{\le 3}\e)\to \pi_{1}^{\le D}\be(\s R_{\le 2}\e)\to \pi_{1}^{\le D}\be(\s R_{\le 1}\e)\to 1.
\]
The latter sequence induces a distinguished triangle in $\dbs$
\[
\pi_{1}^{\le D}\be(\s R_{\le 3}\e)\to \pi_{1}^{\le D}\be(\s R_{\le 2}\e)\to \pi_{1}^{\le D}\be(\s R_{\le 1}\e)\to \pi_{1}^{\le D}\be(\s R_{\le 3}\e)[\le 1\le],
\]
which yields the proposition.
\end{proof}

\begin{remark} Lemma \ref{trl} also yields a new proof of the exactness of the covariant functor $\pi_{1}(G)=\pi_{1}(\s R\e)=\cok\be[\e T_{*}\!\to\! R_{\le*}]$ on the category of reductive $S$-group schemes \cite[Theorem 3.8]{bga}. Indeed, diagram \eqref{bdi} induces an exact and commutative diagram of \'etale twisted constant $S$-group schemes
\[
\xymatrix{
1\ar[r] & (T_{1})_{*}\ar[r]\ar@{^{(}->}[d]^{\rho_{\le 1,\le(*)}} &  (T_{2})_{*}\ar[r]\ar@{^{(}->}[d]^{\rho_{\le 2,\le(*)}} &  (T_{3})_{*}\ar[r]\ar@{^{(}->}[d]^{\rho_{\le 3,\le(*)}}  &1\\
1\ar[r] & (R_{\le 1})_{*}\ar[r]& (R_{\le 2})_{*}\ar[r] & (R_{\e 3})_{*}\ar[r] & 1
}
\]
which immediately yields an exact sequence of \'etale twisted constant $S$-group schemes
\[
1\to\pi_{1}(G_{1}\e)\to\pi_{1}(G_{2}\e)\to\pi_{1}(G_{3}\e)\to 1.
\]
\end{remark}

\begin{proposition} \label{ctl} Let $S$ be a locally noetherian regular scheme, let $G$ be a reductive $S$-group scheme and let $(\s R\e)\colon 1\to T\overset{j}\to H\overset{p}{\to} G\to 1$ be a $t$-resolution of $G$. Then 
\begin{enumerate}
\item[(i)] the map ${\rm U}_{\be S}\lbe(\le p)\colon\ugs\!\to\! {\rm U}_{\lbe H\be/\lbe S}$ induces an isomorphism of \'etale sheaves on $S$
\[
\ugs\isoto \krn{\rm U}_{\be S}\lbe(\e j\le),
\]
and
\item[(ii)] there exists a canonical quasi-isomorphism in $\cbs$
\[
C^{\e\bullet}\lbe({\rm U}_{\be S}(\e j\e))\isoto \pi_{1}^{\le D}\be(\s R\e).
\]
\end{enumerate}
\end{proposition}
\begin{proof} Since $H^{\be\der}$ is simply connected, Corollary \ref{cder} shows that $\pic_{\!\lbe H\be/\lbe S}=0$. Thus Corollary \ref{kcor}(ii) applied to the given $t$-resolution $(\s R\e)$ yields a canonical exact sequence of \'etale sheaves on $S$
\begin{equation}\label{seq1}
1\to G^{\e*}\overset{\! p^{(*)}}{\to} H^{\le*}\overset{\! j^{(*)}}{\to}  T^{\e\le*}\!\overset{\!\be\bm{d}}\to\picgs\to 1,
\end{equation}
where $\bm{d}$ is the map \eqref{bme}. On the other hand, since ${\rm U}_{\lbe S}(\varepsilon_{G})$ is the zero morphism and $p\circ j=\varepsilon_{G}\circ\pi_{T}$, we have ${\rm U}_{\be S}(\e j\e)\circ {\rm U}_{\be S}(\e p\e)=0$.
Thus the sequence 	
\begin{equation}\label{seq3}
1\to \ugs\overset{\!\be {\rm U}_{\be S}\lbe(\e p)}\to {\rm U}_{\lbe H\be/\lbe S}\overset{\be {\rm U}_{\be S}\lbe(\e j\le)}\to {\rm U}_{\lbe T\be/\lbe S}\to\cok  {\rm U}_{\be S}\lbe(\e j\le)\to 1
\end{equation}
is a complex of \'etale sheaves on $S$. Now, since $p$ is faithfully flat and therefore schematically dominant, Lemma \ref{use} shows that ${\rm U}_{\be S}(\e p\e)$ is injective. Thus \eqref{seq3} can fail to be exact only at ${\rm U}_{\lbe H\be/\lbe S}$. Now consider the diagram
\begin{equation}\label{dia}
\xymatrix{
1\ar[r]& \ar @{} [dr]|{\rm (\e \bf{I}\e)}\ugs\ar[d]_(.47){\omega_{\le G}}^(.47){\simeq}\ar[r]^{{\rm U}_{\be S}(\e p\e)}& \ar @{} [dr]|{\rm (\e \bf{II}\e)}{\rm U}_{\lbe H\be/\lbe S}\ar[d]_(.47){\omega_{\le H}}^(.47){\simeq}\ar[r]^{{\rm U}_{\be S}(\e j\e)}& \ar @{} [dr]|{\rm (\e \bf{III}\e)}{\rm U}_{\lbe T\be/\lbe S}\ar[d]_(.47){\omega_{\e T}}^(.47){\simeq}\ar[r]&\cok  {\rm U}_{\be S}\lbe(\e j\le)\ar@{-->}[d]^{\simeq}\ar[r]&1\\
1\ar[r]& G^{\le*}\ar[d]_{\big(\be q_{(G)}^{(*)}\lbe\big)^{\!-1}}^{\simeq}\ar[r]^{p^{\le(*)}}&H^{\le*}\ar[d]_{\!\!\big(\be
q_{(H)}^{(*)}\lbe\big)^{\!-1}}^{\simeq}\ar[r]^{j^{\le(*)}}& T^{\e*}\ar@{=}[d]\ar[r]^(.4){\bm{d}}&\picgs\ar[d]_{\bm{e}^{-1}\circ\e\pic_{\!S}(\le \iota)}^{\simeq}\ar[r]&  1\\
1\ar[r]& (G^{\tor})^{*}\ar[r]& R^{\e *}\ar[r]^{\rho^{\le(*)}}&T^{\e*}\ar[r]^{l^{\e(*)}}\ar[r]& \mu^{\e *}\ar[r]& 1.
}
\end{equation}
The top row above is the complex \eqref{seq3}, and the middle and bottom rows are the exact sequences \eqref{seq1} and \eqref{fdsd}, respectively. The continuous vertical arrows in the upper half of the diagram are the maps \eqref{us}. In the lower half of the diagram, the left-hand and middle vertical arrows are the inverses of the maps \eqref{iso0} associated to $G$ and $H$, respectively. Further, the map $\bm{e}^{-1}\!\circ\!\pic_{\!S}(\le \iota)$ is an isomorphism by the proof of Proposition \ref{key1}. Using the definitions of the maps \eqref{us} (see \eqref{qmap}) and the identities $p\circ \varepsilon_{H}=\varepsilon_{G}$ and $j\circ \varepsilon_{T}=\varepsilon_{H}$, it is not difficult to check that the squares labeled ({\bf I}) and ({\bf II}) in \eqref{dia} commute. Thus the top row of diagram \eqref{dia}, i.e., the complex \eqref{seq3}, is exact, whence (i) follows. Further, there exists a unique way to define the discontinuous vertical arrow in diagram \eqref{dia} so that the resulting map is an isomorphism of \'etale sheaves on $S$ and the square labeled ({\bf III}) in \eqref{dia} commutes. Since the bottom squares in diagram \eqref{dia} also commute by the commutativity of diagrams \eqref{fund} and \eqref{tia3}, the entire diagram \eqref{dia} commutes, whence the map
$C^{\e\bullet}\lbe({\rm U}_{\be S}(\e j\e))\to C^{\e\bullet}\be\big(\lle\rho^{\lle (*)}\big)=\pi_{1}^{\le D}\be(\s R\e)$ with components $\big(\be q_{(H)}^{(*)}\lbe\big)^{\!-1}\!\be\circ\be\omega_{H}\colon {\rm U}_{\lbe H\be/\lbe S}\to R^{\e*}$ and $\omega_{\e T}\colon {\rm U}_{\lbe T\be/\lbe S}\to T^{\e*}$ (in degrees $-1$ and $0$, respectively) is a morphism of complexes. The diagram shows that the map just defined is a quasi-isomorphism in $\cbs$, which completes the proof.
\end{proof}

We may now prove the main theorem of the paper.

\begin{theorem}\label{ct} Let $S$ be a locally noetherian regular scheme and let $G$ be a reductive $S$-group scheme. Then there exists an isomorphism in $\dbs$
\[
\upicgs\isoto\pi_{1}^{\le D}\be(G\e)
\]
which is functorial in $G$.
\end{theorem}
\begin{proof} By definition of $\pi_{1}^{\le D}\be(G\e)$ (see Definition \ref{def}), it suffices to show that, if $(\s R\e)\colon 1\to T\overset{j}\to H\overset{p}{\to} G\to 1$ is any $t$-resolution of $G$, then there exists an isomorphism in $\dbs$
\begin{equation}\label{isor}
\upicgs\isoto \pi_{1}^{\le D}\be(\s R\e)
\end{equation}
which is functorial in $G$. We begin by noting that
\begin{equation}\label{not}
\upic_{S}\lbe(\e j)\be\circ\be \upic_{S}\lbe(\e p)=\upic_{S}\lbe(\e p\lbe\circ\lbe j)=\upic_{S}\lbe(\e \varepsilon_{G}\be\circ\be\pi_{\le T})=0,
\end{equation}	
since $\upic_{S}\lbe(\e \varepsilon_{G})\colon \upicgs\to\upic_{S/S}=0$ is the zero morphism. Now, since $\dbs$ is a triangulated category, there exists a distinguished triangle in $\dbs$ containing the morphism $\upic_{S}(\e p)\colon \upicgs\to \upic_{H/S}$, i.e., for some object $Z$ of $\dbs$, the top row of the following diagram is a distinguished triangle in $\dbs$:
\begin{equation}\label{t7}
\xymatrix{
\upicgs\ar@{-->}[d]_{f}\ar[rr]^(.45){\upic_{S}\lbe(\e p)}&&\upic_{H/S} \ar[d]_(.46){g}^(.46){\simeq}\ar[rr]^{u}&& Z\ar@{-->}[d]_(.45){h}\ar[r]&\upicgs[1]\ar@{-->}[d]_{f[1]}\\
C^{\e\bullet}\lbe({\rm U}_{\be S}(\e j\e))\ar[rr]^{v}&& {\rm U}_{H/S}[1]\ar[rr]^{{\rm U}_{\be S}(\e j\e)[1]}&& {\rm U}_{T/S}[1]\ar[r]&C^{\e\bullet}\lbe({\rm U}_{\be S}(\e j\e))[1].
}
\end{equation}
The bottom row of the above diagram is a distinguished triangle of the form \eqref{v} and the map labeled $g$ above is the isomorphism of Corollary \ref{cder}. Note that, since $H^{r}\lbe(\le \upic_{H/S})=0$ for $r\neq -1$ and
$H^{r}(\e\upicgs)=0$ for $r\neq -1,0$, we have $H^{r}(\e Z\le)=0$ for $r\neq -2, -1$. Now, by \eqref{not} and the commutativity of diagram \eqref{cderr} applied to the $S$-morphism $j$, we have ${\rm U}_{\be S}(\e j\e)[1]\circ g\circ \upic_{S}\lbe(\e p)=0$. Thus, by \cite[Proposition 1.1.9, p.~23]{fp}, there exist morphisms $f\colon \upicgs\to C^{\e\bullet}\lbe({\rm U}_{\be S}(\e j\e))$ and $h\colon Z\to {\rm U}_{T/S}[1]$ in diagram \eqref{t7} such that $(f,g,h)$ is a morphism of distinguished triangles. Now, since $H^{-1}\lbe(\e v)$ is the inclusion $\krn\le {\rm U}_{S}(\e j)\hookrightarrow {\rm U}_{H/S}$, the commutativity of the left-hand square in \eqref{t7} shows that the map $H^{-1}\lbe(\e f\le )\colon H^{-1}\lbe(\e\upicgs)\to H^{-1}(\e C^{\e\bullet}\lbe({\rm U}_{\be S}(\e j\e)))$ is the isomorphism $\ugs\isoto \krn\le {\rm U}_{S}(\e j)$ of Proposition \eqref{ctl}(i). Further, $H^{-2}(\e Z\le)=0$, whence $H^{r}(\e Z\le)=0$ for all $r\neq -1$. We now consider the diagram
\begin{equation}\label{t8}
\xymatrix{
C^{\e\bullet}\lbe({\rm U}_{\be S}(\e j\e))\ar@{-->}[d]_(.45){\widetilde{f}}\ar[rr]^{v}&&{\rm U}_{H/S}[1]\ar[d]_{g^{-1}}^{\simeq}\ar[rr]^{{\rm U}_{\be S}(\e j\e)[1]}&& {\rm U}_{T/S}[1]\ar@{-->}[d]_(.45){\widetilde{h}}\ar[r]&C^{\e\bullet}\lbe({\rm U}_{\be S}(\e j\e))[1]\ar@{-->}[d]_(.45){\widetilde{f}\e[1]}\\
\upicgs\ar[rr]^(.45){\upic_{S}\lbe(\e p)}&&\upic_{H/S}\ar[rr]^{u}&& Z\ar[r]&\upicgs[1].	
}
\end{equation}
We claim that the composition $u\circ g^{-1}\be\circ v\e\colon\lbe C^{\e\bullet}\lbe({\rm U}_{\be S}(\e j\e))\to Z$ is the zero morphism. Indeed, since $v\circ f=g\circ \upic_{S}\lbe(\e p)$ and $H^{-1}\lbe(\e f\le )$ is an isomorphism, we have
\[
H^{-1}(\e u)\circ H^{-1}(\e g^{-1})\circ H^{-1}(v)=H^{-1}(\e u)\circ H^{-1}(\e \upic_{S}\lbe(\e p))\circ H^{-1}(\e f\le)^{-1}=0
\]
since $u\circ\upic_{S}\lbe(\e p)=0$ \cite[Corollary 1.2.2, p.~97]{v}. We conclude, as above, that there exist morphisms $\widetilde{f}$ and $\widetilde{h}$ in diagram \eqref{t8} such that $(\widetilde{f},g^{-1},\widetilde{h}\e)$ is a morphism of distinguished triangles. Now the concatenation of diagrams \eqref{t8} and \eqref{t7} (in that order) is a diagram of the form
\begin{equation}\label{t9}
\xymatrix{
C^{\e\bullet}\lbe({\rm U}_{\be S}(\e j\e))\ar@{-->}[d]_{\alpha}\ar[rr]^{v}&& {\rm U}_{H/S}[1]\ar@{=}[d]\ar[rr]^{{\rm U}_{\be S}(\e j\e)[1]}&& {\rm U}_{T/S}[1]\ar@{-->}[d]_{\beta}\ar[r]&C^{\e\bullet}\lbe({\rm U}_{\be S}(\e j\e))[1]\ar@{-->}[d]^{\alpha[1]}\\
C^{\e\bullet}\lbe({\rm U}_{\be S}(\e j\e))\ar[rr]^{v}&& {\rm U}_{H/S}[1]\ar[rr]^{{\rm U}_{\be S}(\e j\e)[1]}&& {\rm U}_{T/S}[1]\ar[r]& C^{\e\bullet}\lbe({\rm U}_{\be S}(\e j\e))[1]
}
\end{equation}
(namely for $\alpha=f\be\circ\be \widetilde{f}$ and $\beta=h\be\circ\be \widetilde{h}$\e). Since $H^{0}(\e C^{\e\bullet}\lbe({\rm U}_{\be S}(\e j\e)))=\cok {\rm U}_{\be S}(\e j\e)\simeq\mu^{\e *}$ by the proof of Proposition \ref{ctl} (see diagram \eqref{dia}) and ${\rm U}_{T/S}\simeq T^{\e*}$ by Lemma \ref{gros}, we have
\[
\Hom_{\dbs}\lle(\e C^{\e\bullet}\lbe({\rm U}_{\be S}(\e j\e)), {\rm U}_{T/S})=\Hom(\e\mu^{\e *},T^{\e*})=0.
\]
Thus \cite[Proposition 1.1.9, p.~23]{fp} shows that there exist {\it unique} morphisms $\alpha$ and $\beta$ in diagram \eqref{t9} such that $(\alpha, 1_{\e {\rm U}_{\lbe H\lbe/\lbe S}[1]},\beta\e)$ is a morphism of distinguished triangles, i.e.,  $\alpha=1_{\e C^{\le\bullet}\lbe({\rm U}_{\be S}(\e j\e))}$ and $\beta=1_{\e {\rm U}_{\lbe T\lbe/\lbe S}[1]}$. We conclude that $f\be\circ\be \widetilde{f}=1_{\e C^{\le\bullet}\lbe({\rm U}_{\be S}(\e j\e))}$, whence $H^{\le 0}\lbe(\e f\le )\colon H^{\le 0}\lbe(\e\upicgs)\to H^{\le 0}(\e C^{\e\bullet}\lbe({\rm U}_{\be S}(\e j\e)))$ is a surjective morphism of \'etale sheaves on $S$. Since $H^{\le 0}\lbe(\e\upicgs)=\picgs$ and $H^{\le 0}(\e C^{\e\bullet}\lbe({\rm U}_{\be S}(\e j\e)))=\cok {\rm U}_{\be S}(\e j\e)$ are both isomorphic to the \'etale sheaf $\mu^{\le *}$, which has finite stalks, a counting argument now shows that $H^{\le 0}\lbe(\e f\le )$ is, in fact, an isomorphism of \'etale sheaves on $S$. We conclude that $f\colon\upicgs\to C^{\e\bullet}\lbe({\rm U}_{\be S}(\e j\e))$ is an isomorphism in $\dbs$. The composition of the preceding isomorphism and the canonical isomorphism $C^{\e\bullet}\lbe({\rm U}_{\be S}(\e j\e))\isoto \pi_{1}^{\le D}\be(\s R\e)$ of Proposition \ref{ctl}(ii) is the isomorphism \eqref{isor}.

It remains to show that the isomorphism $\upicgs\isoto\pi_{1}(G\e)^{D}$ thus obtained is functorial in $G$. Let $\varphi\colon G^{\e\prime}\to G$ be a morphism of reductive $S$-group schemes. By \cite[Lemma 3.3]{bga}, there exists a $t$-resolution of $\varphi$, i.e., an exact and commutative diagram of reductive $S$-group schemes
\begin{equation}\label{bga}
\xymatrix{
1\ar[r] & T^{\e\prime}\ar[r]^{j^{\e\prime}}\ar[d]^{\gamma}&  H^{\e\prime}\ar[r]^{p^{\e\prime}}\ar[d]^{\delta}&  G^{\e\prime}\ar[r]\ar[d]^{\varphi} &1\qquad (\s R^{\e\prime})\\
1\ar[r] & T\ar[r]^{j}&  H\ar[r]^{p}&  G\ar[r]&1\qquad (\s R\e),
}
\end{equation}
where the top and bottom rows are $t$-resolutions of $G^{\e\prime}$ and $G$, respectively. The left-hand square in the above diagram induces a morphism of complexes $\theta\colon C^{\e\bullet}\lbe({\rm U}_{\be S}(\e j\e))\to C^{\e\bullet}\lbe({\rm U}_{\be S}(\e j^{\e\prime}\e))$ whose components are ${\rm U}_{\be S}(\e \delta\e)$ and ${\rm U}_{\be S}(\e \gamma\e)$ in degrees $-1$ and $0$, respectively. We also note that, since the map $f_{(G\le)}=f_{(G,\e\s R\le)}\colon\upicgs\to C^{\e\bullet}\lbe({\rm U}_{\be S}(\e j\e))$ in diagram \eqref{t7} is an isomorphism, the map $h\colon Z\to {\rm U}_{T/S}[1]$ in that diagram is an isomorphism as well \cite[Corollary 1.2.3, p.~97]{v}. Now let $f_{(G^{\le\prime}\le)}=f_{(G^{\le\prime},\e\s R^{\le\prime}\le)}\colon\upic_{\e G^{\le\prime}\!/S}\isoto C^{\e\bullet}\lbe({\rm U}_{\be S}(\e j^{\le\prime}\e))$ and consider the following diagram whose rows are distinguished triangles in $\dbs$:
\begin{equation}\label{curvy}
\xymatrix{
\ar@/_4.2pc/[ddd]_(.5){\theta}C^{\e\bullet}\lbe({\rm U}_{\be S}(\e j\e))\ar[d]_{f_{(G\le)}^{-1}}^{\simeq}\ar[rr]^{v}&& {\rm U}_{H/S}[1]\ar[d]_{g_{(H)}^{-1}}^{\simeq}\ar[rr]^{{\rm U}_{\be S}(\e j\e)[1]}&& {\rm U}_{T/S}[1]\ar@/^2.3pc/@{-->}[ddd]_(.5){m}\ar[d]_{h^{-1}}^{\simeq}\ar[r]&C^{\e\bullet}\lbe({\rm U}_{\be S}(\e j\e))[1]\ar[d]_{f_{(G\le)}^{-1}[1]}^{\simeq}\ar@/^3.2pc/[ddd]^(.5){\theta\le[1]}\\
\upicgs\ar[d]_{\upic_{\lbe S}(\varphi)}\ar[rr]^(.45){\upic_{S}\lbe(\e p)}\ar @{} [drr]|{\!\!\!\!\!\!\rm (\e \bf{I}\e)}&&\upic_{H/S} \ar[d]_{\upic_{\lbe S}(\delta)}\ar[rr]^{u}&& Z\ar@{-->}[d]_{\!\!\! l}\ar[r]&\upicgs[1]\ar[d]_{\upic_{\lbe S}(\varphi)[1]}\\
\upic_{\e G^{\le\prime}\be/S}\ar[d]_{f_{(G^{\le\prime}\le)}}^{\simeq}\ar[rr]^(.45){\upic_{S}\lbe(\e p^{\e\prime})}&&\upic_{\e H^{\le\prime}\be/S} \ar[d]_{g_{(H^{\le\prime})}}^{\simeq}\ar[rr]^{u^{\e\prime}}&& Z^{\e\prime}\ar[d]_{h^{\e\prime}}^{\simeq}\ar[r]&\upic_{\e G^{\e\prime}\be/S}[1]\ar[d]_{f_{(G^{\le\prime}\le)}[1]}^{\simeq}\\
C^{\e\bullet}\lbe({\rm U}_{\be S}(\e j^{\e\prime}\e))\ar[rr]^{v^{\e\prime}}&& {\rm U}_{H^{\le\prime}\be/S}[1]\ar[rr]^{{\rm U}_{\be S}(\e j^{\e\prime}\e)[1]}&& {\rm U}_{T^{\le\lle\prime}\be/S}[1]\ar[r]&C^{\e\bullet}\lbe({\rm U}_{\be S}(\e j^{\e\prime}\e))[1].}
\end{equation}
Let
\[
\begin{array}{rcl}
\lambda_{1}&=&f_{(G^{\le\prime}\le)}\be\circ\be\upic_{\lbe S}(\varphi)\circ f_{(G\le)}^{-1}\\
\lambda_{2}&=&g_{(H^{\le\prime})}\be\circ\be\upic_{\lbe S}(\delta)\circ g_{(H)}^{-1}
\end{array}
\]
be the first and second vertical compositions in \eqref{curvy}, respectively. By the definitions of $f_{(G\le)}$ and $f_{(G^{\le\prime})}$, $(f_{(G)}^{-1},g_{(H)}^{-1},h^{-1})$ and  $(f_{(G^{\le\prime})},g_{(H^{\le\prime})},h^{\le\prime})$ are morphisms of distinguished triangles. Further, the square labeled $(\e \bf{I}\e)$ above commutes since \eqref{bga} commutes. On the other hand, it follows from the definitions of the maps $g_{(H)}, g_{(H^{\le\prime})},v, v^{\le\prime}$ \eqref{v} and $\theta$ that the following diagram commutes:
\[
\xymatrix{ 
C^{\e\bullet}\lbe({\rm U}_{\be S}(\e j\e))\ar[d]_(.45){\theta}\ar[r]^(.55){v}&{\rm U}_{H/S}[1]\ar[d]^{\lambda_{2}}\\
C^{\e\bullet}\lbe({\rm U}_{\be S}(\e j^{\e\prime}\e))\ar[r]^(.55){v^{\le\prime}}&{\rm U}_{H^{\le\prime}\be/S}[1].
}
\]
Thus, by \cite[Proposition 1.1.9, p.~23]{fp}, there exist morphisms $l\colon Z\to Z^{\e\prime}$ and $m\colon {\rm U}_{T\be/S}[1]\to {\rm U}_{T^{\le\prime}\be/S}[1]$ in \eqref{curvy} such that $(\lambda_{1},\lambda_{2},h^{\e\prime}\circ l\circ h^{-1}\e)$ and $(\theta, \lambda_{2},m\e)$ are morphisms of distinguished triangles with the same second component $\lambda_{2}$. Now, since
\[
\Hom_{\dbs}\lle(\e C^{\e\bullet}\lbe({\rm U}_{\be S}(\e j\e)),{\rm U}_{T^{\le\prime}\be/S})=\Hom\lle(\e\mu^{\e *}, (T^{\e\prime})^{*})=0,
\]
the uniqueness assertion in \cite[Proposition 1.1.9, p.~23]{fp} yields $\theta=\lambda_{1}$, i.e., the following diagram commutes:
\[
\xymatrix{ 
\upicgs\ar[d]_(.45){\upic_{S}\lbe(\e \varphi)}\ar[r]^(.45){f_{(G\le)}}&C^{\e\bullet}\lbe({\rm U}_{\be S}(\e j\e))\ar[d]^{\theta}\\
\upic_{\e G^{\le\prime}\be/S}\ar[r]^(.45){f_{(G^{\le\prime})}}&C^{\e\bullet}\lbe({\rm U}_{\be S}(\e j^{\e\prime}\e)).
}
\]
Therefore the following diagram commutes as well
\[
\xymatrix{ 
\upicgs\ar[d]_(.45){\upic_{S}\lbe(\e \varphi)}\ar[r]^(.45){f_{(G\le)}}&C^{\e\bullet}\lbe({\rm U}_{\be S}(\e j\e))\ar[d]^{\theta}\ar[r]^{\sim}&\pi_{1}^{\le D}\be(\s R\e)\ar[d]^{\pi_{1}^{\le D}\lbe(\varphi)}\\
\upic_{\e G^{\le\prime}\be/S}\ar[r]^(.45){f_{(G^{\le\prime})}}&C^{\e\bullet}\lbe({\rm U}_{\be S}(\e j^{\e\prime}\e))\ar[r]^{\sim}&\pi_{1}^{\le D}\be(\s R^{\e\prime}\e),
}
\]
where the unlabeled maps are the canonical isomorphisms of Proposition \eqref{ctl}(ii). The proof is now complete.
\end{proof}

The sequences in (ii) and (iii) below generalize (respectively) \cite[(6.11.4) and (6.11.2)]{san}.

\begin{corollary}\label{cor1} Let $S$ be a locally noetherian regular scheme and let
\[
1\to G_{1}\to G_{2}\to G_{3}\to 1
\]	
be an exact sequence of reductive $S$-group schemes. Then the above sequence induces
\begin{enumerate}
\item[(i)] 	a distinguished triangle in $\dbs$
\[
\upic_{\e G_{3}/\lbe S}\to \upic_{\e G_{2}/S}\to \upic_{\e G_{1}\lbe/S}\to\upic_{\e G_{3}/\lbe S}[1],
\]
\item[(ii)] an exact sequence of \'etale sheaves on $S$
\[
1\to G_{3}^{\le*}\to G_{2}^{\le*}\to G_{1}^{\le*}\to\pic_{\be G_{3}/\lbe S}\to \pic_{\be G_{2}/S}\to \pic_{\be G_{1}\lbe/S}\to 1
\]
and
\item[(iii)] an exact sequence of abelian groups
\[
\begin{array}{rcl}
1\to G_{3}^{\e*}\lbe(S\e)\to G_{2}^{\e*}\lbe(S\e)\to G_{1}^{\e*}\lbe(S\e)&\to&\pic G_{3}\to \pic G_{2}\to \npic\be(G_{1}\lbe/\lbe S\e)\\
&\to&\br_{\be\rm a}^{\prime}(G_{3}/S\e)\to\br_{\be\rm a}^{\prime}(G_{2}/S\e)\to\br_{\be\rm a}^{\prime}(G_{1}/S\e).
\end{array}
\]
\end{enumerate}	
\end{corollary}
\begin{proof} Assertion (i) is immediate from the theorem and Proposition \ref{dex}\e. By \eqref{up-1} and \eqref{up-0}, the distinguished triangle in (i) induces an exact sequence (of \'etale sheaves on $S$\e) $1\to {\rm U}_{G_{3}/S}\to {\rm U}_{G_{2}/S}\to {\rm U}_{G_{1}/S}\to\pic_{\be G_{3}/\lbe S}\to \pic_{\be G_{2}/S}\to \pic_{\be G_{1}\lbe/S}\to 1$. Assertion (ii) now follows from Lemma \ref{gros}. To prove (iii), let $i\colon G_{1}\to G_{2}$ and $p\colon G_{2}\to G_{3}$ be the given $S$-morphisms and, for $j=1,2$ or $3$, let $\pi_{j}\colon G_{j}\to S$ and $\varepsilon_{j}\colon S\to G_{j}$ denote the structural morphism and unit section of $G_{j}$, respectively. Now recall from Corollary \ref{kcor}(iii) the exact sequence of abelian groups
\begin{equation}
\begin{array}{rcl}\label{ele}
0\to G_{3}^{\e *}\be(\lbe S\le)\to G_{2}^{\e *}\be(\lbe S\le)\to G_{1}^{\e *}\be(\lbe S\le)&\to&\pic G_{3}\to\pic G_{2} \to\npic\be(G_{1}\lbe/\lbe S\e)\\
&\overset{\!\alpha}{\to}&\br_{\be 1}^{\prime}(G_{3}\lbe/\lbe S\le)\overset{\!\br_{\be 1}^{\prime}\le p}{\lra}\br_{\be 1}^{\prime}(G_{2}\lbe/\lbe S\le )\overset{\!\varphi^{\le\prime}}{\to} \br_{\be \varepsilon_{1}}^{\prime}\be(G_{1}\lbe/\lbe S\le),
\end{array}
\end{equation}
where $\varphi^{\le\prime}$ is the map \eqref{vphi3p}, and consider the following diagram
\[
\xymatrix{
&\br_{\be \varepsilon_{3}}^{\prime}\be(G_{3}\lbe/\lbe S\le)\ar[rr]^{\br_{\!\varepsilon_{2}\lbe,\e\varepsilon_{3}}^{\prime}\le p}\ar@{^{(}->}[d]&& 
\br_{\be \varepsilon_{2}}^{\prime}\be(G_{2}\lbe/\lbe S\le)\ar[rr]^{\br_{\!\varepsilon_{1}\lbe,\e\varepsilon_{2}}^{\prime}\le i}\ar@{^{(}->}[d]&&\br_{\be \varepsilon_{1}}^{\prime}\be(G_{1}\lbe/\lbe S\le)\ar@{=}[d]\\
\npic\be(G_{1}\lbe/\lbe S\e)\ar@{-->}[ur]^{\alpha^{\prime}}\ar[r]^(.55){\alpha}&\br_{\be 1}^{\prime}\be(G_{3}\lbe/\lbe S\le)\ar[rr]^{\br_{\be 1}^{\prime}\le p}\ar@{->>}[d]&&
\br_{\be 1}^{\prime}\be(G_{2}\lbe/\lbe S\le)\ar[rr]^{\varphi^{\le\prime}}\ar@{->>}[d]&&\br_{\be \varepsilon_{1}}^{\prime}\be(G_{1}\lbe/\lbe S\le)\\ 
&\img \br_{\! 1}^{\prime}\le \pi_{3}\ar[rr]^{\sim}&&\img \br_{\! 1}^{\prime}\le \pi_{2}&&.
}
\]
The bottom horizontal arrow is the isomorphism \eqref{im} induced by $\br_{\be 1}^{\prime}\le p$ and all squares above commute by Remark \ref{gps2} and the  definition of $\br_{\!\varepsilon_{2}\lbe,\e\varepsilon_{3}}^{\prime}\le p$ \eqref{csig1}. Further, the middle row is exact by the exactness of \eqref{ele} and the left-hand and middle columns are (split) exact sequences of abelian groups by \eqref{spe}. Since the bottom map is injective (in fact, an isomorphism), the commutativity of all three squares in the diagram shows that the top row is exact as well. On the other hand, by Remark \ref{tter}(a) applied to the triples $(G_{2}, G_{3},G_{1}\e)$ and $(S,S,S\e)$, the composition of $\alpha$ and the map $\br_{\be 1}^{\lbe\prime}\varepsilon_{3}\colon \br_{\be 1}^{\prime}(G_{3}\lbe/\lbe S\le)\to\brp S$ factors through $\npic_{\be S}(\varepsilon_{1})\colon\npic\be(\lbe G_{1}\lbe/\be S\e)\to \npic\be(\lbe S\lbe/\be S\e)=0$. Thus $\alpha$ factors through the map labeled $\alpha^{\e\prime}$ above and \eqref{eker} (applied to the $S$-morphism $G_{2}\to G_{3}$) shows that $\img\e \alpha^{\e\prime}=\krn\e\br_{\!\varepsilon_{2}\lbe,\e\varepsilon_{3}}^{\prime}\le p$. Thus \eqref{ele} induces an exact sequence
\[
\begin{array}{rcl}
0\to G_{3}^{\e *}\be(\lbe S\le)\to G_{2}^{\e *}\be(\lbe S\le)&\to& G_{1}^{\e*}\be(\lbe S\le)\to\pic G_{3}\to\pic G_{2} \to\npic\be(G_{1}\lbe/\lbe S\e)\\
&\overset{\!\alpha^{\le\prime}}{\to}&\br_{\be \varepsilon_{3}}^{\prime}\be(G_{3}\lbe/\lbe S\le)\overset{\!\be\br_{\!\varepsilon_{2}\lbe,\e\varepsilon_{\lbe 3}}^{\prime}\le p}{\lra}\br_{\be \varepsilon_{2}}^{\prime}\be(G_{2}\lbe/\lbe S\le)\overset{\!\br_{\!\varepsilon_{1}\lbe,\e\varepsilon_{2}}^{\prime}\le i}{\lra} \br_{\be \varepsilon_{1}}^{\prime}\be(G_{1}\lbe/\lbe S\le).
\end{array}
\]
Finally, by the commutativity of diagram \eqref{fty}, the above sequence induces an exact sequence of abelian groups 
\[
\begin{array}{rcl}
0\to G_{3}^{\e *}\be(\lbe S\le)\to G_{2}^{\e *}\be(\lbe S\le)\to G_{1}^{\e *}\be(\lbe S\le)&\to&\pic G_{3}\to\pic G_{2} \to\npic\be(G_{1}\lbe/\lbe S\e)\\
&\overset{\!\alpha^{\le\prime\prime}}{\to}&\br_{\be\rm a}^{\prime}(G_{3}/S\e)\overset{\!\br_{\! \rm a}^{\prime}\e p}{\lra}\br_{\be\rm a}^{\prime}(G_{2}/S\e)\overset{\!\br_{\! \rm a}^{\prime}\e i}{\lra} \br_{\be\rm a}^{\prime}(G_{1}/S\e),
\end{array}
\]
where $\alpha^{\e\prime\prime}=c_{\e(G_{3}),\e \varepsilon_{3}}\circ  \alpha^{\e\prime}$ and $c_{\e(G_{3}),\e \varepsilon_{3}}$ is given by \eqref{csig0}.
\end{proof}

\begin{remark} \label{lastb} In the setting of the corollary assume, in addition, that $S$ is noetherian and irreducible with function field $K$. As in Remark \ref{tter}(c), the canonical maps $G_{i}^{\e *}(S)\to G_{i}^{\e *}(K)$ (where $i=1,2$ and 3) and $\npic\be(G_{1}\lbe/\lbe S\e)\to \pic G_{1,\e K}$ are isomorphisms of abelian groups. Further, by Lemma \ref{picr}, $\pic G_{i}$ is canonically isomorphic to $\pic S\oplus\pic G_{i,\e K}$ for $i=2$ and $3$. Thus the sequence in part (iii) of the corollary is canonically isomorphic to an exact sequence
\[
\begin{array}{rcl}
1\to G_{3}^{\le*}(K\e)\to G_{2}^{\le*}(K\e)&\to& G_{1}^{\le*}(K\e)\to\pic S\oplus\pic G_{3,\e K}\to \pic S\oplus\pic G_{2,\e K}\\
&\to& \pic G_{1,\e K}\to\br_{\be\rm a}^{\prime}(G_{3}/S\e)\overset{\!\br_{\!\lbe \rm a}^{\prime}\e p}{\lra}\br_{\be\rm a}^{\prime}(G_{2}/S\e)\to\br_{\be\rm a}^{\prime}(G_{1}/S\e).
\end{array}
\]
Now, by a functoriality argument similar to that given in Remark \ref{tter}(c), $\krn\e \br_{\!\lbe \rm a}^{\prime}\, p$ is canonically isomorphic to $\krn\e\br_{\!\lbe \rm a}^{\prime}\, p_{K}$. Thus most of the sequence in part (iii) of the corollary is essentially equivalent to the corresponding subsequence over $K$. A similar fact becomes evident when the sequence in Corollary \ref{kcor}(iii) is compared with \cite[(6.10.3), p.~43]{san}.
\end{remark}

\end{document}